\documentclass[final]{siamltex}
\usepackage{amsmath}              
\usepackage{graphicx}
\usepackage{amssymb, hyperref}
\usepackage{tikz} 
\usetikzlibrary{matrix,arrows}

\newcommand{\vectornorm}[1]{\parallel\hspace{-.1cm}#1\hspace{-.1cm}\parallel}

\newenvironment{myindentpar}[1]%
{\begin{list}{}%
         {\setlength{\leftmargin}{#1}}%
         \item[]%
}
{\end{list}}

\newtheorem{thm}{Theorem}[section]
\newtheorem{cor}{Corollary}[section]
\newtheorem{prop}{Proposition}[section]

\newtheorem{remark}{Remark}[section]
\newtheorem{example}{Example}

\title{On the persistence and global stability of mass-action systems}

\author{Casian Pantea\thanks{
Department of Electrical and Electronic Engineering, Imperial College London
({\tt c.pantea@imperial.ac.uk}). The author's research was supported through grant NIH R01GM086881.}}

\begin{document}

\maketitle

\begin{abstract}
This paper concerns the long-term behavior of population systems, and
in particular of chemical reaction systems, modeled by deterministic 
mass-action kinetics. We approach
two important open problems in the field of Chemical Reaction
Network Theory, the Persistence Conjecture and the Global Attractor
Conjecture. We study the persistence of a large class of
networks called {\em lower-endotactic} and in particular, we 
show that in weakly reversible mass-action systems 
with  two-dimensional stoichiometric subspace all bounded trajectories
are persistent. Moreover, we use these ideas to show that the Global
Attractor Conjecture is true for systems with
three-dimensional stoichiometric subspace.
\end{abstract}



\begin{section}{Introduction}\label{sec:intro}
Mass-action systems are a large class of nonlinear differential
equations, widely used in the modeling of interaction networks in 
chemistry, biology and engineering. Due to the high 
complexity of dynamical systems arising from nonlinear interactions,
it is very difficult, if not impossible, to create general
mathematical criteria about qualitative properties of such systems, like 
existence of positive equilibria, stability properties of equilibria
or persistence (non-extinction) of variables. 
However, a fertile theory that answers this type of questions for
mass-action systems has been developed over the last 40 years in the
context of chemical reaction systems. 
Generally termed “Chemical Reaction Network Theory” 
\cite{feinberg:1972, feinberg:lectures, feinberg:chem_eng_sci,
  feinberg_horn, gunawardena, horn, horn:1974, horn_jackson}  
this field of research originated with the seminal work of Fritz Horn, Roy Jackson and Martin
Feinberg \cite{feinberg:1972, horn, horn_jackson}  
and describes the surprisingly stable dynamic behavior of
large classes of mass-action systems, {\em independently of the values
  of the parameters present in the system}. This fact is very relevant, 
since the exact values of the system parameters
are typically unknown in practical applications.  
Although the results in this paper will be applicable to general population
systems driven by mass-action kinetics, they will be  
developed within the frame of Chemical Reaction
Network Theory.

A large part of this paper is devoted to {\em persistence} properties of mass-action 
systems. A dynamical system on $\mathbb R^n_{\ge 0}$ is called persistent
if forward trajectories that start in the interior of the positive orthant do not
approach the boundary of $\mathbb R^n_{\ge 0}$ (see section
\ref{sec:persdef} for a rigorous definition). Note that, throughout this paper,
{\em trajectory} will always mean {\em bounded trajectory}. For systems with
bounded trajectories, this is equivalent to saying that no trajectories
with positive initial condition have $\omega$-limit points 
on the boundary of $\mathbb R^n_{\ge 0}.$ 
Persistence answers important questions regarding dynamic properties of biochemical
systems, ecosystems, or infectious diseases, e.g. will each chemical
species be available at all future times; or will a species become
extinct in an ecosystem; or will an infection die off.
One of the major open questions of Chemical Reaction Network Theory is
the following:
\smallskip 

\noindent {\bf Persistence Conjecture }\cite{craciun_nazarov_pantea}.
Any weakly reversible mass-action system is persistent. 

\smallskip

\noindent A weakly reversible mass-action system is one for which its directed reaction 
graph has strongly connected components (Definition \ref{def:WR}). 
A version of this conjecture was first mentioned by Feinberg
in \cite[Remark 6.1.E]{feinberg:chem_eng_sci}; that version 
only requires that no trajectory with positive initial condition 
converges to a boundary point. A stronger version of the Persistence
Conjecture (called {\em the Extended Persistence Conjecture}) was
formulated by Craciun, Nazarov and Pantea in
\cite{craciun_nazarov_pantea} and it was shown to be true for
two-species systems. Moreover, in that case, weakly reversible
mass-action systems are not only persistent, but also {\em permanent}
(all trajectories originating in the interior of $\mathbb R^n_{\ge 0}$
eventually enter a fixed compact subset of the interior of $\mathbb R^n_{\ge 0}$). 

In recent approaches to the Persistence Conjecture,
the behavior of weakly reversible mass-action systems near the faces of
their {\em stoichiometric compatibility classes} (minimal linear invariant
subsets) was considered. It is known that $\omega$-limit points may only lie on
faces of the stoichiometric compatibility class that are associated
with a {\em semilocking set} \cite{anderson:primul} 
(see also \cite[Remark 6.1.E]{feinberg:chem_eng_sci}), or  {\em
  siphon} in the Petri net literature \cite{angeli_leenheer_sontag, shiu_sturmfels}.
Anderson \cite{anderson:primul} and 
Craciun, Dickenstein, Shiu and Sturmfels
\cite{craciun_dickenstein_shiu_sturmfels} 
showed that vertices of the stoichiometric compatibility class cannot
be $\omega$-limit points. Moreover, 
Anderson and Shiu \cite{anderson_shiu} proved that for a weakly reversible mass-action system, the trajectories are, in some sense,  repelled away from 
codimension-one faces of the stoichiometric compatibility class. 

In this paper we prove the following version of the Persistence
Conjecture for systems with two-dimensional stoichiometric
compatibility classes (Theorem \ref{thm:pers}):

\smallskip
\noindent {\bf Theorem 5.1.} {\em Any $\kappa$-variable mass-action
 system with bounded trajectories, two-dimensional stoichiometric compatibility classes 
and lower-endotactic stoichiometric subnetworks is persistent.}
\smallskip

\noindent Here a {\em stoichiometric subnetwork} is a union of connected components of the
reaction graph (see Definition \ref{def:canonSub}) and 
the requirement of {\em lower-endotactic} (Definition \ref{def:endo})
stoichiometric subnetworks is less restrictive than that of weak
reversibility. We suggest that the hypothesis of ``lower-endotactic''
arises naturally in the context of persistence of mass-action systems.
Moreover, {\em $\kappa$-variable mass-action}
is a generalization of mass-action where each reaction rate parameter is allowed to vary 
within a compact subset of $(0,\infty)$ (see Definition
\ref{def:massAct}). Therefore this theorem implies the following:

\smallskip
\noindent {\bf Corollary.} {\em Any weakly reversible mass-action system
with two-dimensional stoichiometric compatibility classes and bounded
trajectories is persistent.}  
\smallskip

\noindent Note that our proof of Theorem 5.1. above (and of its corollary) requires
the hypothesis of bounded trajectories. However, it has been conjectured
that all trajectories of weakly reversible mass-action systems are
bounded \cite{anderson:oneLCbd}, and this conjecture has recently been
proved for networks whose reaction graph has a single connected
component \cite{anderson:oneLCbd}.  Also, a stronger statement
is known to be true for two-species networks: any {\em endotactic,
  $\kappa$-variable} mass-action system with two species has bounded trajectories 
\cite{craciun_nazarov_pantea}. 

The Persistence Conjecture is strongly related to another conjecture
which is often considered the most important open problem in the
field of Chemical Reaction Network Theory \cite{anderson:oneLC, 
anderson_shiu, craciun_dickenstein_shiu_sturmfels,
craciun_nazarov_pantea}, namely the Global Attractor
Conjecture. This conjecture was first formulated by Horn
\cite{horn:1974} and concerns the long-term behavior of 
  complex-balanced systems, i.e., systems that admit a positive
{\em complex-balanced equilibrium} (see Definition \ref{def:compBal}). 
Horn and Jackson showed that if a mass-action
system is complex-balanced then there exists a unique positive equilibrium in
each stoichiometric compatibility class and this equilibrium is
complex-balanced \cite{horn_jackson}. Moreover, each such equilibrium
is locally asymptotically stable in its stoichiometric compatibility
class due to the existence of a strict Lyapunov function \cite{horn_jackson}. In two
subsequent papers \cite{feinberg:1972, horn} Feinberg and Horn
showed that weakly reversible mass-action systems which are also 
{\em deficiency zero} (Definition \ref{def:deficiency}) are
complex-balanced. This fact is remarkable since it reveals a wide
class of reaction systems that are complex-balanced only because of
their structure and regardless of parameter values. For a
self-contained treatment of Chemical Reaction Network Theory,
including the results mentioned above, the reader is referred to \cite{feinberg:lectures}. 

The Lyapunov function of Horn and Jackson does not guarantee {\em global}
stability for a positive equilibrium relative to the interior of its
compatibility class. This fact is the object of the Global Attractor Conjecture:

\smallskip
\noindent {\bf Global Attractor Conjecture.} {\em In a complex-balanced
mass-action system, the unique positive equilibrium of a
stoichiometric compatibility class is a global attractor of
the interior of that class.}
\smallskip

\noindent It is known \cite{feinberg:lectures} that complex-balanced systems are
necessarily weakly reversible. On the other hand, trajectories of
complex-balanced systems converge to the set of equilibria 
\cite{anderson:primul, siegel_maclean, sontag:tcell}, so it follows that the
Persistence Conjecture implies the Global Attractor Conjecture.   

A series of partial results towards a proof of the Global Attractor
Conjecture have been obtained in recent years. It is known that the
conjecture is true for systems with two-dimensional stoichiometric
compatibility classes (\cite{anderson_shiu}; see also the recent work of
Siegel and Johnston \cite{siegel_johnston}), and for three-species
systems \cite{craciun_nazarov_pantea}. Recently, 
Anderson proved that the conjecture holds if the reaction graph has a
single connected component \cite{anderson:oneLC}.

In this paper we prove the Global Attractor Conjecture for systems
with three-dimensional stoichiometric compatibility classes (Theorem \ref{thm:gac}):

\smallskip
\noindent {\bf Theorem 6.3.} {\em Consider a complex-balanced weakly reversible
mass-action system having stoichiometric compatibility classes of
dimension three. Then, for any positive initial condition $c_0,$ the
solution $c(t)$ converges to the unique positive equilibrium which is
stoichiometrically compatible with $c_0.$} 
\smallskip
 
Aside from being significant in the field of polynomial dynamical systems
and relevant in important biological models \cite{gopalkrishnan,
gnacadja, siegel_maclean, sontag:tcell}, Chemical Reaction Network Theory, 
and in particular the two conjectures discussed above, have ramifications in
other well-established areas of mathematics. For example, 
\cite{craciun_dickenstein_shiu_sturmfels} stresses 
the connection with toric geometry and computational
algebra; in that work complex-balanced systems are called 
{\em toric dynamical systems} to emphasize their intrinsic algebraic
structure. Also, \cite{perez} studies the rich algebraic structure
of biochemical reaction systems with {\em toric steady states}.
Furthermore, the unique
positive equilibrium in a stoichiometric compatibility class of a 
complex-balanced system is sometimes called  the {\em Birch point} in
relation to Birch's Theorem from algebraic statistics
\cite{craciun_dickenstein_shiu_sturmfels, pachter_sturmfels}.  
 
This paper is organized as follows. 
After a preliminary section of terminology and notation, we
introduce the lower-endotactic networks in section \ref{sec:endo} and
follow with a discussion of our main technical tool, the 
{\em 2D-reduced mass-action system} in section \ref{sec:2Dred}. The
main persistence result of the paper is contained in section
\ref{sec:pers} (Theorem \ref{thm:pers}) and our result on the Global
Attractor Conjecture (Theorem \ref{thm:gac}) is proved in section \ref{sec:gac}. 
A critical part in the proof of the latter theorem resides in the
result of Theorem \ref{thm:repel}, which analyzes the behavior of 
weakly reversible mass-action systems
near codimension-two faces of stoichiometric compatibility classes.

\end{section}

\begin{section}{Preliminaries}\label{sec:prelim}
A chemical reaction network is usually given by a finite list of
reactions that involve a finite set of chemical {\em species.} An example
with four species $A, B, C$ and $D$ and five reactions is
given in (\ref{eq:exNet}). 
\begin{equation}\label{eq:exNet}
B+D\displaystyle \mathop{\rightleftharpoons}^{} A+C
\displaystyle \mathop{\to}^{}_{} A+B
\displaystyle \mathop{\to}^{}_{} C+D \qquad
2A\displaystyle \mathop{\to}^{} A+D\displaystyle \mathop{\gets}^{}2D
\end{equation}
The interpretation of the reaction $A+B\to C+D$,
for instance, is that one molecule of species $A$ combines with one
molecule of species $B$ to produce one molecule of each of the species
$C$ and $D$. The objects on both sides of a reaction are formal linear combinations
of species and are called {\em complexes}. According to the
direction of the reaction arrow, a complex is either {\em source} or
{\em target}. This way, the reaction $A+B\to C+D$ in (\ref{eq:exNet}) has 
$A+B$ as source complex and $C+D$ as target complex.
The concentrations $c_A, c_B, c_C$ and $c_D$ vary with time
by means of a set of ordinary differential equations,
which we will explain shortly. In this preliminary section of the
paper we review the standard concepts of
Chemical Reaction Network Theory (see \cite{feinberg:lectures}) and
introduce some new terminology that will be useful further
on. In what follows, the set of nonnegative, respectively
strictly positive real numbers are denoted by $\mathbb R_{\geq 0}$ and 
$\mathbb R_{>0}$ . For an integer $n\ge 1$ we call $\mathbb R_{>0}^n$ the {\em positive orthant}.
The boundary of a set $K\subset \mathbb R^n$ will be denoted by $\partial K$ and the
convex hull of $K$ will be denoted by $\mathrm{conv}(K).$ Also, we
will denote the transpose of a matrix $A$ by $A^t.$

\begin{subsection}{Reaction networks} \label{sec:CRN}
If $I$ is a finite set then we denote by ${\mathbb Z}_{\geq 0}^I$ and 
${\mathbb R}_{\geq 0}^I$  the set of all formal sums
$\alpha=\displaystyle\sum_{i \in I}\alpha_ii$ where $\alpha_i$ are
nonnegative integers, respectively nonnegative reals. 

\begin{definition}\label{def:CRN} 
A {\em chemical reaction network} is a triple $(\mathcal{S}, \mathcal{C},
\mathcal{R})$, where $\mathcal{S}$ is the set of
species, $\mathcal{C}\subseteq \mathbb Z_{\geq 0}^{\mathcal
  S}$ is the set of complexes, and $\mathcal{R}$ is a relation on
$\mathcal{C}$, denoted $P\to P'$, representing the set of
reactions of the network. The reaction set $\mathcal R$ cannot
contain elements of the form $P\to P$ and each complex in $\mathcal C$ is
required to appear in at least one reaction.
\end{definition}

For simplicity, we
will often denote a reaction network by a single letter, for instance
$\mathcal N = (\mathcal S,\mathcal C,\mathcal R).$
For technical reasons we have chosen to neglect a third
requirement that is usually included in the definition of a reaction network
(see \cite{feinberg:lectures}): 
each species appears in at least one complex. This condition
is not essential in the setting of this paper. 

In (\ref{eq:exNet}) the set of species is
$\mathcal{S}=\{A, B, C, D\}$, and the set of complexes is
$\mathcal{C}=\{B+D,A+C,A+B,C+D,2A,A+D,2D\}.$  

Once we fix an order among the species, any complex may be viewed as a
column vector of dimension equal to the number of elements of
$\mathcal S.$
For example, the complexes $A+B$ and $2D$ in (\ref{eq:exNet}) may be represented by
the vectors $(1\ 1\ 0\ 0)^t,$ and $(0\ 0\ 0\ 2)^t.$ With this
identification in place, we may now define the {\em reaction vector}
of a reaction  $P\to P'\in\mathcal R$ to be $P'-P.$ 

\begin{definition}\label{def:stoichSub}
The {\em stoichiometric subspace} of the reaction network $\mathcal N = (\mathcal S,
\mathcal C, \mathcal R)$ is  
$S= \mathrm{span}\{P'-P\mid P\to P'\in \mathcal R\}.$
\end{definition}

\begin{example}\label{ex:stoich}
The stoichiometric subspace of the reaction network (\ref{eq:exNet})
is the column space of the {\em stoichiometric matrix} 
$$
A=\left(
\begin{tabular}{rrrrrr}
  1 &-1 &  0&-1&-1 &  1\\
-1 &  1 &  1&-1&  0 &  0\\
  1 &-1 &-1&  1&  0 &  0\\
-1 &  1 &  0&  1&  1 &-1\\
\end{tabular}
\right).
$$ 
It is easy to see that $S=\{(a,\ b,\ -b,\ -a)^t | (a,b)\in \mathbb R^2\}.$
\end{example}

Any reaction network $\mathcal N$ can be viewed as a directed graph whose
vertices are the complexes of $\mathcal N$ and whose edges correspond to
reactions of $\mathcal N$. Each connected component of this graph is
called a {\em linkage class} of $\mathcal N.$ 
\begin{definition}\label{def:WR} 
A reaction network $\mathcal N$ is called {\em weakly reversible} if its
associated directed graph has strongly connected components.
\end{definition}

In other words, $\mathcal N$ is weakly reversible if whenever there exists
a directed arrow pathway (consisting of one or more reaction arrows)
from one complex to another, there also exists a directed arrow
pathway from the second complex back to the first. 
\end{subsection}

\begin{subsection}{Reaction systems}\label{sec:RSyst} 
Throughout this paper we let $n$ denote the number of species of a
reaction network $\mathcal N=(\mathcal S, \mathcal C, \mathcal R),$ we fix an order among the
species, and we denote $\mathcal S=\{X_1,\ldots, X_n\}.$ We also let   
$c(t)\in \mathbb R^{\mathcal S}\cong\mathbb R^{n}$ denote the (column) vector of
species concentrations at time $t\ge 0$. From here on,
``vector'' or ``point of $\mathbb R^n$'' will always mean ``column
vector'', even if, for simplicity, the notation $^t$ may not always be used.
The concentration vector $c(t)$ is governed by a set of ordinary differential equations
that involve a {\em reaction rate function} for each reaction in ${\mathcal R}.$

\begin{definition}\label{def:reactionSystem}
A {\em (non-autonomous) reaction system} is a quadruple $({\mathcal S}, {\mathcal
  C}, {\mathcal R}, K)$ where $\mathcal N=({\mathcal S}, {\mathcal C}, {\mathcal R})$ is a
reaction network with $n$ species and 
$K:\mathbb R_{\ge 0}\times \mathbb R_{\ge 0}^n\to \mathbb R_{>0}^{\mathcal
  R}$ is a piecewise differentiable function called the kinetics of the system.
The component $K_{P\to P'}$ of $K$ is called the rate function of
reaction $P\to P'.$ Letting $P = (m_1,\ldots, m_n)$ and ${\bf
  x}=(x_1,\ldots, x_n),$ $K_{P\to P'}$ is assumed to satisfy, for any
$t\ge 0,$ the following: if $x_i=0$ and $m_i\neq 0$ then $K_{P\to P'}(t, {\bf x})=0.$  
The dynamics of the system is given by the following system of
differential equations for
the concentration vector $c(t)$:
\begin{equation}\label{eq:reactionSystem}
\dot c(t) = \sum_{P\to P'}K_{P\to P'}(t,c(t))(P'-P).
\end{equation}
\end{definition} 

Note that we will often use the short notation $(\mathcal N, K)$ for a
reaction system. 

The regularity condition on $K$ may be replaced by any other condition
that guarantees uniqueness of solutions for (\ref{eq:reactionSystem}). 
Sometimes, additional properties are required of $K$ 
\cite{anderson_shiu, banaji_craciun:1, banaji_craciun:2}. For example, it
is commonly assumed that if the $i$th species is not a reactant in
$P\to P'$ (i.e. $m_i=0$) then $K_{P\to P'}$ does not depend on $x_i.$ 
Another widespread assumption is that $K$ is increasing with respect to 
reactant concentrations, i.e. $\frac{\partial}{\partial
  x_i}K_{P\to P'}\ge 0$ if $m_i\neq 0.$ These conditions are automatically satisfied 
for the kinetics treated in this paper.
 
If $c_0\in\mathbb R^n_{\ge 0}$ we let
$$T(c_0)=\{c(t)\mid t\ge 0, c(0)=c_0\}$$
denote the
{\em trajectory of $(\mathcal N, K)$ with initial condition $c_0.$} If
$K(t,{\bf x})=K({\bf x})$ does not depend explicitly on time, we
say that $c_*\in\mathbb R^n_{\ge 0}$ is an equilibrium of the
reaction system $(\mathcal N, K)$ if it is an equilibrium of the
corresponding differential equations (\ref{eq:reactionSystem}).

Note that the condition on $K_{P\to P'}$ imposed in Definition
\ref{def:reactionSystem} makes the nonnegative orthant $\mathbb
R^n_{\ge 0}$ forward invariant for (\ref{eq:reactionSystem}). 
Under mild additional assumptions on $K,$ the positive orthant
$\mathbb R^n_{>0}$ is also forward-invariant for 
(\ref{eq:reactionSystem}) (see \cite{sontag:tcell}). For example, this
will be the case for {\em $\kappa$-variable mass-action kinetics},
the main type of kinetics considered in this paper (Definition \ref{def:massAct}).

Integrating (\ref{eq:reactionSystem}) yields 
$$c(t) = c(0) + \sum_{P\to P'} \left(\int_{0}^t K_{P\to P'}(s,c(s))ds\right) (P'-P)$$
and it follows that $c(t)$ is contained in the affine subspace $c(0)+S$ for all $t\ge
0.$ Combining this with the preceding observation we see that 
$(c(0)+S)\cap \mathbb R_{\ge 0}^n$ is forward invariant for
(\ref{eq:reactionSystem}).
\begin{definition}\label{def:compClass}
Let $c_0\in \mathbb R^n.$ The polyhedron $(c_0+S)\cap \mathbb R_{\ge
  0}^n$ is called the {\em stoichiometric compatibility class of
  $c_0.$} 
\end{definition}

Note that, throughout this paper, ``polyhedron'' will always mean
``convex polyhedron'', i.e., an intersection of finitely many half-spaces.
To prepare for the next definition we introduce the following notation: given two
vectors $u,v\in\mathbb R^n_{\ge 0},$ we denote 
$u^{v}=\displaystyle\prod_{i=1}^nu_i^{v_i}$, with the convention $0^{0}=1$.

\begin{definition}[\cite{craciun_nazarov_pantea}]\label{def:massAct}
A {\em $\kappa$-variable mass-action system} is a reaction system $(\mathcal N,
K)$ where $\mathcal N=(\mathcal S, \mathcal C, \mathcal R)$ 
and the rate function of $P\to P'\in\mathcal R$ is given by 
\begin{equation}\label{eq:MAKin}
K_{P\to P'}(t, {\bf x})=\kappa_{P\to P'}(t){\bf x}^P.
\end{equation}
Here $\kappa:\mathbb R_{\geq 0}\to (\eta, 1/\eta)^{\mathcal R}$ for some
$\eta < 1$ is a piecewise differentiable function called the {\em rate-constant function.} 
\end{definition} 

To emphasize the rate-constant function, we denote a
$\kappa$-variable mass-action system by $(\mathcal S,\mathcal C, \mathcal R,
\kappa),$ or by $(\mathcal N, \kappa).$
Note that if the rate-constant function is fixed in
time, $\kappa$-variable mass-action becomes the usual
mass-action. A few biological examples of $\kappa$-variable mass-action
models that are not mass-action are presented in \cite{craciun_nazarov_pantea}.

Therefore, a $\kappa$-variable mass-action system gives rise to the
following non-autonomous, and usually nonlinear, 
system of coupled differential equations:
\begin{equation}\label{eq:varMassAct}
\dot c(t) = \sum_{P\to P'}\kappa_{P\to P'}(t) c(t)^P (P'-P).
\end{equation}

\begin{example}\label{ex:MAex}
 We endow the reaction network (\ref{eq:exNet}) with
  $\kappa$-variable mass-action kinetics of rate-constant
  function specified on the reaction arrows in (\ref{ex:MA}).
\begin{equation}\label{ex:MA}
B+D\displaystyle \mathop{\rightleftharpoons}^{\kappa_1(t)}_{\kappa_2(t)} A+C
\displaystyle \mathop{\to}^{\kappa_3(t)} A+B
\displaystyle \mathop{\to}^{\kappa_4(t)} C+D \qquad
2A\displaystyle \mathop{\to}^{\kappa_5(t)} A+D\displaystyle \mathop{\gets}^{\kappa_6(t)}2D
\end{equation}
We have $c(t)=(c_A(t),\ c_B(t),\ c_C(t),\ c_D(t))$ and note that, for
example, $c(t)^{A+B}=c(t)^{(1\ 1\ 0\ 0)}=c_A(t)c_B(t).$ From
(\ref{eq:varMassAct}) we have
$$
\dot c = \kappa_1(t)c_Bc_DA_1+\kappa_2(t)c_Ac_CA_2
+\kappa_3(t)c_Ac_CA_3+\kappa_4(t)c_Ac_BA_4+\kappa_5(t)c_A^2A_5+\kappa_6(t)c_D^2A_6
$$
for all $t\ge 0,$ where $A_i$ is column $i$ of the stoichiometric
matrix $A$ given in (\ref{ex:stoich}), i.e., the reaction vector of
reaction $i.$ Therefore the differential equations corresponding to
(\ref{ex:MA}) are
\begin{eqnarray}\label{eq:exMA}
\dot c_A &=&\kappa_1(t)c_Bc_D-\kappa_2(t)c_Ac_C-\kappa_4(t)c_Ac_B
-\kappa_5(t)c_A^2+\kappa_6(t)c_D^2\\\nonumber
\dot c_B &=&-\kappa_1(t)c_Bc_D+\kappa_2(t)c_Ac_C+\kappa_3(t)c_Ac_C-\kappa_4(t)c_Ac_B\\\nonumber
\dot c_C &=&\kappa_1(t)c_Bc_D-\kappa_2(t)c_Ac_C-\kappa_3(t)c_Ac_C+\kappa_4(t)c_Ac_B\\\nonumber
\dot c_D &=& -\kappa_1(t)c_Bc_D+\kappa_2(t)c_Ac_C+\kappa_4(t)c_Ac_B
+\kappa_5(t)c_A^2-\kappa_6(t)c_D^2,
\end{eqnarray}
\end{example}
\end{subsection}

\begin{subsection}{Sums of reaction systems}\label{sec:sums}
A reaction network $(\mathcal S, \mathcal C, \mathcal R)$ is called a {\em subnetwork} of
$(\mathcal S,' \mathcal C', \mathcal R')$ if $\mathcal S\subseteq \mathcal S'$, $\mathcal C\subseteq
\mathcal C'$ and $\mathcal R\subseteq \mathcal R'.$ If $\mathcal R'$ has an associated
kinetics $K$ then restricting $K$ to reactions of $\mathcal R$ defines a kinetics for
$\mathcal R.$ On the other hand, if $\mathcal N_s=(\mathcal S_s, \mathcal C_s, \mathcal
R_s)$, $s\in\{1,\ldots,p\}$ are reaction networks, their {\em union},
denoted by  $\bigcup_{s=1}^p (\mathcal S_s, \mathcal C_s, \mathcal R_s)$ or simply
by $\bigcup_{s=1}^p \mathcal N_s$, and defined as the triple
$(\bigcup_{s=1}^p \mathcal S_s, \bigcup_{s=1}^p\mathcal C_s,\bigcup_{s=1}^p \mathcal R_s)$
is also a reaction network. If each $\mathcal N_s$ has an associated
kinetics $K_s,$ we can define a kinetics for $\bigcup_{s=1}^p \mathcal N_s$ by
simply adding all $K_s.$  

\begin{definition}\label{def:sums}
The {\em sum of the reaction systems $(\mathcal S_s, \mathcal C_s, \mathcal R_s, K_s)$}
is the reaction system $(\mathcal S, \mathcal C, \mathcal R, K)$ where $(\mathcal S, \mathcal C, \mathcal
R)=\bigcup_{s=1}^p (\mathcal S_s, \mathcal C_s, \mathcal R_s)$ and 
$$K_{P\to P'}(t,{\bf x})=\sum_{\{s:P\to P'\in\mathcal R_s\}} K_{s,P\to
  P'}(t,\bf x)$$ 
for and all $(t,{\bf x})\in\mathbb R_{\ge 0}\times \mathbb
R^n_{\ge 0}.$ We will denote this $(\mathcal S, \mathcal C, \mathcal R, K)$ by  
$\bigcup_{s=1}^p(\mathcal S_s, \mathcal C_s, \mathcal R_s, K_s)$ or simply by
$\bigcup_{s=1}^p(\mathcal N_s, K_s),$ where $\mathcal N_s=(\mathcal S_s,\mathcal
C_s,\mathcal R_s).$
\end{definition}

For example,  any reaction system is the
sum of the reaction systems corresponding to its linkage classes.
Similarly, any reaction system is the sum of the reaction systems
corresponding to its {\em stoichiometric subnetworks}, which we define next.

\begin{definition}\label{def:canonSub}
A reaction network $(\mathcal S, \mathcal C, \mathcal R)$ with 
stoichiometric subspace $S$ can be written uniquely as a union of subnetworks
$$(\mathcal S, \mathcal C, \mathcal R) =
\bigcup_{s=1}^p(\mathcal S_s, \mathcal C_s, \mathcal R_s)$$ 
where $\{{\mathcal C}_s\}_{s\in\{1,\ldots,p\}}$ is a partition of $\cal C$
such that two complexes in $\cal C$ are in the same block of the
partition if and only if their difference is in $S$.
We call each $(\mathcal S_s, \mathcal C_s, \mathcal R_s)$ a {\em
  stoichiometric subnetwork of 
$(\mathcal S, \mathcal C, \mathcal R)$.} 
\end{definition} 

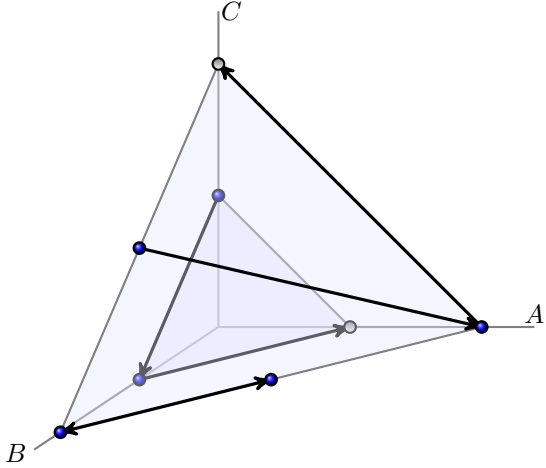
\begin{figure}
\begin{center}
\begin{tikzpicture}[thick, >=stealth', scale = .35] 

\draw[gray] (0,0) -- (12,0); 
\draw[gray] (0,0) -- (0,12); 
\draw[gray] (0,0) -- (-7,-4.66);

\draw[gray, fill=blue!11!white, fill opacity=0.5] (-3,-2) -- (5,0)--(0,5)--(-3,-2);

\draw[->, very thick] (-3,-2)--(5,0);
\draw[->, very thick] (0,5)--(-3,-2);

\shadedraw [shading=ball] (0,5) circle (.22cm);
\shadedraw [shading=ball] (-3,-2) circle (.22cm);
\shadedraw [] (5,0) circle (.22cm);

\draw[gray, fill=blue!9!white, fill opacity=0.4] (-6,-4) -- (10,0)--(0,10)--(-6,-4); 

\draw[<->, very thick] (-6,-4)--(2,-2);
\draw[->, very thick] (-3,3)--(10,0);
\draw[->, very thick] (10,0)--(0,10);

\shadedraw [shading=ball] (2,-2) circle (.22cm);
\shadedraw [shading=ball] (-3,3) circle (.22cm);
\shadedraw [shading=ball] (-6,-4) circle (.22cm);
\shadedraw [shading=ball] (10,0) circle (.22cm);
\shadedraw [] (0,10) circle (.22cm);

\node at  (12,.5) {$A$};
\node at  (-7.7,-4.8) {$B$};
\node at  (.5,12) {$C$};

\end{tikzpicture} 
\caption{Stoichiometric subnetworks for Example \ref{ex:stoichSubnet}.}\label{fig:stoichSubnet}
\end{center}
\end{figure}

\begin{example}\label{ex:stoichSubnet}
The diagram in Figure \ref{fig:stoichSubnet} represents the reaction
network 
$$C\to B\to A,\quad 2B\rightleftharpoons A+B,\quad B+C\to 2A\to 2C.$$
This reaction network has three linkage classes, and two
stoichiometric subnetworks
$$\{C\to B\to A\} \text{ and } \{2B\rightleftharpoons A+B,\quad B+C\to 2A\to 2C\}.$$ 
\end{example}

Note that there exist vectors $a_1,\ldots, a_p\in\mathbb R^n$ such that  for
all $s\in\{1,\ldots, p\},$ we have  $\mathcal C_s\subset a_s+S$ and
the affine subspaces $a_s+S,$ $s\in\{1,\ldots, p\}$ are pairwise disjoint. 
Also note that each stoichiometric subnetwork is a union of linkage classes.

\begin{example}\label{ex:canon}
The reaction network (\ref{eq:exNet}) has two 
linkage classes which belong to 
two different stoichiometric subnetworks since
$(A+B)-2A = (-1,\ 1,\ 0,\ 0)\notin S$ 
(recall $S$ from Example \ref{ex:stoich}).
Therefore the reaction network has exactly two stoichiometric
subnetworks which coincide with its linkage classes.
\end{example}
\end{subsection}

\begin{subsection}{Projected reaction systems}\label{sec:projRS}
For $W\subset \{1,\ldots, n\}$ we define
$$\pi_{W}:\mathbb R^n\to \mathbb R^W$$
to be the  {\em projection onto W}, i.e. the orthogonal projection
onto $\mathbb R^W.$ 
\begin{definition}[\cite{pantea:thesis}; see also \cite{anderson:oneLC}]\label{def:projNet}
Let $\mathcal N = (\mathcal S, \mathcal C, \mathcal R)$ be a reaction network with $\mathcal S =
\{X_1,\ldots, X_n\}$ and let $W\subset \{1,\ldots, n\}.$  
Let $\mathcal S_W=\{X_i\mid i\in W\},$ 
$$\mathcal R_W=\{\pi_W(P)\to \pi_W(P')\mid P\to P'\in \mathcal R, \text{ such that
}\pi_W(P)\neq \pi_W(P')\}$$
and $\mathcal C_W\subset \pi_W(\mathcal C)$ be the set of complexes in $\mathcal R_W.$
The reaction network $(\mathcal S_W, \mathcal C_W, \mathcal R_W),$
is called {\em $\mathcal N$ projected onto $W$} and denoted $\pi_W(\mathcal N).$ 
\end{definition}
 
In other words, the projection of $(\mathcal S, \mathcal C, \mathcal R)$ onto $W$ is obtained by
deleting the species $X_i$, $i\in\complement W$ from all reactions
in $\mathcal R$ and further removing the resulting reactions for which the
source and the target complexes are the same. Here and from now on $\complement W$ denotes 
the complement of $W$ in $\{1,\ldots, n\}$.

If $(\mathcal N, K)=(\mathcal S, \mathcal C, \mathcal R, K)$ is a reaction system and
$c(t)=(c_1(t),\ldots, c_n(t))$ is a solution of the corresponding
system of differential equations 
(\ref{eq:reactionSystem}) with initial condition $c_0\in\mathbb
R^n_{\ge 0},$ then $\pi_W(c)(t)$ is a solution of the following
system of differential equations: 
\begin{equation}\label{eq:projj}
\frac{d}{dt}\pi_W(c) = \sum_{Q\to Q'\in\mathcal R_W}\Bigg(\sum_{
\substack{{\{P\to P'\in\mathcal
  R:}\\ {\pi_W(P)=Q, \pi_W(P')=Q'\}}}} \overline K_{P\to P'}(t, \pi_{W}(c)) \Bigg)(Q'-Q)
\end{equation}
with initial condition $\pi_W(c_0).$
Equation (\ref{eq:projj}) is obtained from (\ref{eq:reactionSystem})
by $(i)$ keeping only the equations for $\dot c_i$ with $i\in W$; 
$(ii)$ writing $K_{P\to P'}(t,c)=\overline K_{P\to P'}(t,\pi_{W}(c))$ 
to illustrate that $c_i,$ $i\in\complement W,$ are written either in
terms of $c_i,$ $i\in W,$ or as functions of $t;$
and $(iii)$
lumping together the rates of reactions $P\to P'$ that project to the
same reaction in $\mathcal R_W.$ The system of differential equations (\ref{eq:projj})
defines a kinetics $K_W$ for  $\pi_W(\mathcal N),$ where, for any
reaction $Q\to Q'\in {\cal R}_W,$ $K_{W,Q\to Q'}$ is given by the sum from
the parentheses in (\ref{eq:projj}). 
We call the resulting reaction system $(\pi_W(\mathcal N), K_W)$
a {\em projection of $\mathcal N$ onto $W$.} 
Note that $\overline K$ is not unique. Which variables $c_i,$ $i\in
\complement W$
are written in terms of $t$ and which are written in terms
of $c_i,$ $i\in W$ is a matter of context. 
For instance, Example \ref{ex:projNet} below describes two different functions $\overline K$
associated with system (\ref{ex:MA}) projected onto $\{1,2\}.$

A natural way of defining $\overline {K}$ 
for $\kappa$-variable mass-action systems is to include $c_i,$
$i\in\complement W$ in the rate-constant function:
$\overline K(t,c)=\overline{\kappa}(t)\pi_W(c)^{\pi_W(P)},$ where 
$\overline{\kappa}(t)=\kappa(t)\pi_{\complement W}(c(t))^{\pi_{\complement
    W}(P)}.$ The differential equations (\ref{eq:projj}) in this case are

\begin{equation}\label{eq:projjj}
\frac{d}{dt}\pi_W(c) = \sum_{Q\to Q'\in\mathcal R_W}\Bigg(
\sum_{\substack{{\{P\to P'\in\mathcal
  R:}\\ {\pi_W(P)=Q, \pi_W(P')=Q'\}}}} \kappa_{P\to P'}(t)\pi_{\complement
  W}(c)^{\pi_{\complement W}(P)}\Bigg) \pi_{W}(c)^{Q} (Q'-Q).
\end{equation}
Note that this projection has the form of $\kappa$-variable mass-action,
with rate-constant function for the reaction $Q\to Q'\in \mathcal R_W$ given by the
second sum in (\ref{eq:projjj}).
However, this rate-constant is not necessarily bounded in a compact interval
of $(0,\infty).$

\begin{example}\label{ex:projNet}
Following (\ref{eq:projjj}), the projection of reaction system 
(\ref{ex:MA}) onto $W=\{1,2\}$
i.e., onto species $A$ and $B,$
can be written in the $\kappa$-variable mass-action form (without
necessarily being $\kappa$-variable mass-action):
$$B\mathop\rightleftharpoons^{\overline{\kappa}_1(t)}_{\overline{\kappa}_2(t)} A
\mathop{\to}^{\overline{\kappa}_3(t)} 
A+B\mathop\to^{\kappa_4(t)} 0\qquad
2A\mathop\to^{\kappa_5(t)} A\mathop\gets^{\overline{\kappa}_6(t)} 0$$ 
with differential equations 
\begin{eqnarray*}
\dot c_A &=&\overline{\kappa}_1(t)c_B-\overline{\kappa}_2(t)c_A-\kappa_4(t)c_Ac_B
-\kappa_5(t)c_A^2+\overline{\kappa}_6(t)\\\nonumber
\dot c_B
&=&-\overline{\kappa}_1(t)c_B+\overline{\kappa}_2(t)c_A+\overline{\kappa}_3(t)c_A-\kappa_4(t)c_Ac_B
\end{eqnarray*}
where (recall equation (\ref{eq:exMA}))
$\overline{\kappa}_1(t)=\kappa_1(t)c_D(t),$ 
$\overline{\kappa}_2(t)=\kappa_2(t)c_C(t),$ 
$\overline{\kappa}_3(t) =\kappa_3(t)c_C(t),$ and 
$\overline{\kappa}_6(t) =\kappa_6(t)c_D(t)^2.$

On the other hand, let $T=\{c(t)\mid t\ge 0\}$ be a trajectory of (\ref{ex:MA}) with
initial condition $(\alpha,\ \beta,\ \gamma,\ \eta)\in\mathbb R_{>0}^4.$
Then $c_A+c_D=\alpha+\eta,$ $c_B+c_C=\beta+\gamma$ and therefore
$\pi_{\{1,2\}}(T)$ is a trajectory of the following projection of (\ref{ex:MA}): 
$$B\mathop\rightleftharpoons^{\overline K_1}_{\overline K_2} A
\mathop{\to}^{\overline K_3} 
A+B\mathop\to^{\overline K_4} 0\qquad
2A\mathop\to^{\overline K_5} A\mathop\gets^{\overline K_6} 0,$$
where, denoting $\pi_{\{1,2\}}(c)=(x,y),$ the rate function $\overline K(t,(x,y))$ is given by  
$\overline K_1(t,(x,y))=\kappa_1(t)y(\alpha+\eta-x),$ 
$\overline K_2(t,(x,y))=\kappa_2(t)x(\beta+\gamma-y),$
$\overline K_3(t,(x,y))=\kappa_3(t)x(\beta+\gamma-y),$
$\overline K_4(t,(x,y))=\kappa_4(t)xy,$
$\overline K_5(t,(x,y))=\kappa_5(t)x^2$ and 
$\overline K_6(t,(x,y))=\kappa_6(t)(\alpha+\eta-x)^2.$
\end{example}

\begin{remark}\label{rem:projWR}
Let $W\subset \{1,\ldots, n\}$ and $P,P'\in\mathcal C.$ If $\pi_W(P)\neq \pi_W(P'),$
then any directed path in $\mathcal R$ from $P$ to $P'$ projects onto a directed 
path from $\pi_W(P)$ to $\pi_W(P')$ in $\mathcal R_W$. If, on the other
hand, $\pi_W(P)=\pi_W(P'),$ then a directed path from $P$ to $P'$ either projects onto 
a cycle in ${\cal R}_W$ or is eliminated by the projection.

Therefore projection preserves weak reversibility: if
$\mathcal N$ is weakly reversible, then so is 
$\pi(\mathcal N).$ This result appears in \cite{pantea:thesis} and is
also the object of Lemma 3.4. in \cite{anderson:oneLC}.
\end{remark} 
\end{subsection}

\begin{subsection}{Complex-balanced systems and deficiency of a
    network}\label{sec:comBal}
Complex-balanced systems are defined in the context of mass-action
kinetics, i.e., the rate-constants are fixed positive numbers. 
\begin{definition}\label{def:compBal}
An equilibrium $c_*\in \mathbb R_{\ge 0}^n$ of a mass-action system
$(\mathcal R, \mathcal S, \mathcal C, \kappa)$ is  called {\em complex-balanced
equilibrium} if, at $c_*$, for any complex $P_0\in\cal C,$ 
the flow into $P_0$ is equal to the flow out of $P_0.$
More precisely, for each $P_0\in{\mathcal C}$ we have
$$\sum_{P\to P_0}\kappa_{P\to P_0}{c_*}^P=\sum_{P_0\to P}\kappa_{P_0\to P}{c_*}^{P_0}.$$
A complex-balanced system is a mass-action system that admits a
strictly positive complex-balanced equilibrium. 
\end{definition}

\begin{definition}\label{def:deficiency}
Let $(\mathcal S,\mathcal C,\mathcal R)$ be a reaction network with $m$ complexes,
$l$ linkage classes and whose stoichiometric subspace has dimension
$s.$ The {\em deficiency} of the reaction network $\mathcal R$ is 
$m-l-s.$
\end{definition} 

The deficiency of a reaction network is always non-negative
\cite{feinberg:lectures}. 
It has been shown that weakly reversible systems whose deficiency is
equal to zero are complex-balanced \cite{feinberg:lectures}. This
remarkable fact reveals a large class of mass-action systems which are
complex-balanced regardless of the choice of their rate constants.
\end{subsection}

\begin{subsection}{Persistence and the sub-tangentiality
    condition}\label{sec:persdef} $\text{ }$ 

\begin{definition}\label{def:pers} A trajectory
$T(c_0)=\{(x_1(t),\ldots, x_n(t))\mid t\geq 0\}$ 
with positive initial condition $c_0\in\mathbb R_{>0}^n$ 
of an n-dimensional dynamical system is called {\em persistent} if 
$$\liminf_{t\to \infty} x_i(t)>0\ for\ all\ i\in\{1,\ldots, n\}.$$
\end{definition}

Some authors call a trajectory that satisfies the condition in
Definition \ref{def:pers} {\em strongly persistent} \cite{takeuchi}. In
their work, persistence requires only that  $\limsup_{t\to \infty}
x_i(t)>0$ for all $i\in\{1,\ldots, n\}.$ We say that {\em a dynamical
system (or a reaction system)
is persistent} if all its trajectories with strictly positive initial
condition are persistent. 

\begin{definition}\label{def:omega}
Let $T(c_0)=\{{\bf x}(t)\mid t\geq 0\}$ denote a forward trajectory of a
dynamical system with initial condition $c_0\in \mathbb R_{>0}^n.$ 
The {\em $\omega$-limit set of $T(c_0)$} is
$$\textstyle\lim_{\omega} T(c_0) = \{l\in \mathbb R^n\mid \lim_{n\to\infty}c(t_n) = l \text{ for some sequence } t_n\to\infty\}.$$
The elements of $\lim_{\omega} T(c_0)$ are called $\omega$-limit points of $T(c_0)$.
\end{definition}

Note that a bounded trajectory of a dynamical system with positive initial condition  
is persistent iff it has no $\omega$-limit points on 
$\partial\mathbb R_{\geq 0}^n.$ 

In this paper we will prove persistence of trajectories $T(c_0)$ for
reaction systems with special properties. Our approach will consist of showing that a
certain convex polyhedron included in $\mathbb R^n_{> 0}$
contains $T(c_0).$ To this end we will use the following version of a result of Nagumo
\cite{blanchini}. Recall
that for a closed, convex set $K\subset \mathbb R^n$ and for ${\bf x}\in
K$ the {\em normal cone of $K$ at $\bf x$} is defined as follows:
$$N_K({\bf x})=\{{\bf n}\in\mathbb R^n\mid {\bf n}\cdot({\bf y}-{\bf
  x})\le 0\text{ for all }{\bf y}\in K\}.$$

\begin{thm}[Nagumo, \cite{blanchini}]\label{thm:nagumo}
Let $K\subset \mathbb R^n$ be a closed, convex set. Assume that
the system $\dot {\bf x}(t)=f(t,{\bf x}(t))$ has unique solution for any initial value, and let  
$T(c_0)=\{{\bf x}(t)\mid t\ge 0,\ x(0)=c_0\}$ be a forward trajectory of this
system with $c_0\in K.$ If for any $t_0\ge 0$ such that $x(t_0)\in
\partial K$ we have the {\em sub-tangentiality condition}
$${\bf n}\cdot f(t_0,x(t_0))\le 0\text{ for all } {\bf n}\in N_{K}(
x(t_0))$$
then $T(c_0)\subset K.$
\end{thm}
\end{subsection}

\end{section}

\begin{section}{Lower-endotactic networks}\label{sec:endo} 
Let $\mathcal N = (\mathcal S, \mathcal C, \mathcal R)$ be a reaction network
with species $X_1,\ldots, X_n$ and let $S\subseteq \mathbb R^n$ denote
its stoichiometric subspace. By a useful abuse of notation, we view
the source complexes of $\mathcal N$ as lattice points in  $\mathbb Z^n:$
$${\mathcal{SC}}({\mathcal N}) = \{(m_1,\ldots, m_n)\in
\mathbb{Z}_{\geq0}^n\mid m_1X_1+\ldots + m_nX_n\in\mathcal C\text{ is a
  source complex}\}.$$

In this section we revisit the notion of {\em lower-endotactic
  network}, first introduced in \cite{craciun_nazarov_pantea}
 for the case of two-species networks, and we extend it to {\em planar
   reaction networks}, defined below. 
We let $\mathrm{aff}(\mathcal N)$ denote the affine hull of $\mathcal C$, i.e. the
minimal affine subspace of $\mathbb R^n$ that contains $\mathcal C.$ 
\begin{definition}\label{def:planar}
The reaction network $\mathcal N$ is called
planar if $\mathrm{dim(aff}(\mathcal N))\le 2.$
\end{definition}

Let $\mathrm{aff}_+(\mathcal N)=\mathrm{aff}(\mathcal N)\cap \mathbb R^n_{\ge 0}.$
The following definition is similar to the one in \cite[section 4]{craciun_nazarov_pantea}. 

\begin{definition}\label{def:esupp}
 Let $\mathcal N=(\mathcal S,\mathcal C, \mathcal R)$ be a reaction network such that $\mathrm{aff}(\mathcal
 N)$ has dimension two and let $\bf v$ be a vector in S.

(i) The {\em $\bf v$-essential subnetwork} 
$\mathcal N_{\bf v}=(\mathcal S,\mathcal C_{\bf v}, \mathcal R_{\bf v})$ of 
$(\mathcal S, \mathcal C, \mathcal R)$ is defined by the reactions of
$\mathcal R$ whose reaction vectors are not orthogonal to $\bf v$: 
$$\mathcal R_{\bf v} = \{P\to P'\in{\cal R}\mid  (P'-P)\cdot {\bf v} \ne 0\};$$
$\mathcal C_{\bf v}$ is defined as the set of complexes appearing in
reactions of  $\mathcal R_{\bf v}.$

(ii) The {\em $\bf v$-essential support} of 
$\mathcal N$ is the supporting line $L$ of $\mathrm{conv}(\mathcal{SC(N}_{\bf v}))$
that is orthogonal to $\bf v$ and
 such that the positive direction of
$\bf v$ lies on the same side of $L$ as $\mathcal{SC(N}_{\bf v});$
(in other words, for any $P\in \mathrm{aff}(\mathcal
 N),$ the intersection of the half-line $\{P+t{\bf v} |, t\ge 0\}$ with the
half-plane bounded by $L$ that contains $\mathcal{SC(N}_{\bf v})$ 
is unbounded.)
 The line $L$ is denoted by $esupp^{\bf v}(\mathcal N).$
\end{definition}

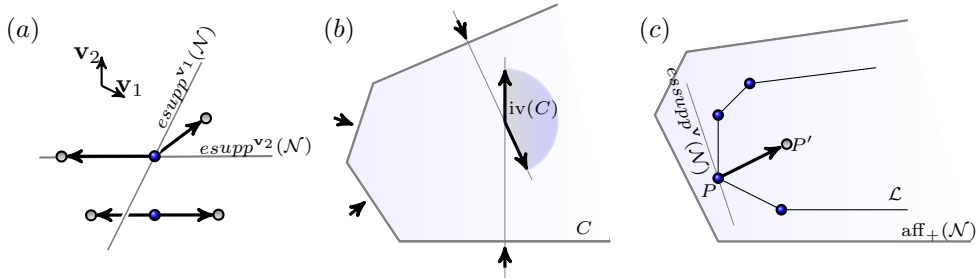
\begin{figure}
\begin{tikzpicture}[thick, >=stealth', scale = .35] 
\node[] at (-5, 7) {$(a)$};
\node[] at (7, 7) {$(b)$};
\node[] at (19, 7) {$(c)$};
\begin{scope}[shift={(0,0)}]
\begin{scope}[rotate={(-26.5)}]
\draw[->,thick]  (-4,3.5)--(-3,3.5);
\draw[->,thick]  (-4,3.5)--(-4.5,4.5);
\node at  (-3,3.9) {${\bf v}_1$};
\node at  (-5,4.3) {${\bf v}_2$};
\draw[gray] (-4.9,0) -- (3,4); 
\draw[<->, very thick] (-2,-1) -- (2,1); 
\draw[->, very thick] (-1,2) -- (0,4);
\draw[->, very thick] (-1,2) -- (-4,.5); 
\draw[draw = white, double=gray] (-1,-2) -- (-1,6); 
\shadedraw [shading=ball] (0,0) circle (.19cm);
\shadedraw [shading=ball] (-1,2) circle (.19cm); 
\shadedraw [] (2,1)+(.16,.09) circle (.19cm);
\shadedraw [] (-2,-1)+(-.16,-.09) circle (.19cm);
\shadedraw [] (-4,.5)+(-.16,-.09) circle (.19cm);
\shadedraw [] (0,4)+(.1,.16) circle (.19cm);
\scriptsize
\node[rotate=1.5] at (2.3,4) {$esupp^{{\bf v}_2}(\mathcal N)$};
\node[rotate=63.5] at (-1.3,5.2) {$esupp^{{\bf v}_1}(\mathcal N)$};
\end{scope}
\end{scope}
\begin{scope}[shift={(7.3,-1)}]
\begin{scope}[draw = gray]
\draw[clip] (10,0)--(2,0)--(0,3)--(1,6)--(8,9);
\shadedraw[left color = blue!6!white,right color=white,draw=gray] 
	(0,0) rectangle +(10,10); 
\end{scope}
\draw[gray] (10,0)--(2,0)--(0,3)--(1,6)--(8,9);

\begin{scope}[shift={(6,-1)}]
\begin{scope}[rotate = 90] 
\draw[gray, very thin] (-.4,0)--(8,0);
\draw[->, very thick] (0,0)--(.8,0);
\end{scope}
\end{scope}

\begin{scope}[shift={(.1,1)}]
\begin{scope}[rotate = 40]
\draw[->, very thick] (0,0)--(.8,0);
\end{scope}
\end{scope}

\begin{scope}[shift={(-.5,4.7)}]
\begin{scope}[rotate = -20]
\draw[->, very thick] (0,0)--(.8,0);
\end{scope}
\end{scope}

\begin{scope}[shift={(4.2,8.4)}]
\begin{scope}[rotate = -65]
\draw[gray, very thin] (-.4,0)--(6.7,0);
\draw[->, very thick] (0,0)--(.8,0);
\end{scope}
\end{scope}

\begin{scope}[shift={(6,4.53)}]
\shadedraw[left color=gray!30!white, right color=blue!20!white,
draw=blue!15!white] 
(0,0) -- (0,2) arc (90:-65:2) -- cycle; 
\draw[->, very thick] (0,0)--(0,2);
\begin{scope}[rotate=-65]
\draw[->, very thick] (0,0)--(2,0);
\end{scope}
\end{scope}

\scriptsize
\node[] at (9,.5) {$C$};
\node[] at (7.1,5) {$\mathrm{iv}(C)$};
\end{scope}
\begin{scope}[shift={(19,-1)}, scale=1.2]
\begin{scope}[draw = gray]
\draw[clip] (10, 0)--(2,0)--(0,4)--(1,6)--(8,7);
\shadedraw[left color = blue!6!white,right color=white,draw=white] 
	(0,0) rectangle +(10,8); 
\end{scope}
\draw[gray] (10, 0)--(2,0)--(0,4)--(1,6)--(8,7);
\scriptsize
\node at  (9,.3) {$\mathrm{aff}_+(\mathcal N)$};
\node at  (7.6,1.5) {$\mathcal L$};
\node at  (1.7,1.6) {$P$};
\node at  (4.65,3.1) {$P'$};
\node[rotate = -74] at  (1,3.8) {$essupp^{\bf v}(\mathcal N)$};
\draw[->, very thick] (2,2) -- (4,3); 
\draw[thin] (8,1)--(4,1)--(2,2)--(2,4)--(3,5)--(7,5.5);
\draw[very thin, gray] (2.5,.5)--(1,5);
\shadedraw [shading=ball] (4,1) circle (.15cm);
\shadedraw [shading=ball] (2,2) circle (.15cm);
\shadedraw [shading=ball] (2,4) circle (.15cm);
\shadedraw [shading=ball] (3,5) circle (.15cm);
\shadedraw [] (4.17,3.08) circle (.15cm);
\end{scope}
\end{tikzpicture} 
\caption{(a) Essential supports corresponding to ${\bf v}_1$ and ${\bf
  v}_2;$ (b) Example of a set of inward vectors; (c) Illustration for
  the proof of Lemma \ref{lem:sweepTest}.}\label{fig:esupp}
\end{figure}

Figure \ref{fig:esupp}(a) illustrates the notion of $\bf v$-essential
support for a planar reaction network with six complexes and four
reactions. This reaction network has two source complexes and note
that ${\cal N}_{{\bf v}_1}$ is equal to $\cal N,$  whereas ${\cal
  N}_{{\bf v}_2}$ is strictly smaller than $\cal N$ and contains
only one source complex.

We denote by $esupp^{\bf v}(\mathcal N)_{<0}$ the intersection
of $\mathrm{aff}_+(\mathcal N)$ with the
open half-plane in $\mathrm{aff}(\mathcal N)$ bounded by $esupp^{\bf v}(\mathcal N)$ 
that does not contain the positive direction of $\bf v:$ 
$$esupp^{\bf v}(\mathcal N)_{<0} = \{P\in \mathrm{aff}_+(\mathcal N)\mid (P - Q) \cdot {\bf
  v}< 0 \text{ for all } Q\in esupp^{\bf v}(\mathcal N)\}$$
and we define $esupp^{\bf v}(\mathcal N)_{>0}$ similarly.

\begin{definition}\label{def:innerSet}
Let $C\subset \mathbb R^n$ be a closed and convex set, and let
$S$ be the linear subspace of $\mathbb R^n$ such that  
the affine hull of $C$ is a translation of $S$. 
Then
$$\mathrm{iv}(C)=-\bigcup_{{\bf x}\in \partial C} (N_{C}({\bf x})\cap S)$$ 
is called the set of {\em inward vectors} of $C.$ Here $\partial C$
denotes the relative boundary of $C.$
\end{definition}

An example of a set of inward vectors for a two-dimensional set $C$ is
depicted in Figure \ref{fig:esupp}(b).

\begin{remark}\label{rem:ivC}
(i) $\mathrm{iv}(C)$ is a convex cone and if $C$ is bounded 
then $\mathrm{iv}(C)=S.$ 

(ii) If $C$ is a half-line then $\mathrm{iv}(C)$ consists of 
all vectors parallel with $C$ pointing in the unbounded 
direction of $C.$ If $C$ is a bounded  line segment, then  $\mathrm{iv}(C)$
consists of all vectors parallel with $C.$

(iii) If $\mathrm{aff}(C)$ has dimension two, then the set of inward vectors 
$\mathrm{iv}(C)$ is two-dimensional and consists of the 
normal vectors ${\bf v}\in S$ of all supporting lines $L$ of $C$ 
such that the positive direction of $\bf v$ is on the same 
side of $L$ as $C$ (see Figure \ref{fig:esupp}(b) for an example). 
\end{remark}

\begin{definition}\label{def:endo}
Let $\mathcal N$ be a planar reaction network with stoichiometric subspace $S.$ 
Then $\mathcal N$ is called {\em lower-endotactic} if the set 
\begin{equation}\label{eq:sweep}
\{P\to P'\mid P\in esupp^{\bf v}(\mathcal N) \text{ and } P'\in
esupp^{\bf v}(\mathcal N)_{<0}\}
\end{equation}
is empty for all nonzero vectors ${\bf v}\in \mathrm{iv}(\mathrm{aff}_+(\mathcal N)).$
\end{definition} 

Definition \ref{def:endo}(ii) is easily explained by the ``parallel
sweep test'' \cite{craciun_nazarov_pantea}. 
A reaction network $\mathcal N=(\mathcal S, \mathcal C, \mathcal R)$
is lower-endotactic if and only if it
passes the following test for any nonzero inward vector $\bf v$ of
$\mathrm{aff}_+(\mathcal N)$: sweep the plane $\mathrm{aff}(\mathcal N)$ with a line $L$ orthogonal to
$\bf v$, coming from infinity and going in the direction of $\bf v$, and stop when 
$L$ encounters a source complex corresponding to a reaction which is
not parallel to $L$. Now check  that no reaction with source on $L$
points towards the swept region. If $\mathcal R_{\bf v} = \O$, then all reaction vectors of 
$\mathcal R$ are perpendicular to $\bf v$ and $L$ never stops in the
parallel sweep test. In this case we still say that the network has passed the test for
$\bf v$. 

A reaction network is lower-endotactic if its reactions with sources
that are ``closest'' to the boundary of $\mathrm{aff}_+(\mathcal N)$
point ``inside'' $\mathrm{aff}_+(\mathcal N)$. 
Note that this special property of lower-endotactic networks in the
lattice space $\mathbb Z^n_{\ge 0}$ is analogous with the behavior of persistent trajectories in the phase space $\mathbb R^n_{\ge 0}:$ once a persistent
trajectory gets  ``close enough'' to the relative boundary of  its
stoichiometric compatibility class $S(c_0)$, it is pushed back
``inside''. In this sense, the requirement 
that a reaction network be lower-endotactic appears very naturally in
the context of persistence of a corresponding reaction system. 

\begin{remark}\label{rem:endo}
Following \cite{craciun_nazarov_pantea}, a planar reaction network 
$\mathcal N$ is called {\em endotactic} if the parallel sweep test
holds for all nonzero vectors ${\bf v}\in S$. An
endotactic network is also lower-endotactic; the two notions coincide
if $\mathrm{aff}_+(\mathcal N)$ is bounded. 
\end{remark}

\begin{remark}\label{rem:le_hyperplane}
The definition of endotactic networks has been extended in
\cite{craciun_nazarov_pantea} for networks that are not necesarilly planar,
using the parallel sweep test with hyperplanes instead of lines
(\cite[Remark 4.1]{craciun_nazarov_pantea}). 
Definition \ref{def:endo} is in fact a special case of the following more 
general definition of lower-endotactic networks:

\begin{definition}
A reaction network (not necessarily planar) with $n$ species is called
lower-endotactic if it passes the parallel sweep test for any inward
vector of the non-negative orthant $\mathbb R_{\ge 0}^n.$ 
\end{definition}  

Whereas the definition above is easier to state, the more technical
Definition \ref{def:endo} is better suited for planar networks in the
context of this paper.
\end{remark}

\begin{remark}\label{rem:wrEndo}
A weakly reversible reaction network $\mathcal N$ is always endotactic, and in
particular, lower-endotactic. Indeed, if $P\in esupp^{\bf v}(\mathcal N)$
for some vector ${\bf v}\in S$ and $P\to P'$ is a reaction of $\mathcal N$ then $P'\in
esupp^{\bf v}(\mathcal N)_{\ge 0},$ for otherwise the fact that $P'$ is
also a source complex would contradict $P\in esupp^{\bf v}(\mathcal N).$ 
\end{remark}

\begin{remark}\label{rem:2D}
If $\mathrm{aff}(\mathcal N)$ is one-dimensional we let $\mathcal P$ be a
two-dimensional affine subspace of $\mathbb R^n$ such that $\mathrm{aff}(\mathcal
N)\subset \mathcal P.$ The parallel sweep test for $\mathcal N$
with vectors of $\mathrm{iv}(\mathcal P\cap \mathbb R^n_{\ge 0})$ 
provides the same result as the ``true'' parallel sweep test
with vectors of $\mathrm{iv}(\mathrm{aff}_+(\mathcal N)).$ 
We may pretend that $\mathrm{aff}_+(\mathcal
N)$ coincides with the two-dimensional set $\cal P$ and therefore,
from this point of view, 
lower-endotactic planar reaction networks with one-dimensional
stoichiometric subspace do not need a special discussion. In what
follows, unless stated otherwise, we will assume that
$\mathrm{dim}(\mathrm{aff}(\mathcal N)) = 2.$ 
\end{remark}

Note that, if $\mathrm{aff}(\mathcal N)$ has dimension one, then
$\mathrm{aff}_+({\mathcal N})$ contains vectors with at most two possible 
positive directions (see Remark \ref{rem:ivC}).
The following lemma shows that, even if
$\mathrm{aff}(\mathcal N)$ has dimension two, the parallel sweep test only 
needs to be performed for a finite set of directions $\bf v$ 
(see also \cite[Proposition 4.1]{craciun_nazarov_pantea}):
\begin{lemma}\label{lem:sweepTest}
Let $\mathcal N$ be a planar reaction network with 
$\mathrm{dim(aff}(\mathcal N))=2.$ 

(i) If $\mathrm{aff}_+(\mathcal N)$ is bounded, then  
$\mathcal N$ is lower-endotactic if and only if it passes the parallel 
sweep test for vectors $\bf v$ that are orthogonal to a side of 
the polygon $\mathrm{conv}(\mathcal{SC(N)}).$ 

(ii) If $\mathrm{aff}_+(\mathcal N)$ is unbounded, then  
$\mathcal N$ is lower-endotactic if and only if it passes the parallel 
sweep test for vectors $\bf v$ that are either orthogonal to a side of 
the polygon $\mathrm{conv}(\mathcal{SC(N)}),$ or are generators 
of the cone  $\mathrm{iv}(\mathrm{aff}_+(\mathcal N))$.
\end{lemma}
\begin{proof}
The sides of $\mathrm{conv}(\mathcal {SC(N)})$ whose inward normal vectors are
in $\mathrm{iv}(\mathrm{aff}_+(\mathcal N))$ form a polygonal line
$\mathcal L$. As in Figure \ref{fig:esupp}(c), if $\mathrm{aff}_+(\mathcal N)$ is
unbounded, we augment $\mathcal L$ with
half-lines of directions given by the generators of $\mathrm{iv}(\mathrm{aff}_+(\mathcal N)).$ 
If a vector ${\bf v}\in \mathrm{iv}(\mathrm{aff}_+(\mathcal N))$ does not correspond to
$(i)$ or $(ii)$ in the statement of the lemma, then $esupp^{\bf v}(\mathcal N)$
contains exactly one vertex $P$ of $\mathcal L.$ Let $P\to P'\in\mathcal
R.$ Since the inward normal vectors of the two sides of $\mathcal L$
adjacent to $P$ belong to cases $(i)$ or $(ii)$ from the statement of
the lemma, it follows that $P'$ lies in the interior or on the sides  of the angle
$\angle P$ of $\mathcal L,$ and therefore in $P\in esupp^{\bf v}(\mathcal N)_{>0}.$ 
In conclusion, the parallel sweep test holds for all ${\bf v}\in \mathrm{iv}(\mathrm{aff}_+(\mathcal N))$ and
$\mathcal N$ is lower-endotactic.
\end{proof}

\begin{example} A few examples are illustrated in Figure
\ref{fig:examples}. The source complexes are depicted using solid
dots and the various lines represent the final
positions of the sweeping lines from Lemma \ref{lem:sweepTest}.
Note that the reaction networks in $(a)$ and $(b)$ look the same, 
but, since $\mathrm{aff}_+(\mathcal N)$ is unbounded in $(a)$ and
bounded in $(b),$ the reaction network in $(a)$ is lower-endotactic,
whereas the reaction network in $(b)$ is not. The same thing happens
for $(e)$ and $(d).$ 
\end{example}

\begin{figure}
\includegraphics[scale=.205]{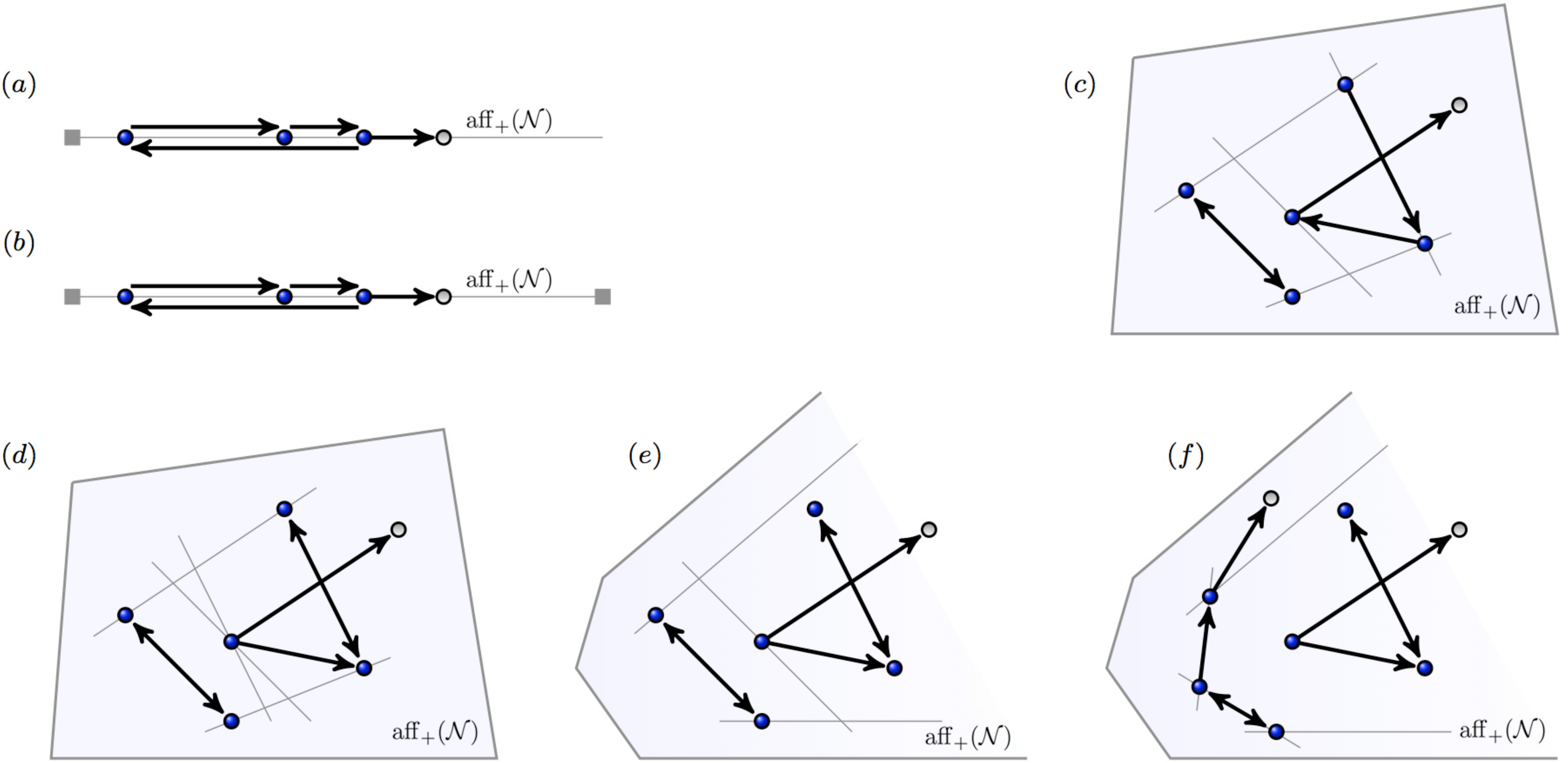}
\caption{Examples of lower-endotactic networks - (a), (c), (e) and non
  lower-endotactic 
networks - (b), (d), (f).}\label{fig:examples}
\end{figure}

\noindent {\bf Affine transformations of reaction networks.} An
important observation that is used often throughout this paper
is that projections of lower-endotatic networks are also lower-endotactic.   
We prove this fact in the larger context of {\em affine
  transformations of reaction networks}.
Let $\mathcal N=(\mathcal S, \mathcal C, \mathcal R)$ be a planar reaction network and consider an affine
transformation $U:\mathrm{aff}(\mathcal N)\to \mathbb R^d$ 
such that  $U(\mathrm{aff}_+(\mathcal N))\subset
\mathbb R^d_{\ge 0}.$ Similarly to the definition of a projected network,
we consider the ``generalized'' reaction network $U(\mathcal N),$ with reactions 
$U(\mathcal R)=\{U(P)\to U(P')\mid P\to P'\in\mathcal R\}$
and complexes in the set
$U(\mathcal C)$ which are allowed to
have nonnegative real coordinates. We have 
$\mathrm{aff}_+(U(\mathcal N)) = \mathrm{aff}(U(\mathcal N))\cap \mathbb
R^d_{\ge 0}$ and we may ask whether $U(\mathcal N)$ is lower-endotactic. 

\begin{prop}\label{prop:transform}
Let $\mathcal N$ be a planar reaction network and let  
$U:\mathrm{aff}(\mathcal N)\to\mathbb R^d$ be an affine
transformation such that
$U(\mathrm{aff}_+(\mathcal N))\subset \mathbb R^d_{\ge 0}.$
Then, if $\mathcal N$ is lower-endotactic, 
the planar reaction network $U(\mathcal N)$ is 
also lower-endotactic. Moreover,
if $\mathcal N$ is endotactic, then $U(\mathcal N)$ is also endotactic.
\end{prop} 
\begin{proof}
We show the lower-endotactic case;
the proof for the endotactic case is similar.
Also, we assume that $U$ has rank two; a simpler version of the argument below 
works if $U$ has rank one. 
Let  
${\bf w}\in \mathrm{iv(aff}_+(U(\mathcal N))),$ and let $L=esupp^{\bf w}(U(\mathcal N)).$ 
Since $U$ takes parallel lines to parallel lines, there
exists a vector $\bf v$ such that $U^{-1}(L)=esupp^{\bf v}(\mathcal N)$
and 
$$esupp^{\bf w}(U(\mathcal N))_{<0}=U(esupp^{\bf v}(\mathcal N))_{<0}).$$
Because $U(\mathrm{aff}_+(\mathcal N))\subseteq \mathrm{aff}_+(U(\mathcal N)),$ we have $\mathrm{iv(aff}_+(U(\mathcal N)))
\subseteq \mathrm{iv} (U(\mathrm{aff}_+(\mathcal N)))$ and therefore ${\bf w}\in
\mathrm{iv}(U(\mathrm{aff}_+(\mathcal N))$. 
It follows that ${\bf v}\in \mathrm{iv} (\mathrm{aff}_+(\mathcal N)).$ Then, if 
$U(P)\to U(P')\in U(\mathcal R)$ such that $U(P)\in esupp^{\bf w}(U(\mathcal N))$ 
and $U(P')\in esupp^{\bf w}(U(\mathcal N))_{<0}$
then $P\to P'\in\mathcal R,$ $P\in esupp^{\bf v}(\mathcal N)$ and 
$P'\in esupp^{\bf v}(\mathcal R)_{<0},$ contradicting the fact that $\mathcal
N$ is lower-endotactic.  
\end{proof}
\end{section}

\begin{section}{2D-reduced mass-action systems}\label{sec:2Dred}

A key ingredient in the proof of
our main persistence result consists of studying projections of
trajectories of $\kappa$-variable mass-action systems
onto well-chosen two-dimensional subspaces of $\mathbb R^n.$
These special projected trajectories obey a specific type of
dynamics which we call {\em 2D-reduced mass-action}.  
In this section we show that bounded forward
trajectories of such dynamical systems are persistent. To this end we will
extend significantly the ideas from \cite{craciun_nazarov_pantea},
where they were introduced in the context of
two-species $\kappa$-variable mass-action systems.

\begin{subsection}{Definition and comparison of reaction rates}\label{sec:2DredDef}
Fix an integer $n\ge 2$ and let $l,k$ be two fixed elements of
$\{1,\ldots, n\}$ such that $l<k.$ 
Let $p_i,\ q_i$ be nonnegative
rational numbers for $i\in\{1,\ldots, n\}$ such that, for any $i,$ not both $p_i$
and $q_i$ are zero,
and such that $p_l=q_k=1$ and $p_k=q_l=0.$
Denote
 \begin{equation}\label{eq:psi}
\Psi=
\left(
\begin{matrix}
p_1 &\ldots &p_{l-1} &1 &p_{l+1}&\ldots &p_{k-1} &0 &p_{k+1} &\ldots &p_n\\
q_1 &\ldots &q_{l-1} &0 &q_{l+1}&\ldots &q_{k-1} &1 &q_{k+1} &\ldots &q_n\\
\end{matrix}
\right)^t.
\end{equation}
\begin{definition}\label{def:projVarMassAct}
Let $\Psi$ be a matrix of the form (\ref{eq:psi}).

(i) Let $\mathcal N=(\mathcal S, \mathcal C, \mathcal R)$ be a 
reaction network with two species, let
$\kappa:\mathbb R_{\geq 0}\to \mathbb R_{>0}^{\mathcal R}$
be a piecewise differentiable function and let $a\in\mathbb R^n.$
For all reactions $P\to P'\in \cal R,$ we define $K_{P\to P'}:\mathbb R_{\ge 0}\times \mathbb R^2_{\ge 0}\to \mathbb R_{\ge 0},$
\begin{equation}\label{eq:rRate}
K_{P\to P'}(t,{\bf x}) = \kappa_{P\to P'}(t)(\Psi {\bf x})^{\Psi P+a}.
\end{equation}
The reaction system $(\mathcal N, K)$ is called 
a {\em 2D-reduced planar mass-action system} and is denoted by
$(\mathcal N, \Psi, \kappa, a).$

(ii) For each $s\in\{1,\ldots, p\},$ let $\mathcal N_s=(\mathcal S, \mathcal C_s,
\mathcal R_s)$ be a two-species reaction network and let $(\mathcal N_s,\Psi,\kappa_s, a_s)$ be a 2D-reduced
mass-action system.
The sum $(\mathcal N, K) = \bigcup_{s=1}^p (\mathcal N_s,\Psi,\kappa_s, a_s)$ 
(recall
Definition \ref{def:sums}) 
is called a {\em 2D-reduced mass-action system.} 
\end{definition}

Therefore the concentration vector $c(t)=(x(t),y(t))$ of a 2D-reduced
mass-action system $\bigcup_{s=1}^p (\mathcal N_s,\Psi,\kappa_s, a_s)$
satisfies the following  differential equation:
\begin{equation}\label{eq:projVarMassAct}
\dot c(t)=\sum_{s=1}^p \sum_{P\to P'\in\mathcal
      R_s}\kappa_{s,P\to P'}(t)(\Psi c(t))^{\Psi P+a_s}(P'-P).
\end{equation}
Note that, by definition, $\kappa$ needs not be bounded away from zero and
infinity, as is the case for $\kappa$-variable mass-action
systems. However, we will require this condition to prove
persistence of 2D-reduced mass-action systems in Corollary \ref{cor:proj}.

The goal of this section is to study the persistence of 
2D-reduced mass-action systems.
One important component of our analysis is highlighting the reaction
whose rate at time $t\ge 0$ ``dominates'' the other reaction rates.
In view of (\ref{eq:rRate}) we then consider, for $A>0$ and for any $s\in
\{1,\ldots, p\},$ 
the sign of the difference
\begin{equation}\label{eq:monomialDiff}
(\Psi {\bf x})^{\Psi P+a_s} - A(\Psi {\bf x})^{\Psi P'+a_s},
\end{equation}
\noindent for all pairs of distinct source complexes $P,\ P'$
of $\mathcal N_s.$  
For simplicity, and without loss of generality, we assume that
$k=2$ and $l=1.$ Then
(\ref{eq:monomialDiff}) has the 
same sign as the following expression, which we denote by $\Lambda_{\alpha,\beta}^A(x,y):$ 
\begin{equation}\label{eq:lambdaFunction}
\Lambda^A_{\alpha, \beta} (x,y) = x^\alpha y^{-\beta}(p_3x+q_3y)^{p_3\alpha-q_3\beta}\ldots(p_nx+q_ny)^{p_n\alpha-q_n\beta} - A^D.
\end{equation}
Here $(\alpha,-\beta)=D(P-P')$ and $D$ denotes the least
common denominator of all  nonzero $p_i$ and $q_i$, $i\in\{3,\ldots, n\}.$
Note that all the exponents in (\ref{eq:lambdaFunction}) are integers.

The geometry of the curves $\Lambda^A_{\alpha, \beta}(x,y)=0$ within
$\mathbb R_{>0}^2$
is very relevant to our discussion. An immediate goal, which we
pursue next, is to find simple approximations for these curves. We
will see that within appropriate subsets of $\mathbb R_{>0}^2,$ 
$\Lambda^A_{\alpha, \beta}(x,y)=0$ may be approximated by 
power curves $y=Cx^{\tau}$ that are ordered in a useful way, as we
will explain later in the paper. 
Let
\begin{equation}\label{eq:delta'delta''}
\overline{\Delta}=\left(1+\sum_{i=3}^n p_i\right)\alpha,\  
\overline{\overline{\Delta}}=\left(1+\sum_{i=3}^n q_i\right)\beta
\text{ and } \Delta =  \overline{\Delta} - \overline{\overline{\Delta}}.
\end{equation}

\begin{lemma}\label{lem:uniqueBranch} Suppose $\Delta\neq 0.$

(i) If $\alpha\beta < 0$  then for all $x>0$ there exists a unique $y>0$ such that $\Lambda^A_{\alpha,\beta}(x,y)=0.$

(ii) If $\alpha\beta > 0$  then for all small
enough $x>0$ there exists a unique $y>0$ such that $\Lambda^A_{\alpha,\beta}(x,y)=0.$
\end{lemma}
\begin{proof}
(i) If  $\alpha\beta<0,$ without loss of generality we may take $\alpha>0$
and  $\beta<0.$ 
For any fixed $x>0,$ $\Lambda_{\alpha,\beta}^A(x,y)$ is a polynomial in $y$ 
whose coefficients are all positive, except for its free term $-A^D.$ 
The Descartes rule of signs implies that this polynomial has a unique positive root.

\noindent (ii) If $\alpha\beta > 0,$ we may assume that $\alpha>0$ and  $\beta>0.$
$\Lambda^A_{\alpha,\beta}(x,y) = 0$ implies 
\begin{equation}\label{eq:LambdaSetBefore}
A^Dy^{\beta}(p_3x+q_3y)^{q_3 \beta}\ldots (p_nx+q_ny)^{q_n
  \beta}=x^{\alpha}(p_3x+q_3y)^{p_3 \alpha}\ldots (p_nx+q_ny)^{p_n \alpha}.
\end{equation}
\noindent We rewrite this equality by excluding the factors of zero
power and merging the powers of $x$ in the left hand side and the
powers of $y$ in right hand side. We denote 
\begin{equation}\label{eq:abA'}
\lambda_x=1+\sum_{\substack{{i\in\{3,\ldots,n|}\\{q_i=0\}}}}p_i
\quad \text{ and }
\lambda_y=1+\sum_{\substack{{i\in\{3,\ldots,n|}\\{p_i=0\}}}}q_i,
\end{equation}
and we let $i_l$, $l\in\{1,\ldots, I\}$ be the indices for
which both  $p_{i_l}$ and $q_{i_l}$ are strictly positive.
Then we have
\begin{equation}\label{eq:LambdaSet}
A'y^{\lambda_y\beta}(p_{i_1}x+q_{i_1}y)^{q_{i_1}\beta}\ldots (p_{i_I}x+q_{i_I}y)^{q_{i_I}\beta} 
-x^{\lambda_x\alpha}(p_{i_1}x+q_{i_1}y)^{p_{i_1} \alpha}\ldots (p_{i_I}x+q_{i_I}y)^{p_{i_I} \alpha} = 0
\end{equation}
where $A' = A^D\prod_{i\in\{1,\ldots, n\}\backslash\{i_1,\ldots,
  i_I\}} q_i^{q_i\beta} p_i^{-p_i\alpha}. $ We denote the polynomial
in (\ref{eq:LambdaSet}) by $F(x,y).$

If $\lambda_y\beta>\overline\Delta-\lambda_x\alpha$ then equation
(\ref{eq:LambdaSet}) yields 
\begin{eqnarray*}
F(x,y)&=&\overline{\overline{C}}_0 y^{\overline{\overline{\Delta}}}+(\overline{\overline{C}}_1\ x)y^{\overline{\overline{\Delta}}-1}+\ldots+(\overline{\overline{C}}_{\overline{\overline{\Delta}}-\lambda_y\beta}\ x^{\overline{\overline{\Delta}}-\lambda_y\beta})y^{\lambda_y\beta}\\
&-&(\overline C_{\lambda_x\alpha} x^{\lambda_x\alpha})y^{\overline\Delta-\lambda_x\alpha}-(\overline C_{\lambda_x\alpha+1}\ x^{\lambda_x\alpha+1})y^{\overline\Delta-\lambda_x\alpha-1}-(\overline C_{\overline\Delta -1}\ x^{\overline\Delta -1})y-\overline C_{\overline\Delta}\ x^{\overline\Delta}=0,
\end{eqnarray*}
where the coefficients $\overline{\overline C}_k$ and $\overline C_k$
are positive and are obtained from expanding the first, respectively
second term of the difference (\ref{eq:LambdaSet}),
and if $\lambda_y\beta\leq\overline\Delta-\lambda_x\alpha$ we have 
\begin{eqnarray*}\label{eq:Fxy2}
F(x,y)&=&\overline{\overline{C}}_0y^{\overline{\overline{\Delta}}}+(\overline{\overline{C}}_1\
x)y^{\overline{\overline{\Delta}}-1} +\ldots+
(\overline{\overline{C}}_{\overline{\overline{\Delta}}-(\overline\Delta-\lambda_x\alpha+1)}\ x^{\overline{\overline{\Delta}}-(\overline\Delta - \lambda_x\alpha+1)})y^{\overline\Delta - \lambda_x\alpha+1}+\\
&+&(\overline{\overline
  C}_{\overline{\overline\Delta}-(\overline\Delta-\lambda_x\alpha)}\
x^{\overline{\overline\Delta}-(\overline\Delta-\lambda_x\alpha)} -
\overline C_{\lambda_x\alpha}
x^{\lambda_x\alpha})y^{\overline\Delta-\lambda_x\alpha}+\ldots\\
&+&(\overline{\overline C}_{\overline{\overline\Delta}-\lambda_y\beta}\ x^{\overline{\overline\Delta}-\lambda_y\beta} - \overline C_{\overline\Delta-\lambda_y\beta}\ x^{\overline\Delta-\lambda_y\beta})y^{\lambda_y\beta}\\
&-&(\overline C_{\overline\Delta - (\lambda_y\beta-1)}\ x^{\overline\Delta - (\lambda_y\beta-1)})y^{\lambda_y\beta-1}-\ldots-
(\overline C_{\overline\Delta -1}\ x^{\overline\Delta -1})y-\overline C_{\overline\Delta}\ x^{\overline\Delta}=0.
\end{eqnarray*}

In the first case, for any fixed $x>0,$ the coefficients of $F_x(y) = F(x,y)$ viewed as a polynomial in 
$y$ change sign exactly once. It follows from the Descartes rule of signs that $F_x(y)=0$ 
has a unique positive solution. In the second case, the coefficients of $F_x(y)$ that are 
binomials in $x$ are of the form 
$\overline{\overline C}x^{\overline{\overline\Delta}-k}-\overline C x^{\overline\Delta - k}$ 
and  for small $x$ are all either positive if $\overline{\overline\Delta}<\overline\Delta$ 
or negative if $\overline{\overline\Delta}>\overline\Delta.$ Therefore for small enough 
$x$ the polynomial $F_x(y)$ changes the sign of its coefficients only once either at 
$y^{\overline\Delta-\lambda_x\alpha}$ if $\Delta<0$ or at $y^{\lambda_y\beta-1}$ if $\Delta>0$. 
It follows that the equation $F_y(x)=0$ a unique positive solution.
\end{proof}

\begin{remark}\label{rem:yalphabeta}
Lemma $\ref{lem:uniqueBranch}$ implies that for $\alpha\beta<0,$ the curve
$\{(x,y)\in\mathbb R_{>0}^2\mid\Lambda^A_{\alpha,\beta}(x,y)=0\}$ is
the graph of a function $y^A_{\alpha,\beta}:\mathbb R_{>0}\to \mathbb
R_{>0}.$  It is easy to see that
$\lim_{x\to 0} y^A_{\alpha,\beta}(x) = \infty \text{ and  } \lim_{x\to \infty} y^A_{\alpha,\beta}(x) = 0.$
On the other hand, if $\alpha\beta > 0,$ the function
$y^A_{\alpha,\beta}$ is defined only for small $x>0$: 
there exists $M_{\alpha,\beta}^A>0$ such that 
$\{(x,y)\in(0,M_{\alpha,\beta}^A)\times \mathbb R_{>0}\mid\Lambda^A_{\alpha,\beta}(x,y)=0\}$
is the graph of $y^A_{\alpha,\beta}:(0,M_{\alpha,\beta}^A)\to \mathbb R_{>0}.$
We claim that in this case we have $\lim_{x\to
  0}y_{\alpha,\beta}^A=0.$ 
Indeed, suppose $\Delta<0.$ If for some
$0\le l<\infty,$ $(0,l)$ is a limit point of the curve $\{(x,y)\in\mathbb R^2_{>0}\mid
\Lambda^A_{\alpha,\beta}(x,y)=0\},$ then  plugging $(0,l)$ into
(\ref{eq:LambdaSet}) yields $l=0.$ It remains to check that
$(0,\infty)$ is not a limit point of  the curve above. 
If $\{(x_n,y_n)\}_{n>0}$ is a sequence of points 
with positive coordinates such that $\lim_{n\to \infty}(x_n,y_n) =
(0,\infty)$ and $\Lambda^A_{\alpha,\beta}(x_n,y_n) = 0,$ 
from (\ref{eq:LambdaSet}) we get
$\mathcal O(x_n^{\lambda_x\alpha})=\mathcal O(y_n^{-\Delta +
  \lambda_x\alpha})$ 
as $n\to\infty,$ which contradicts $\Delta<0.$ 
The case $\Delta>0$ follows from symmetry.
\end{remark}

\begin{lemma}\label{lem:powerFunction} Suppose $\alpha\beta>0.$

(i) If $\Delta = 0$ there exist positive
constants $\gamma^A_{\alpha,1},\ldots, \gamma^A_{\alpha, N}$ for some
integer $N>1$ such that 
$$\Lambda^A_{\alpha,\beta}(x,y)=0\text{ for some }(x,y)\in\mathbb R_{>0}^2 \text{ if and only if }y=\gamma^A_{\alpha, i}x \text{ for some }i\in \{1,\ldots, N\}.$$

(ii) If $\Delta\neq 0$ there exists a strictly
increasing function $\tau:\mathbb R _{>0}\to\mathbb R_{>0}$ such that, 
for function $y_{\alpha,\beta}^A$ introduced in Remark
\ref{rem:yalphabeta}, the limit
$$\displaystyle\lim_{x\to 0}\frac{y^A_{\alpha,\beta}(x)}{x^{\tau(\alpha/\beta)}}$$
\noindent exists, is positive and finite.
We denote this limit by $C^A_{\alpha, \beta}.$
\end{lemma}
\begin{proof}
$(i).$ Dividing (\ref{eq:LambdaSet}) by
$x^{\overline\Delta}=x^{\overline{\overline\Delta}}$ and letting $\gamma=y/x$ yields
$$A'\gamma^{\lambda_y\beta}(p_{i_1}+q_{i_1}\gamma)^{q_{i_1}\beta} \ldots (p_{i_I}+q_{i_I}\gamma)^{q_{i_I}\beta}-(p_{i_1}+q_{i_1}\gamma)^{p_{i_1}\alpha} \ldots (p_{i_I}+q_{i_I}\gamma)^{p_{i_I}\alpha}=0.$$
We denote by  $\overline F(\gamma)$ the polynomial above. 
The positive term in the expression of $\overline F$ has degree $\overline{\Delta},$
and the negative term has degree $\overline\Delta - \lambda_x\alpha,$
therefore $\lim_{\gamma\to\infty}\overline F(\gamma)=\infty.$ Since
$\overline F(0)<0,$ there exists at least one positive root of
$\overline F.$ Denoting the positive roots of $\overline F$ by
$\gamma^A_{\alpha,1},\ldots, \gamma^A_{\alpha, N}$ completes the proof.

\noindent $(ii)$ We know from see Remark \ref{rem:yalphabeta} that  
$(0,0)$ is a limit point of  
$\{(x,y)\in\mathbb R_{>0}^2\mid \Lambda_{\alpha,\beta}^A(x,y)=0\}.$
Lemma \ref{lem:uniqueBranch} implies
that this curve has a unique Puiseux expansion 
in a neighborhood of $(0,0)$ (see \cite{sturmfels:pol_eq} for a
discussion of Puiseux expansions). By making
$M_{\alpha,\beta}^A$ from Remark \ref{rem:yalphabeta} as small as
necessary for the Puiseux expansion to hold in
$x\in(0,M_{\alpha,\beta}^A)$, we have, for all
$x\in(0,M_{\alpha,\beta}^A)$: 
\begin{equation}\label{puiseux}
y_{\alpha,\beta}^A(x) = C_{\alpha,\beta}^A x^T + \text{ higher order terms in } x
\end{equation}
\noindent where $C_{\alpha,\beta}^A\in\mathbb R$ and $T >0$ is a
rational number. The exponent $T$ in (\ref{puiseux})
is equal to the negative of one of the slopes in the lower  boundary 
of the Newton polygon of the polynomial $F(x;y)$ defined in 
(\ref{eq:LambdaSet}); (see \cite{sturmfels:pol_eq} for more details). 
\begin{figure}
\begin{center}
\begin{tabular}{cc}
\hspace{-.5cm}
\begin{tikzpicture}[thick, >=stealth, scale = .31] 
\draw[very thin,gray](0,0) -- (16,0); 
\draw[very thin,gray](0,0) -- (0,16); 

\scriptsize
\node at (-1.5, 15){$(0, \overline \Delta)$};
\node at (11.5, 7){$(\overline \Delta-\lambda_x\alpha,\lambda_x\alpha)$};
\node at (10, -1){$(\overline{\overline \Delta},0)$};
\node[fill=white] at (.5,6){$(\lambda_y\beta, \overline{\overline \Delta}-\lambda_y\beta)$};

\draw[very thin,gray](0,15) -- (15,0); 
\draw[very thin,gray](0,10) -- (10,0); 

\draw[very thick] (10,0)--(4,6)--(0,15);

\shadedraw [shading=ball] (0,15) circle (.2cm);
\shadedraw [shading=ball] (1,14) circle (.2cm);
\shadedraw [shading=ball] (2,13) circle (.2cm);
\shadedraw [shading=ball] (6,9) circle (.2cm);
\shadedraw [shading=ball] (7,8) circle (.2cm);
\shadedraw [shading=ball] (8,7) circle (.2cm);

\shadedraw [shading=ball] (10,0) circle (.2cm);
\shadedraw [shading=ball] (9,1) circle (.2cm);
\shadedraw [shading=ball] (8,2) circle (.2cm);
\shadedraw [shading=ball] (6,4) circle (.2cm);
\shadedraw [shading=ball] (5,5) circle (.2cm);
\shadedraw [shading=ball] (4,6) circle (.2cm);

\node[rotate=-45, fill=white] at (4,11){$\bf \ldots$} ;
\node[rotate=-45, fill=white] at (7,3){$\bf \ldots$} ;
\end{tikzpicture} 
&\hspace{.8cm}
\begin{tikzpicture}[thick, >=stealth, scale = .31] 
\draw[very thin,gray](0,0) -- (16,0); 
\draw[very thin,gray](0,0) -- (0,16); 

\scriptsize
\node at (-1.5, 10){$(0, \overline \Delta)$};
\node[fill=white] at (2.7, 3.5){$(\overline \Delta-\lambda_x\alpha,\lambda_x\alpha)$};
\node at (15, -1){$(\overline{\overline \Delta},0)$};
\node[fill=white] at (10.5,8.6){$(\lambda_y\beta, \overline{\overline \Delta}-\lambda_y\beta)$};

\draw[very thin,gray](0,15) -- (15,0); 
\draw[very thin,gray](0,10) -- (10,0); 

\draw[very thick] (0,10)--(6,4)--(15,0);

\shadedraw [shading=ball] (0,10) circle (.2cm);
\shadedraw [shading=ball] (1,9) circle (.2cm);
\shadedraw [shading=ball] (2,8) circle (.2cm);
\shadedraw [shading=ball] (4,6) circle (.2cm);
\shadedraw [shading=ball] (5,5) circle (.2cm);
\shadedraw [shading=ball] (6,4) circle (.2cm);

\shadedraw [shading=ball] (15,0) circle (.2cm);
\shadedraw [shading=ball] (14,1) circle (.2cm);
\shadedraw [shading=ball] (13,2) circle (.22cm);
\shadedraw [shading=ball] (9,6) circle (.22cm);
\shadedraw [shading=ball] (8,7) circle (.2cm);
\shadedraw [shading=ball] (7,8) circle (.2cm);

\node[rotate=-45, fill=white] at (11,4){$\bf \ldots$} ;
\node[rotate=-45, fill=white] at (3,7){$\bf \ldots$} ;
\end{tikzpicture} 
\end{tabular}
\end{center}
\caption{Monomials of $F(x;y)$ (see \ref{eq:LambdaSet}). The lattice point $(a,b)$ represents the monomial $x^by^a$ and the lower boundary of negative slope of the 
associated Newton polygon is the thick polygonal line. The picture on the left corresponds to $\Delta>0$ and the one on the right corresponds to $\Delta<0.$}\label{fig:newton}
\end{figure}
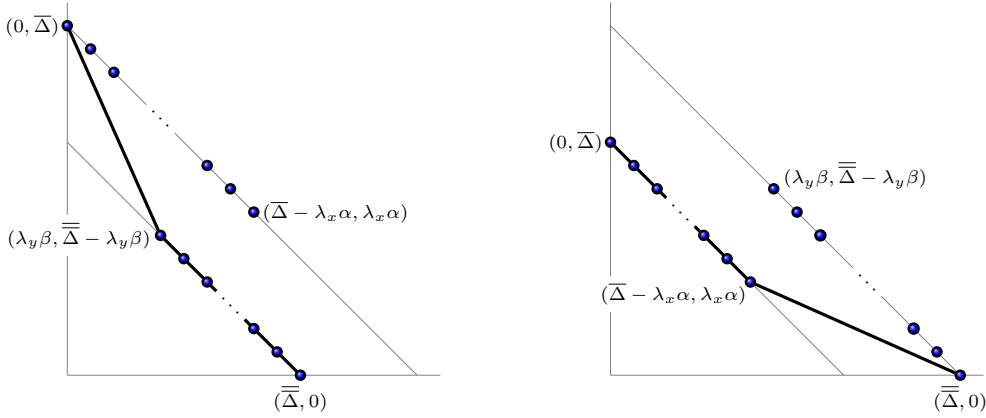
Since $F$ is the difference of two
homogeneous polynomials of degrees $\overline{\Delta}$ and
$\overline{\overline{\Delta}},$ its Newton polygon can be easily illustrated
(see Figure \ref{fig:newton}) and the slopes of its lower boundary are 
\begin{equation}\label{eq:newtonAB}
\left\{-1,
  -\frac{\overline\Delta-(\overline{\overline\Delta}-\lambda_y\beta)}{\lambda_y\beta}\right\}
\text{ or }\left\{-1, -\frac{\lambda_x\alpha}{\overline{\overline\Delta}-(\overline\Delta-\lambda_x\alpha)}\right\} 
\end{equation}
if $\Delta > 0$ or  $\Delta < 0,$ respectively.
Note that $T$ cannot be equal to one, for otherwise from (\ref{eq:LambdaSet}) it
would follow that $\mathcal O(x^{\overline\Delta})=\mathcal
O(x^{\overline{\overline\Delta}})$, contradicting $\overline\Delta\neq\overline{\overline\Delta};$ therefore $-T$ is given by the fractions in
(\ref{eq:newtonAB}). Then, if we define the function 
$\tau:\mathbb R _{>0}\to \mathbb R_{>0}$  as follows:
$$\tau(\sigma)=
\begin{cases}
\lambda_x\left(\displaystyle\frac{1}{\sigma}\displaystyle\sum_{i=1}^nq_i-\sum_{i=1}^n
  p_i\right)^{-1},&\text{ if } \displaystyle\sigma < \sum_{i=1}^n q_i/\sum_{i=1}^n p_i,\\
1, &\text{ if } \displaystyle\sigma = \sum_{i=1}^n q_i/\sum_{i=1}^n p_i,\\
\displaystyle{\frac{1}{\lambda_y}}\left(\sigma\sum_{i=1}^np_i-\sum_{i=1}^nq_i\right),
&\text{ if }  \displaystyle\sigma > \sum_{i=1}^n q_i/\sum_{i=1}^n p_i,
\end{cases}
$$
we can easily see that $T=\tau(\alpha/\beta).$ Note that $\tau$ is
continuous and strictly increasing. Moreover, the value of $\sigma$
for which $\tau(\sigma)=1$ corresponds to $\Delta=0$,
and thus the statement in part $(i)$ of the lemma is incorporated in
part $(ii).$
\end{proof}

\end{subsection}

\begin{subsection}{The domination lemma}\label{sec:domLem} 
The following key result reinforces our
motivation for considering differences (\ref{eq:monomialDiff}) of
reaction rates. Roughly speaking, it
shows that if, at time $t,$ the rate of a reaction $P'_0\to P_0$ dominates all the
other reaction rates, then the reaction vector $P'_0-P_0$ ``dictates''
the direction of the flow $\dot c(t).$
\begin{lemma}\label{lem:staysInside}
Let $(\mathcal N, K)=(\mathcal S,\mathcal C,\mathcal R, K)$ be a reaction system, let $P_0\to P'_0\in {\mathcal R}$ and let $\bf{v}$ be a vector such that $(P'_0-P_0)\cdot {\bf v}>0.$ Also let  $\mathcal U\subseteq \mathcal{SC}(\mathcal N)\backslash \{P_0\}.$
There exists a positive constant  $\mu$ such that if  for some $(t, {\bf
x})\in\mathbb R_{\ge 0}\times \mathbb R_{\ge 0}^n$ we have $K_{P_0\to
P'_0}(t,{\bf x})>0$ and 
$$K_{P_0\to P'_0}(t,{\bf x})  > \mu K_{P\to P'}(t,{\bf x})  \text{ for all } P\to P'\in\mathcal R \text{ with } P\in \mathcal U$$
\noindent then 
$$\Bigg(K_{P_0\to P'_0}(t,{\bf x})(P'_0-P_0)+\sum_{P\to P', P\in
    \mathcal U }K_{P\to P'}(t,{\bf x}) (P'-P)\Bigg)\cdot{\bf{v}} > 0.$$
\end{lemma}
\begin{proof} We take 
$$\mu = \frac{\vectornorm{{\bf v}}\displaystyle\sum_{P\to P'\in\mathcal R}\vectornorm{P-P'}}{(P'_0-P_0)\cdot \bf{v}}$$
\noindent and we have 
\begin{eqnarray*} \label{case2}
&&\Bigg(K_{P'_0\to P_0}(t,{\bf x})(P'_0-P_0)+\sum_{P\to P', P\in \mathcal U }K_{P\to P'}(t,{\bf x}) (P'-P)\Bigg)\cdot{\bf{v}} \\\nonumber
&&> K_{P_0\to P'_0}(t,{\bf x})(P'_0-P_0)\cdot {\bf v} - \sum_{P\to P', P\in \mathcal U} \Bigg((1/\mu) K_{P_0\to P'_0}(t,{\bf x})\vectornorm{{\bf v}}\vectornorm{P-P'} \Bigg ) \\
&&\ge\left((P'_0-P_0)\cdot {\bf{v}} - (1/\mu) \vectornorm{{\bf v}}\sum_{P\to P'\in \mathcal R}\vectornorm{P'-P}\right)K_{P_0\to P'_0}(t,{\bf x}) = 0,
\end{eqnarray*}
\noindent where the first inequality was obtained using the Cauchy-Schwarz inequality.
\end{proof}

\begin{remark}\label{rem:AfromMu}
 If $(\mathcal S, \mathcal C, \mathcal R, \Psi,\kappa, a)$ is a
 2D-reduced planar mass-action system 
and if $\kappa(t)\in(\eta,1/\eta)^{\mathcal R}$ for some $t\ge 0$ then 
$K_{P\to P'}(t,{\bf x})=\kappa_{P\to P'}(t)(\Psi {\bf x})^{\Psi P+a},$ and therefore
$(\Psi {\bf x})^{\Psi P_0+a} > (\mu/\eta^2)(\Psi {\bf x})^{\Psi P+a}$ implies 
$K_{P_0\to P'_0}(t,{\bf x}) > \mu K_{P\to P'}(t,{\bf x})$, which
is exactly the condition needed in Lemma \ref{lem:staysInside}.  Or,
using the notation in (\ref{eq:lambdaFunction}) and letting  $(\alpha,-\beta)=D(P_0-P):$
\begin{equation}\label{eq:AfromMu}
\Lambda^{\mu/\eta^2}_{\alpha,\beta} ({\bf x}) > 0 \text{ implies }
K_{P_0\to P'_0}(t,{\bf x}) > \mu K_{P\to P'}(t,{\bf x}).
\end{equation} 
\end{remark}
\end{subsection}

\begin{subsection}{Geometric constructions in the phase plane}\label{sec:constr} 
Our strategy for proving that 
a trajectory $T(c_0)$ of a certain reaction system is persistent
relies on building a convex set $\mathcal L^+\subset \mathbb R^2_{>0}$ that contains
$T(c_0)$ and stays away from $\partial \mathbb R^2_{\ge 0}.$ As in \cite{craciun_nazarov_pantea},
we partition the phase plane into subsets where one reaction rate dominates all the others, and
therefore, by Lemma \ref{lem:staysInside}, its corresponding reaction vector dictates the
direction of the vector field. The set $\mathcal L^+$ is constructed such that, on each 
subset of the partition, the dominating reaction vector (and therefore, by Lemma \ref{lem:staysInside},
the vector field), points towards the interior of  $\mathcal L^+$. This is the rather simple 
idea behind the proof of Theorem \ref{thm:proj}, but the technical details involved are quite
delicate. We start with the construction of the set $\mathcal L^+$, which is discussed next.  

For any $s\in\{1,\ldots, p\},$ let $\mathcal N_s=(\mathcal S,\mathcal
C_s,\mathcal R_s)$ be a lower-endotactic two-species  
reaction network and let 
$(\mathcal N, K)=\bigcup_{s=1}^p(\mathcal N_s, \kappa_s, \Psi, a_s)$ 
be a 2D-reduced mass-action
system. 
We also let $\eta<1$ be a positive constant.  
Let $D$ denote the  least common denominator of all nonzero elements
of $\Psi.$ For each $s\in\{1,\ldots, p\}$ let 
$$\{r^s_1,\ldots,r^s_{e(s)}\}=\big\{\alpha/\beta\mid (\alpha,-\beta)\in
\{D(P-P')\mid P,P' \in {\mathcal{SC}}(\mathcal N_s),\alpha\beta >0\}\big\};$$
we assume that  $r_1^s<\ldots <r_{e(s)}^s$ and define the set
\begin{equation}\label{eq:ris}
{\bf V} = \{(1,\tilde r_1),\ldots,(1,\tilde r_e)\}=\bigcup_{s=1}^p
\{(1,r^s_i)\mid i\in\{1,\ldots,e(s)\}\}
\end{equation}
where we take $\tilde r_1 < \ldots < \tilde r_e.$ 
Let $\{\bf i,j\}$ be the standard
basis of the cartesian plane. 
Since $\mathcal N_s$ is endotactic, for  each
$s\in\{1,\ldots, p\}$ and for all vectors $\bf n\in{\bf V}\cup\{\bf i,
\bf j\},$ there exists a reaction 
\begin{equation}\label{eq:importantReaction}
P_{s,\bf n}\to P'_{s,\bf n}\in \mathcal R_s
\text{ such that } 
P_{s,\bf n}\in esupp^{\bf n}(\mathcal N_s) 
\text{ and } 
P'_{s,\bf n}\in esupp^{\bf n}(\mathcal N_s)_{>0}. 
\end{equation}

Note that there might exist multiple reactions as in
(\ref{eq:importantReaction}), out of which $P_{s,\bf n}\to P'_{s,\bf
  n}$ is chosen and fixed
for the remaining of this paper.

If $\mu_{{\bf n},s}$ denotes the  constant from Lemma
\ref{lem:staysInside} that corresponds to reaction 
$P_{s,\bf n}\to P'_{s,\bf n}$ $\bf v = \bf n$ and 
$\mathcal U=\mathcal{SC(N}_s)\cap esupp^{\bf n}(\mathcal N_s)_{>0}$ we define
\begin{equation}\label{eq:delta} 
\mu = \max\big\{\mu_{{\bf n},s}\mid {\bf n}\in{\bf V}\cup\{{\bf i, j}\},
s\in\{1,\ldots, p\}\big\}.
\end{equation}

We also let 
$$\mathcal D = \bigcup_{s=1}^p\{D(P-P')\mid P, P'\in \mathcal{SC(N}_s)\},$$
and finally, inspired by Remark \ref{rem:AfromMu}, we denote $A = \mu/\eta^2.$  

Let $M>1$ be a fixed number. 
We choose $0<\delta<1$ and $0<\xi<1$ to satisfy the following properties: 

\begin{myindentpar}{.3cm}
\begin{tabular*}{\textwidth}{ l @{\extracolsep{4cm}} l}
\hspace{-.2cm}(P1) & $\delta<\displaystyle\frac{1}{2}\min\{C_{\alpha,\beta}^A,\gamma_{\alpha,i}^A \mid (\alpha,-\beta)\in{\mathcal D},\alpha\beta >0\}$,\\ 
&$\displaystyle\frac{1}{\delta} >\frac{3}{2}\max\{C_{\alpha,\beta}^A,\gamma_{\alpha,i}^A \mid
(\alpha,-\beta)\in {\mathcal D},\alpha\beta >0\}$\\
\end{tabular*}
\noindent where $C_{\alpha,\beta}^A$ and $\gamma_{\alpha,i}^A$ are defined in Lemma \ref{lem:powerFunction};\\
(P2) all pairwise intersections from the strictly positive quadrant of the  $2e$ curves 
$y = \delta x^{\tau(\tilde r_{i})},\ y = (1/\delta) x^{\tau(\tilde r_{i})},$ $i\in\{1,\ldots, e\},$ lie in $(\xi,\infty)^2$;\\
(P3) the square $[0,\xi]^2$ lies below the curves $\Lambda_{\alpha,\beta}^A(x,y)=0$ for all 
$(\alpha, -\beta) \in {\mathcal D} \text{ with } \alpha\beta\le 0;$ \\
(P4) for all $(\alpha,-\beta)\in \mathcal D$ such that $\alpha\beta>0$ we
have $\xi < M^A_{\alpha,\beta}$ (recall that $M^A_{\alpha,\beta}$ is such that
$y_{\alpha,\beta}^A:(0,M_{\alpha,\beta}^A)$ admits a Puiseux series
representation), and for all $x\in (0,\xi]$ we have
$$\frac{1}{2}C^A_{\alpha,\beta}x^{\tau(\alpha/\beta)}<y^A_{\alpha,\beta}(x)<\frac{3}{2}C^A_{\alpha,\beta}x^{\tau(\alpha/\beta)}.$$\\
\end{myindentpar}

\begin{figure}
\hspace{-1cm}
\begin{tikzpicture}[>=stealth, scale =.41]

\begin{scope}
\clip (0,0) rectangle + (32,32);
\begin{scope}
\clip (.5,32)--(.5,13.25)--(2,8)--(5,4)--(9,2)--(13.5,.5)--(32,.5)--(32,32);
\draw[inner color = blue!5,outer color=blue!5,draw=white] (-40,-40) rectangle +(100,100);
\end{scope}
\end{scope}

\draw (0,0)--(32,0);
\draw (0,0)--(0,32);

\node at (16.8,-.5) {$(\xi,0)$};
\node at (-.9,16.2) {$(0,\xi)$};
\node[fill=white] at (0,13.25) {$A_1$};
\node[fill=white] at (13.4,0) {$A_{e+1}$};
\node at (-.5,15){$A_0$};
\node at (15,-.5){$A_{e+2}$};

\fill[black] (0,15) circle (.12);
\fill[black] (15,0) circle (.12);

\begin{scope}
\clip (0,0) rectangle + (32,32);

\draw[color=red, dashed, very thick] plot[domain = 3:32] (\x,300/\x);
\draw[color=red, dashed,very thick] plot[domain = 3:32] (\x,400/\x - \x/2);
\draw[color=red, dashed,very thick] plot[domain = 3:32] (500/\x - \x/1.5,\x);

\draw[color=gray, very thin] plot[domain = 0:32] (\x,\x^2/170);
\draw[color=gray, very thin] plot[domain = 0:32] (\x,\x^2/60);

\draw[color=gray, very thin] plot[domain = 0:32] (\x^2/50,\x);
\draw[color=gray, very thin] plot[domain = 0:32] (\x^2/190,\x);

\draw[color=gray, very thin] plot[domain = 0:23] (\x^3/520,1.1*\x^2/16);
\draw[color=gray, very thin] plot[domain = 0:23] (\x^3/270,\x^2/16);

\draw[color=gray,very thin, dashed] plot[domain = 0:32] (\x,\x^2/100);
\draw[color=gray,very thin, dashed] plot[domain = 0:32] (\x^2/100,\x);
\draw[color=gray,very thin, dashed] plot[domain = 0:23] (\x^3/400,\x^2/16);

\draw[color=gray, very thin] (0,15)--(.5,13.25);
\draw[color=gray, very thin] (15,0)--(13.5,.5);

\begin{scope}
\clip (0,0) rectangle + (16,16);
\draw[thick] plot[domain = 0:32] (\x,\x^2/170);
\draw[thick] plot[domain = 0:32] (\x,\x^2/60);

\draw[thick] plot[domain = 0:32] (\x^2/50,\x);
\draw[thick] plot[domain = 0:32] (\x^2/190,\x);

\draw[thick] plot[domain = 0:23] (\x^3/520,1.1*\x^2/16);
\draw[thick] plot[domain = 0:23] (\x^3/270,\x^2/16);

\draw[color=red,dashed,very thick] plot[domain = 0:32] (\x,\x^2/100);
\draw[color=red,dashed,very thick] plot[domain = 0:32] (\x^2/100,\x);
\draw[color=red,dashed,very thick] plot[domain = 0:23] (\x^3/400,\x^2/16);
\end{scope}

\node[rotate = 76, fill=blue!5, text = black] at (2.1,20) {$y=(1/\delta) x^{\tau(\tilde r_1)}$};
\node[rotate = 61, fill=blue!5, text = black] at (3.4,13.2) {$y=\delta x^{\tau(\tilde r_1)}$};

\node[rotate = 29, fill=blue!5, text = black] at (19,6) {$y=(1/\delta) x^{\tau(\tilde r_e)}$};
\node[rotate = 11, fill=blue!5, text = black] at (20,2.5) {$y=\delta x^{\tau(\tilde r_1)}$};

\node[rotate = 37, fill=blue!5, text = black] at (12.9,14.4) {$y=\delta x^{\tau(\tilde r_2)}$};
\node[rotate = 47, fill=blue!5, text = black] at (8.5,14) {$\alpha/\beta = \tilde r_2$};
\node[rotate = 56, fill=blue!5, text = black] at (5.4,13.5) {$y=(1/\delta) x^{\tau(\tilde r_2)}$};
\node[rotate = 77, fill=blue!5, text = black] at (1.9,14) {$\alpha/\beta = \tilde r_1$};
\node[rotate = 15, fill=blue!5, text = black] at (14,2) {$\alpha/\beta = \tilde r_e$};

\node[rotate = -50, fill=blue!5, text = black] at (24,4.7) {$\beta = 0$};
\node[rotate = -21, fill=blue!5, text = black] at (29,10.3) {$\alpha\beta < 0$};
\node[rotate = -16, fill=blue!5, text = black] at (30,13) {$\alpha = 0$};

\node[fill=white] at (1.4,8) {$A_2$};
\node at (4.4,4) {$A_3$};
\node[fill=white] at (8.7,1.6) {$A_e$};

\draw [gray](16,0)--(16,16);
\draw [gray](0,16)--(16,16);
 
\draw [blue, very thick](.5,13.25)--(.5,32);
\draw [blue, very thick](.5,13.25)--(2,8);
\draw [blue, very thick](2,8)--(5,4);

\draw [blue, very thick, dashed](5,4)--(9,2);

\draw [blue,very thick](9,2)--(13.5,.5);
\draw [blue,very thick](13.5,.5)--(32,.5);

\fill[blue] (.5,13.25) circle (.12);
\fill[blue] (2,8) circle (.12);
\fill[blue] (5,4) circle (.12);
\fill[blue] (9,2) circle (.12);
\fill[blue] (13.5,.5) circle (.12);
\fill[blue] (29,.5) circle (.12);
\fill[blue] (.5,29) circle (.12);

\node at (1.7,29){$(d,M)$};
\node at (29, 1){$(M,d)$};

\node[rotate = 60, fill=white, text=blue] at (7,3) {\normalsize$\approx$};
\node at (30,30){\huge$\mathcal L^+$};
\end{scope}
\end{tikzpicture} 
\caption{Construction of $\mathcal L^+.$ Each dashed line represents a curve $\Lambda^A_{\alpha,\beta}=0$ in the positive quadrant and is labeled with a relation that the corresponding $\alpha$ and $\beta$ satisfy. The curves $y=\delta x^{\tau(\tilde r_i)}$ and $y=(1/\delta) x^{\tau(\tilde r_i)}$ are depicted using solid lines. 
The highlighted area represents the set $\mathcal L^+.$}\label{fig:bigPicture}
\end{figure}
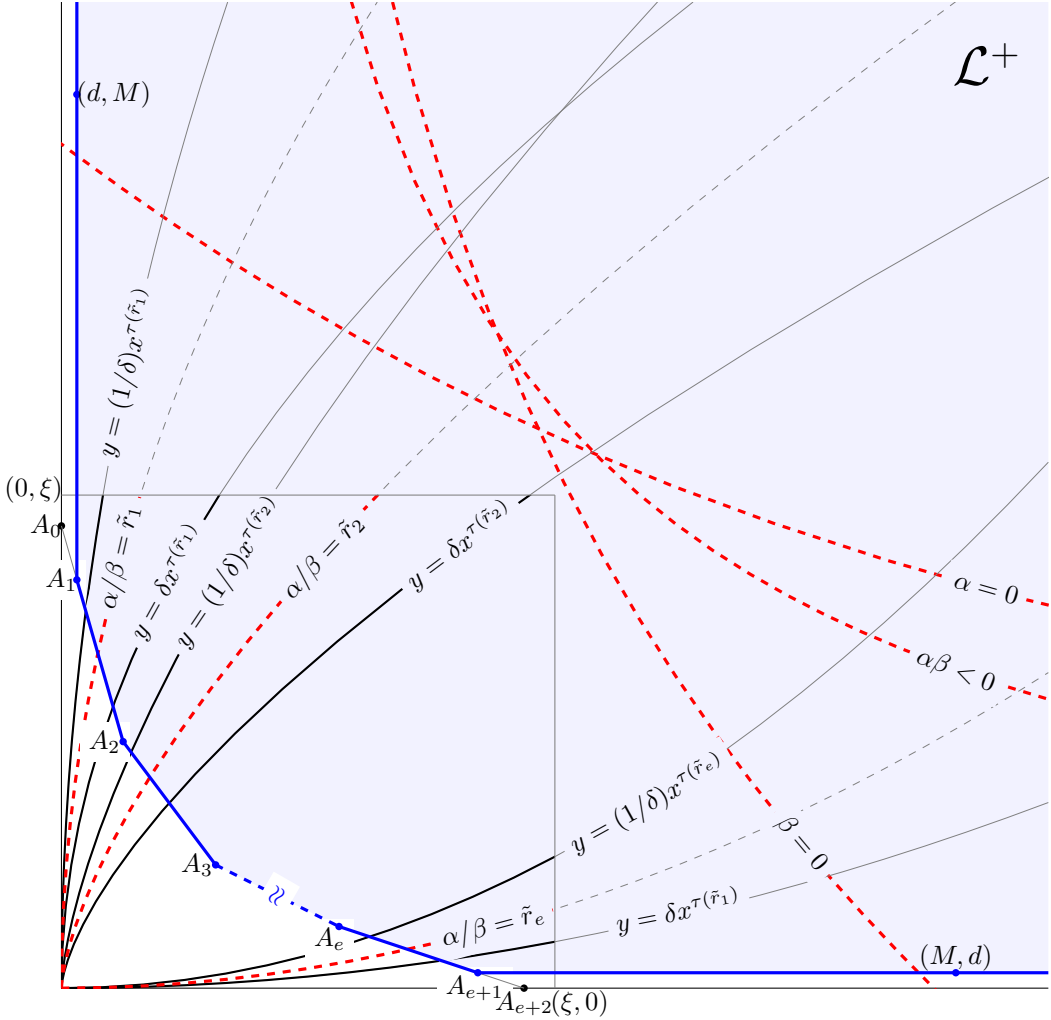

Clearly, $\xi$ can be chosen small enough so that (P2) and (P3)
are satisfied (see Figure \ref{fig:bigPicture}). The existence of $\xi$ that also satisfies (P4) is a consequence of Lemma \ref{lem:powerFunction}. Since $\tau$ is a strictly increasing function,  condition (P2)  implies that for $x\in (0,\xi]$ we have 
$$(1/\delta) x^{\tau(\tilde r_{i+1})} < \delta x^{\tau(\tilde r_i)} \text{ for all } i\in\{1,\ldots, e-1\}.$$
Moreover, for $x<\xi,$ (P1) and (P4) imply that for $\alpha\beta>0$
the curve $\{(x,y)\in\mathbb R_{>0}^{2}\mid
\Lambda^A_{\alpha,\beta}(x,y) = 0\}$ lies between the curves $\delta
x^{\tau(\alpha/\beta)}$ and   $(1/\delta) x^{\tau(\alpha/\beta)}$
(see Figure \ref{fig:bigPicture}).

Now, for the actual construction of $\mathcal L^+,$ first assume that
${\bf V}\neq\O$ and start with a point $A_0\in
\{0\}\times (0,\xi]$ on the $y$ axis. We choose a point $A_2$ between the curves 
$y = (1/\delta) x^{\tau(\tilde r_2)}$  and  $y = \delta x^{\tau(\tilde
  r_1)}$ such that the slope of the line $A_0A_2$ is $-1/\tilde
r_{1}.$ Inductively, choose points $A_{i+1},\ i\in\{2,\ldots, e-1\}$ such that 
\begin{eqnarray*}
&& (i)\ \text{ the point }A_{i+1} \text{ lies between the curves }y = (1/\delta) x^{\tau(\tilde r_{i+1})}  \text{ and } y = \delta x^{\tau(\tilde r_{i})}\\
&& (ii)\text{ the line } A_iA_{i+1} \text{ has slope  } -1/\tilde r_{i}.
\end{eqnarray*}
Finally, $A_{e+2}$ is defined on the $x$ axis such that $A_eA_{e+2}$ has slope $-1/\tilde r_e.$

The polygonal line 
$[A_0A_2A_3\ldots A_{e-1}A_eA_{e+2}]$ is convex. 
We move  $A_0$ closer to the origin if necessary, such that all the
points $A_i$ defined above lie in the square $(0,\xi]^2.$  
  
If ${\bf V}=\O$ we let $A_0$ and $A_3$ be two points on the $y$ and $x$ axes, respectively, such that $A_0,\ A_3\in[0,\xi]^2$ and the slope of the line $A_0A_3$ is -1.

The last step of the construction consists of defining $d$ small enough such
that (P5) and (P6) below are satisfied:  

\begin{myindentpar}{.3cm}
(P5) If ${\bf V}\neq \O$ then the vertical half-line $\{d\}\times [0,\infty)$ intersects the segment $A_0A_{2}$ at a point $A_1$ above the curve $y=(1/\delta)x^{\tau(\tilde r_1)}$; 
also, the horizontal half-line $[0,\infty)\times \{d\}$ intersects the
segment $A_eA_{e+2}$ at a point $A_{e+1}$ below the curve $y=\delta
x^{\tau(\tilde r_e)}.$ 
If ${\bf V}= \O$ then $d$ is chosen such that $d<\xi/2,$ and the
intersection of the vertical  half-line $\{d\}\times [0,\infty)$ with the
segment $A_0A_{3}$ is denoted $A_1,$ whereas the intersection of the
horizontal half-line $[0,\infty)\times \{d\}$ with the segment
$A_0A_3$ is denoted $A_{2};$ 

\begin{tabular*}{\textwidth}{ l @{\extracolsep{4cm}} l}
\hspace{-.45cm}(P6) & $\Lambda^A_{\alpha,\beta}(d, M) > 0\text{ for } (\alpha,-\beta)\in {\mathcal D},\alpha<0, \beta\ge 0,$\\ 
& $\Lambda^A_{\alpha,\beta}(M, d) > 0\text{ for } (\alpha,-\beta)\in {\mathcal D},\alpha\ge 0, \beta < 0;$\\
\end{tabular*}
\end{myindentpar}

Recall that $M>1$ was fixed at the beginning of the construction. It
is easy to see that (\ref{eq:lambdaFunction}) implies that (P6) holds for $d$ small enough.

Let 
$$\mathcal L = \{A_1+(0,t)\mid t\ge 0\}\cup [A_1\ldots A_{e+1}]\cup \{A_{e+1}+(t,0)\mid t\ge 0\}$$ 
be made out of the polygonal line $[A_1\ldots A_e],$ (called, for
future reference, {\em the finite part of} $\mathcal L$)  completed with a
vertical and a horizontal half-lines (the union of which we call {\em
  the infinite part of} $\mathcal L$). Finally, we define $\mathcal L^+ = \mathrm{conv}(\mathcal L).$

To indicate the quantities that $\mathcal L^+$ depends on, we write 
$\mathcal L^+ = \mathcal L^+(\{\mathcal N_s\}_{1\le s\le p}, \Psi, \eta, M).$
The polygonal line $[A_0,\ldots, A_{e+2}]$ will also be useful later
in this paper. We denote it by 
$\overline{\mathcal L} = \overline{\mathcal L}(\{\mathcal N_s\}_{1\le s\le p},
\Psi, \eta)$ (note that $M$ is not required for $\overline{\mathcal L}$) 
and we let $\overline{\mathcal L}^+$ be the unbounded
part of $\mathbb R^2_{\ge 0}$ that is delimited by $\overline{\mathcal L}.$

\end{subsection}

\begin{subsection}{Persistence of 2D-reduced mass-action systems} 
The following theorem is the  main result of this section.

\begin{thm}\label{thm:proj}
Let 
$(\mathcal N, K)=\bigcup_{s=1}^p(\mathcal N_s,\kappa_s, \Psi, a_s)$ 
be a 2D-reduced mass-action system
where $\mathcal N_s=(\mathcal S, \mathcal C_s, \mathcal R_s)$
is lower-endotactic for all $s\in\{1,\ldots,
p\}$ and denote
$$f(t,{\bf x})=\sum_{s=1}^p \sum_{P\to P'\in\mathcal
      R_s}\kappa_{s,P\to P'}(t)(\Psi {\bf x})^{\Psi P+a_s}(P'-P).$$
Let $\eta\in (0,1)$ and  $M>1$ be real numbers. Then for any $t\ge 0$ such
that $\kappa_s(t)\in(\eta,1/\eta)^{\mathcal R_s}$ for all
$s\in\{1,\ldots, p\}$, and for any ${\bf
  x}\in {\mathcal L}(\{\mathcal N_s\}_{1\le s\le p}, \Psi, \eta, M)\cap
[0,M]^2$ we have
\begin{equation}\label{eq:thmstate}
{\bf n}\cdot f(t,{\bf x})\ge 0 \text{ for all } {\bf n}\in -N_{\mathcal
  L^+}(\bf x).
\end{equation}
\end{thm}
\begin{proof} 
The cone $ -N_{\mathcal L^+}({\bf x})$ is
degenerate unless ${\bf x}$ is a vertex of $\mathcal L$ and its
generators belong to $\bf V \cup \{\bf i,j\},$ where $\bf V$ is the set of
vectors defined in (\ref{eq:ris}). It then suffices to show that 
for any $s\in\{1,\ldots, p\}$ and for any 
${\bf n}\in -N_{\mathcal L^+}({\bf x})\cap (\bf V \cup \{i,j\})$
we have
\begin{equation}\label{eq:newMainIneq}
\Bigg(\sum_{P\to P'\in \mathcal R_{s,\bf n}} K_{s, P\to P'}(t, {\bf x})(P'-P)\Bigg)\cdot {\bf n}\ge 0
\end{equation}
for all $s\in\{1,\ldots, p\},$ where 
$K_{s,P\to P'}(t,{\bf x})=\kappa_{s,P\to P'}(t)(\Psi {\bf
  x})^{\Psi P+a_s}$ is the rate of the reaction $P\to P'\in\mathcal
R_s$ at time $t.$
We fix $s\in\{1,\ldots, p\}$ and recall that 
$\mathcal R_{s,\bf n} = \{P\to P'\in {\mathcal R}_s\mid (P'-P)\cdot{\bf
  n}\neq 0\}$ denotes the $\bf n$-essential  subnetwork of $\mathcal R_s.$
 The inequality (\ref{eq:newMainIneq})
is trivially true if $\mathcal R_{s,\bf n}=\O.$ Otherwise, recall the 
reaction $P_{s,\bf n}\to P'_{s,\bf n}\in \mathcal
R_{s,\bf n}$ from (\ref{eq:importantReaction}). 
We rewrite the left hand side of (\ref{eq:newMainIneq}) by separating the reactions with source on  $esupp^{\bf n}(\mathcal N_s)$ and emphasizing the reaction $P_{s,\bf n}\to P'_{s,\bf n}:$

\begin{eqnarray}\label{eq:cdot}\nonumber
&&\Bigg(\sum_{P\to P'\in \mathcal R_{s,\bf n}}K_{s,P\to P'}(t,{\bf x})
  (P'-P)\Bigg )\cdot {\bf n} = \Bigg(\sum_{\substack{{\{P\to P'\in \mathcal
        R_{s,\bf n}}|\\{P\in esupp^{{\bf n}}(\mathcal N_s)}\\{P\to P' \neq P_{s,\bf n}\to P'_{s,\bf n}
      \}}}}
K_{s,P\to P'}(t,{\bf x})(P'-P)+\\\nonumber
&&+K_{s,P_{s,\bf n}\to P'_{s,\bf n}}(t,{\bf x}) (P'_{s,\bf n}-P_{s,\bf n})
+\sum_{\substack{{\{P\to P'\in\mathcal N_{s,\bf n}}|\\ {P\notin esupp^{{\bf n}}(\mathcal N_s)\}}}}K_{s,P\to P'}(t,{\bf x}) (P'-P)\Bigg)\cdot{\bf n}.\\\nonumber
\end{eqnarray}
Since all source complexes of $\mathcal R_{s,\bf n}$ lie in
$esupp^{{\bf n}}({\mathcal N}_s)_{\ge 0}$, the reaction vector
$P'-P$ with source $P\in esupp^{{\bf n}}(\mathcal N_s)$ satisfies $(P'-P) \cdot  {\bf n}\geq 0.$ 
It is therefore enough to show that
\begin{equation}\label{eq:reduced}
\Bigg(K_{s,P_{s,\bf n}\to P'_{s,\bf n}}(t,{\bf x}) (P'_{s,\bf n}-P_{\bf
    n})+\sum_{\substack{{\{P\to P'\in\mathcal R_{s,\bf n}}|\\ {P\notin
        esupp^{{\bf n}}(\mathcal N_s)\}}}}K_{s,P\to P'}(t,{\bf x}) (P'-P)\Bigg)\cdot {\bf n} \geq 0
\end{equation}
in order to verify (\ref{eq:newMainIneq}). In turn,
(\ref{eq:reduced}) will follow from Lemma \ref{lem:staysInside} with 
$\mathcal U = \mathcal{SC(N}_s)\cap esupp^{{\bf n}}(\mathcal R_s)_{>0}$ 
and the fact that 
\begin{equation}\label{eq:yaineq}
K_{P_{s,\bf n}\to P'_{s,\bf n}}(t,{\bf x}) > \mu K_{P\to P'}(t,{\bf x})
\end{equation}
for all $P\to P'\in{\mathcal R}_s$ with $P\in \mathcal U$ (recall
$\mu$ from (\ref{eq:delta})). Therefore showing inequality (\ref{eq:yaineq}) 
will complete the proof of the theorem.

As noted in Remark \ref{rem:AfromMu}, (\ref{eq:yaineq}) is implied by 
\begin{equation}\label{eq:ineq}
\Lambda^A_{\alpha,\beta}({\bf x}) > 0,    
\end{equation}
where $(\alpha,-\beta)=D(P_{s,\bf n}-P)$ and $A=\mu/\eta^2.$
To verify (\ref{eq:ineq}), we consider different cases,
according on the location of ${\bf x}=(x,y)$ on $\mathcal L.$ First suppose that 
${\bf x}$ lies on the line segment $[A_iA_{i+1}]$ for some
$i\in\{1,\ldots, e\}.$ Then (see Figure \ref{fig:bigPicture}), if
$\tilde r_{i'}<\tilde r_i<\tilde r_{i''}$, or
equivalently, if $i'<i<i''$, we have 
\begin{equation}\label{eq:order}
(1/\delta)x^{\tau(\tilde r_{i''})} < y < \delta x^{\tau(\tilde r_{i'})}.
\end{equation}
Depending on the sign combination of $\alpha$ and $\beta,$ the source
complex $P$ may belong to one of the three shaded regions in Figure \ref{fig:relPos}(a).

\begin{figure}
\begin{center}
\begin{tabular}{cc}
\begin{tikzpicture}[>=stealth, scale =.24]
\scriptsize
\begin{scope}
\clip (0,0) rectangle+(16,16);

\begin{scope}[shift = {(.5,-.18)}] 
\clip (7,6)--(16,6)--(16,0)--(14,0);
\shadedraw[inner color = blue!40!white,outer color=white,draw=white] (-9,-8) rectangle +(30,30);
\end{scope}

\begin{scope}[shift = {(-.2,.45)}] 
\clip (7,6)--(0,12)--(0,16)--(7,16);
\shadedraw[inner color = blue!35!white,outer color=white,draw=white] (-9,-8) rectangle +(30,30);
\end{scope}

\begin{scope}
\clip (7,6) rectangle + (10,10);
\shadedraw[inner color = blue!25!white,outer color=white,draw=white] (-9,-8) rectangle +(30,30);
\end{scope}

\node[rotate = -40] at (2.6,8.7) {$esupp^{\bf n}(\mathcal{N}_s)$};
\node at (6.5,5.5) {$P_{s,\bf n}$};
\fill[black] (7,6) circle (.12);
\node at (15.5,.5) {$I$};
\node at (15.5,15.5) {$II$};
\node at (6,15.5) {$III$};

\draw(0,12)--(14,0);
\draw(7,16)--(7,6)--(16,6);
\draw[gray](0,0)--(0,16)--(16,16)--(16,0)--(0,0); 
\draw[gray](0,0)--(0,16)--(16,16)--(16,0)--(0,0); 
\end{scope}
\normalsize
\node at (-2,15) {$(a)$};
\end{tikzpicture} 
&\hspace{2.5cm}
\begin{tikzpicture}[>=stealth, scale =.24]

\begin{scope}
\clip (0,0) rectangle+(16,16);
\scriptsize
\begin{scope}[shift = {(.2,0)}] 
\clip (5,8) rectangle + (11,8);
\shadedraw[inner color = blue!40!white,outer color=white,draw=white] (-9,-8) rectangle +(30,30);
\end{scope}

\begin{scope}[shift = {(.2,-.2)}] 
\clip (5,0) rectangle + (11,8);
\shadedraw[inner color = blue!35!white,outer color=white,draw=white] (-9,-8) rectangle +(30,30);
\end{scope}

\node[rotate = -90] at (4.2,12.7) {$esupp^{\bf n}(\mathcal{N}_s)$};
\node at (4.2,8) {$P_{s,\bf i}$};
\fill[black] (5,8) circle (.12);

\node at (15.5,.5) {$I$};
\node at (15.5,15.5) {$II$};

\draw(5,0)--(5,16);
\draw(5,8)--(16,8);
\draw[gray](0,0)--(0,16)--(16,16)--(16,0)--(0,0); 
\draw[gray](0,0)--(0,16)--(16,16)--(16,0)--(0,0); 
\end{scope}
\normalsize
\node at (-2,15) {$(b)$};
\end{tikzpicture} 

\end{tabular}
\end{center}
\caption{Positions of a source complex $P$ relative to $P_{s,\bf n}$}\label{fig:relPos}
\end{figure}
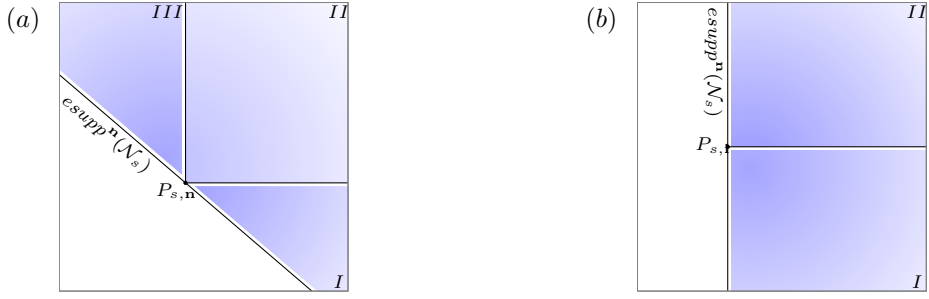

\noindent {\em Region I.} Here $\alpha<0$ and $\beta<0$; (\ref{eq:ineq}) is equivalent to $y>y^A_{\alpha,\beta}(x).$ 
The set 
$$\{r^s_j\mid j\in\{1,\ldots, e(s)\} \text{ with } r^s_j>\tilde r_i\}$$ 
is nonempty, since $\alpha/\beta$ is one of its elements. If
$r^s_{j_0}$ denotes the minimum of this set, then $r^s_{j_0}\le \alpha/\beta$ and it follows that 
$$y > (1/\delta) x^{\tau(r^s_{j_0})} > (3/2)C^A_{\alpha,\beta}x^{\tau(\alpha/\beta)} > y^A_{\alpha,\beta}(x).$$
\noindent The first inequality is implied by (\ref{eq:order}), since $r_{j_0}^s>\tilde{r}_i$. The
second inequality holds because of (P1), because $x<\xi<1$ 
and because $\tau$ is increasing. The last inequality
is a consequence of  (P4). Note that the calculation above corresponds
to $\Delta \neq 0$, but the same argument works if $\Delta = 0$ with
$C^A_{\alpha,\beta}$ replaced by
 $\gamma^A_{\alpha}=\displaystyle\max_{i}\gamma^A_{\alpha,i}.$

\noindent {\em Region II.} We have  $\alpha \le 0$ and $\beta \ge 0.$
From (P3) we know that ${\bf x}\in [0,\xi]^2$ is below the curve
$\Lambda^A_{\alpha,\beta}=0,$ and so
$\Lambda^A_{\alpha,\beta}({\bf x})>0.$

\noindent {\em Region III.} This case is similar to Region $I.$

Finally, suppose that ${\bf x}$ lies on one of the two unbounded sides
of $\mathcal L$, for instance, on the vertical side. Then $\alpha < 0$ and there are two regions for $P$, according to the sign of $\beta$
(see Figure \ref{fig:relPos}(b)). For region $I$ $(\beta<0)$, the
proof is the same as in the case of region $I$ from Figure
\ref{fig:relPos}(a). For region $II$ we have $\beta \ge 0$ and $\alpha
< 0$ and, since $y\le M,$ we have  
$$\Lambda^A_{\alpha,\beta}({\bf x})=\Lambda^A_{\alpha,\beta}({d,y})\ge \Lambda^A_{\alpha,\beta}(d,M)>0$$
from (P6).
\end{proof}

The following corollary follows from
Theorem \ref{thm:proj}  using Nagumo's Theorem \ref{thm:nagumo} 
and implies that bounded trajectories of
2D-reduced mass-action systems are persistent. 
\begin{cor}\label{cor:proj}
Let $\bigcup_{s=1}^p(\mathcal N_s,\kappa_s,\Psi, a_s)$ be a 2D-reduced 
mass-action system 
where $\mathcal N_s=(\mathcal S,\mathcal C_s,\mathcal R_s)$ are
lower-endotactic networks for all $s\in\{1,\ldots, p\}$, 
and suppose that
$\kappa_s(t)\in(\eta,1/\eta)^{\mathcal R_s}$ for all $s\in\{1,\ldots, p\}$
and all $t\ge 0.$
Let $T(c_0)$ be a trajectory of $(\mathcal N,\kappa)$ and let $M>1$ be
such that $T(c_0)\subset [0,M]^2$. 
If ${\mathcal L}^+={\mathcal L}^+(\{\mathcal N_s\}_{1\le s\le p}, \Psi,
\eta, M)$ is constructed such that $c_0\in \mathcal L^+$, then
$T(c_0)\subset \mathcal L^+.$
\end{cor}

\begin{remark}\label{rem:Lbar}
Corollary \ref{cor:proj} remains valid if instead of $c_0\in \mathcal
L^+$ we have $c_0\in {\overline{\mathcal L}}^+=  
{\overline{\mathcal L}}^+(\{\mathcal N_s\}_{1\le s\le p},\Psi, \eta).$
In that case we conclude  that
$T(c_0)\subset \overline{\mathcal L}^+.$
\end{remark}

We conclude this section with the following result, which will be
useful in section \ref{sec:gac}.

\begin{cor}\label{cor:x+y}
Let $\bigcup_{s=1}^p(\mathcal N_s,\kappa_s,\Psi, a_s)$ be a
2D-reduced mass-action system 
where $\mathcal N_s=(\mathcal S,\mathcal C_s,\mathcal R_s)$ are
lower-endotactic networks for all $s\in\{1,\ldots, p\}$ 
and suppose that
$\kappa_s(t)\in(\eta,1/\eta)^{\mathcal R_s}$ for all $s\in\{1,\ldots, p\}$
and all $t\ge 0.$
Then there exist $\epsilon>0$ and $\tau>0$ such that if $c_0=(x_0,y_0)\in
[0,\epsilon]^2$ then $x+y\ge \tau (x_0+y_0)$ for all $(x,y)\in T(c_0).$
\end{cor}
\begin{proof}
Let 
${\overline{\mathcal L}}={\overline{\mathcal L}}(\{\mathcal N_s\}_{1\le s\le p},\Psi, \eta)$
and let 
$\epsilon>0$ be such that 
$[0,\epsilon]^2\cap {\overline{\mathcal L}}(\{\mathcal N_s\}_{1\le s\le p}, \Psi, \eta)=\O.$
Once ${\overline{\mathcal L}}$ is constructed, we can shift it as close to the
origin as desired. 
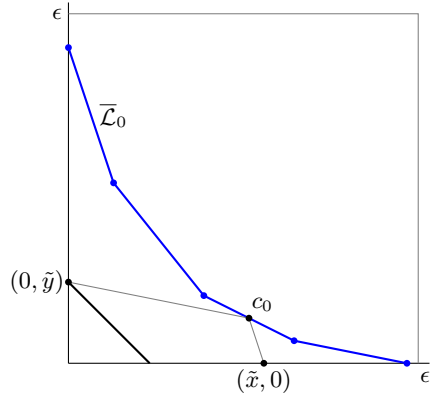
\begin{figure}[h]
\begin{center}
\begin{tikzpicture}[thick, >=stealth, scale = .3] 
\draw[thin](0,0) -- (16,0); 
\draw[thin](0,0) -- (0,16); 

\draw[blue](15,0)--(10,1)--(6,3)--(2,8)--(0,14);
\fill[blue] (15,0) circle (.15);
\fill[blue] (10,1) circle (.15);
\fill[blue] (6,3) circle (.15);
\fill[blue] (2,8) circle (.15);
\fill[blue] (0,14) circle (.15);

\draw[thin,gray](8.66,0)--(8,2)--(0,3.6);
\draw(0,3.6)--(3.6,0);

\draw[thin,gray](0,15.5)--(15.5,15.5)--(15.5,0);

\fill (8,2) circle (.15);
\fill (8.66,0) circle (.15);
\fill (0,3.6) circle (.15);
\small
\node at (8.6, 2.6){$c_0$};
\node at (8.66, -0.8){$(\tilde x, 0)$};
\node at (-1.4,3.6){$(0,\tilde y)$};
\node at (2, 11) {$\overline{\mathcal L}_0$};
\node at (-0.5, 15.5){$\epsilon$};
\node at (15.8,-0.6){$\epsilon$};
\end{tikzpicture} 
\caption{Illustration for Corollary \ref{cor:x+y}}\label{fig:x+y}
\end{center}
\end{figure}
In particular,
if $c_0\in [0,\epsilon]^2$, we may construct 
$\overline{\mathcal L}_0=\overline{\mathcal L}_0(\{\mathcal N_s\}_{1\le s\le p},
\Psi, \eta)$ such that $c_0\in \overline{\mathcal L}_0.$
Corollary \ref{cor:proj} and Remark \ref{rem:Lbar} imply that $T(c_0)$
lies in the unbounded part of  $\mathbb R_{\ge 0}^2$ delimited by
$\overline{\mathcal L}_0.$ We draw lines through $c_0$ that are parallel
to the extreme line segments of $\overline{\mathcal L}_0$ and denote their
intersection with the coordinate axes by $(0,\tilde y)$ and $(\tilde x,0)$
(see Figure \ref{fig:x+y}). Then $T(c_0)$ lies above the line
$x+y=\min\{\tilde x, \tilde y\}.$ 
Using the notation from (\ref{eq:ris}) we have $\tilde x=x_0+r_1y_0$
and $\tilde y = x_0/r_e +y_0$ and the conclusion follows by choosing 
$\tau=\min\{1,r_1,1/r_e\}.$ 
\end{proof}

\end{subsection}

\end{section}

\begin{section}{Persistence of $\kappa$-variable mass-action systems 
with two-dimensional stoichiometric subspace}\label{sec:pers}

The main persistence result of this paper is the following. 
\begin{thm}\label{thm:pers}
In any $\kappa$-variable mass-action system with two-dimensional
stoichiometric subspace and lower-endotactic stoichiometric subnetworks, 
all bounded trajectories are persistent.
\end{thm}

A little additional terminology and a couple of lemmas are needed to arrive at a proof
of Theorem \ref{thm:pers}. 
For the remainder of this section, we fix a $\kappa$-variable mass-action
system $(\mathcal N,\kappa)=(\mathcal S, \mathcal C, \mathcal R,
\kappa)$ having $n$ species, stoichiometric subspace
$S$ of dimension two, and lower-endotactic stoichiometric subnetworks
$\mathcal N_s=(\mathcal S, \mathcal C_s, \mathcal R_s), \ s\in\{1,\ldots, p\}.$ 
We assume that $\kappa(t)\in (\eta,1/\eta)^{\mathcal R}.$ We also let 
$T(c_0)=\{c(t)=(x_1(t),\ldots, x_n(t)) \mid t\ge 0, c(0)=c_0\}$ be a
bounded trajectory of $(\mathcal N, \kappa)$ such that
$c_0\in\mathbb R^n_{>0}$ and $T(c_0)\subset [0,M]^n$ for some $M>1.$

\begin{subsection}{Preliminary setup}\label{subsec:prelim}
For $W\subseteq \{1,\ldots, n\}$ we let 
$$Z_W=\{(x_1,\ldots, x_n)\in\mathbb R^n|x_i=0 \text{ for all }i\in W\}.$$
The relative boundary of the polyhedron
$S(c_0)$ is included in $\partial\mathbb R^n_{\ge 0}$ and we may
identify a face of $S(c_0)$ by the minimal face of $\mathbb
R^n_{\ge 0}$ that contains it. More precisely, if a face of
$S(c_0)$ is included in $Z_{W}$ and  $W\subset\{1,\ldots,
n\}$ is maximal with this property, then we denote 
that face by $F_W.$ Note that if $W^1\subset W,$ then $F_W\subset
F_{W^1}$, and $F_W= F_{W^1}$ if and only if $W=W^1.$

Decreasing $n$ if necessary, we can assume that 
$S$ intersects the open positive orthant,  
$S\cap \mathbb R_{>0}^n\neq \O.$ 
Indeed, if this is not true, then some coordinates $x_i,$ $i\in
\complement V\subset \{1,\ldots, n\}$ of $c(t)$ are constant and
may be disregarded, by replacing $(\mathcal N, \kappa)$ with its
projection  $(\mathcal N_V, K_V)$ onto
$V.$ Note that the properties of $\mathcal N$ 
are inherited by $\mathcal N_V$: the stoichiometric
subspace of  $\mathcal N_V$ has the same dimension as $S$; the
projected kinetics is $\kappa$-variable mass-action (see
(\ref{eq:projjj})) and $\mathcal N_V$ has lower-endotactic stoichiometric 
subnetworks from Proposition \ref{prop:transform};
finally, $\pi_V(T(c_0))\in [0,M]^V.$

Each vertex $F_W$ of $S(c_0)$ has two adjacent edges, which we will henceforth
denote $F_{W^1}$ and $F_{W^2}.$ We have $W^1\cup W^2 \subseteq W$ and also 
$W^1\cap W^2= \O,$ for otherwise $S$ would not intersect the interior
of the positive orthant. Let ${\bf v}^1(W)=(p_1,\ldots, p_n)$ and  
${\bf v}^2(W)=(q_1,\ldots, q_n)$ be vectors (unique up to positive
scalar multiplication) along $F_{W^1}$ and
$F_{W^2}$ respectively, such that the cone with
vertex at $F_W$ generated by $F_{W^1}$ and $F_{W^2}$ contains $S(c_0).$  
Then, for any $i\in\{1,\ldots, n\},$ not both $p_i$ and $q_i$ can be
zero, for otherwise, again, $S\subset \partial\mathbb R^n_{\ge 0}.$ 

\begin{remark}\label{rem:signpq}
We have
\begin{eqnarray}\label{eq:signpq}
&p_i=0\text{ for }i\in W^1, p_i\neq 0\text{ for }i\in\complement W^1\text{ and } p_i > 0\text{ for }i\in W\backslash W^1,\\  
\nonumber
&q_i=0\text{ for }i\in W^2, q_i\neq 0\text{ for }i\in\complement W^2\text{ and } q_i > 0\text{ for }i\in W\backslash W^2.
\end{eqnarray}
Indeed, if $p_i=0$ for some $i\in \complement W^1,$ then
$F_{W^1}=F_{W^1\cup\{i\}}$, contradiction.
On the other hand, for any ${\bf x}=(x_1,\ldots, x_n)\in F_{W^1}\backslash
F_W,$ we have ${\bf x}=F_W+a_{\bf x}(p_1,\ldots, p_n)$ for some $a_{\bf x}> 0,$
and therefore $x_i=a_{\bf x}p_i$ for all $i\in W.$ Since $x_i\in F_{W^1}$ it follows that $p_i=0$ for $i\in W^1.$ Moreover, since $x_i>0$
for $i\in W\backslash W^1$, we have $p_i>0$ for $i\in W\backslash W^1.$
The explanation for $q_i,\ i\in\{1,\ldots, n\}$ is similar.
\end{remark}

\begin{remark}\label{rem:lsmk}
If $(k,l)\in W^1\times W^2,$ then we may rescale ${\bf v}^1(W)$ and
${\bf v}^2(W)$ such that $p_l=q_k=1.$ Moreover, by swapping 
${\bf v}^1(W)$ with ${\bf v}^2(W)$ if necessary,
we may also assume that $l<k.$ 
Let 
\begin{equation}\label{eq:Theta}
\Theta = ({\bf v}^1(W)\quad {\bf v}^2(W))
\end{equation}
be the matrix with columns ${\bf v}^1(W)$ and ${\bf v}^2(W).$
Since ${\bf v}^1(W)$ and ${\bf v}^2(W)$
generate $S,$ we have
\begin{equation}\label{eq:Thetapi}
{\bf x}=\Theta \pi_{k,l}({\bf x}) \text{ for any }{\bf x}\in S
\end{equation}
(here, if ${\bf x} = (x_1,\ldots, x_n),$ then $\pi_{k,l}({\bf x})=(x_l,x_k)$). 
In particular, it follows that
\begin{equation}\label{eq:ThetapiSc0}
{\bf x}-F_W=\Theta \pi_{k,l}({\bf x})\text{ for all } {\bf x}\in S(c_0).
\end{equation} 
\end{remark}

\begin{remark}\label{rem:a_s}
Let $(k,l)\in W^1\times W^2$ and let
$\mathcal N_s=(\mathcal S, \mathcal C_s, \mathcal R_s), \
s\in\{1,\ldots, p\}$ denote the stoichiometric subnetworks of $\cal N$. 
Then, for each $s\in\{1,\ldots, p\},$ we may choose 
$a_s\in\mathbb R^n$ such that $\mathcal C_s\subset S+a_s$ and
$\pi_{k,l}(a_s)=(0,0)$ (recall that $\pi_{k,l}$ denotes the projection
onto $\{k,l\}$).
Indeed, suppose $\mathcal C_s\subset S+(\gamma_1,\ldots,\gamma_n)$ for some
$(\gamma_1,\ldots,\gamma_n)\in\mathbb R^n.$ 
Then we look for $a_s$ such that 
$$a_s-(\gamma_1,\ldots,\gamma_n)\in S=\mathrm{span}\{{\bf v}^1(W), {\bf v}^2(W)\}$$ and
$\pi_{k,l}(a_s)=(0,0).$
Assuming $p_l=q_k=1$ as in Remark \ref{rem:lsmk}, we define
$$a_s=(\gamma_1,\ldots,\gamma_n)-\gamma_l{\bf v}^1(W)-\gamma_k{\bf v}^2(W).$$
\end{remark}

\begin{example}\label{ex:vertices}
It might be helpful at this point to illustrate the notations introduced thus far in this
section by revisiting the network in Example \ref{eq:exNet}. If we let $c_0=(1, 1, 1, 1),$
it is not hard to see that $S(c_0)$ is a square with vertices at $(0,0,2,2),\ (0,2,0,2)\ (2,0,2,0)$
and $(0,0,2,2)$. Let $W=\{3,4\}$ and consider the vertex $F_W=(2,2,0,0)$ of $S(c_0).$ For any 
$(x_1,x_2,x_3,x_4)\in S(c_0)$ we may write 
$$(x_1,x_2,x_3,x_4) - (2,2,0,0) = x_3(0,-1,1,0)+x_4(-1, 0, 0, 1)$$
and therefore ${\bf v}^1(W)=(0,-1,1,0)$,  ${\bf v}^2(W)=(-1,0,0,1)$ and
$W^1=\{4\}$, $W^2=\{3\}.$ The face $F_{W^1}$ of $S(c_0)$ is
$\{(2,2-x,x,0)\in\mathbb R^4|x\in [0,2]\}$ and is parallel to ${\bf v}^1(W).$
As for the vectors $a_s$ from Remark \ref{rem:a_s}
corresponding to our two stoichiometric subnetworks, we have 
$a_1=(2,2,0,0)$ and $a_2=(2,0,0,0).$  
\end{example}

\begin{remark}\label{rem:injective}
Since ${\bf v}^1(W)$ and ${\bf v}^2(W)$ generate $S$,
it follows that 
\begin{equation}\label{eq:Theta2}
{\bf x}=\Theta\pi_{k,l}({\bf x}) \text{ for any }{\bf x}\in S.
\end{equation} 
In particular, we have $\ker \pi_{k,l}\cap S =0$ and therefore
$\pi_{k,l}$ is injective on $S.$ It follows that $\pi_{k,l}$ is
also injective on $a_s+S.$ 

If $Q$ is a column vector in $\mathbb R^{\{k,l\}}\simeq \mathbb R^2$
then $\Theta Q\in S$ and we have 
$$\pi_{k,l}(a_s+\Theta Q)=\pi_{k,l}(a_s)+\pi_{k,l}(\Theta Q)=Q.$$
Since, as explained above, $\pi_{k,l}$ is injective on $a_s+S,$ we
conclude that 
\begin{equation}\label{eq:preImage}
\pi_{k,l}^{-1}(Q)\cap (a_s+S)=\{a_s+\Theta Q\}.
\end{equation} 
\end{remark}

\begin{subsection}{A glimpse into the rest of section \ref{sec:pers}} \label{sec:glimpse}
Before we dive into the technical arguments that lead to a proof of Theorem \ref{thm:pers}, 
it is worth considering a couple of examples. The aim is to illustrate how the 
machinery of projections and 2D-reduced mass-action systems comes into place 
and to hint at the idea behind the proof of Theorem \ref{thm:pers}.

Let us first consider the following $\kappa$-variable mass-action system:

\begin{equation}\label{ex:netProof1}
A+B\mathop{\rightleftharpoons}^{\kappa_1}_{\kappa_2}2C\quad A+C\mathop{\rightleftharpoons}^{\kappa_3}_{\kappa_4}A.
\end{equation}

We denote ${\bf x}=(c_A,c_B,c_C)$ the concentration vector and we write
the corresponding differential equations in the form

\begin{equation}\label{eq:exProof1}
\left(
\begin{matrix}
\dot c_A\\
\dot c_B\\
\dot c_C\\
\end{matrix}
\right)=
\kappa_1(t){\bf x}^{A+B}
\left(
\begin{matrix}
-1\\
-1\\
2\\
\end{matrix}
\right)+
\kappa_2(t){\bf x}^{2C}
\left(
\begin{matrix}
1\\
1\\
-2\\
\end{matrix}
\right)+
\kappa_3(t){\bf x}^{A+C}
\left(
\begin{matrix}
0\\
0\\
-1\\
\end{matrix}
\right)+
\kappa_4(t){\bf x}^{A}
\left(
\begin{matrix}
0\\
0\\
1\\
\end{matrix}
\right).
\end{equation}

Note that the stoichiometric subspace of (\ref{ex:netProof1}) is $S=\{(x_1,x_2,x_3)\in\mathbb R^3\ |x_1=x_2\}$ 
and it intersects the positive orthant $\mathbb R^3_{>0}.$ Let then $c_0\in S\cap \mathbb R^3_{>0}$ be a
positive initial condition for (\ref{eq:exProof1}) such that $T(c_0)$ is bounded. If we let 
$$\Psi=
\left(
\begin{matrix}
1 &0\\
1 &0\\
0 &-1
\end{matrix}
\right)
$$ 
then for ${\bf x}=(c_A,c_B,c_C)\in T(c_0)$ and denoting ${\bf y} = (c_A,c_C)$ we have
\begin{equation}\label{eq:exPsi1}
{\bf x}=
\Psi{\bf y}.
\end{equation} 
Moreover, note that
$$A+B = 
\left(
\begin{matrix}
1\\
1\\
0
\end{matrix}
\right)=
\Psi
\left(
\begin{matrix}
1\\
0
\end{matrix}
\right)=
\Psi(A),
$$ 
where, by a useful abuse of notation, the argument $A$ of $\Psi$ is viewed as the vector 
of two coordinates $(1, 0)$ in the plane $AC$. One can write similar equalities to obtain
\begin{equation}\label{eq:PsiComp}
A+B = \Psi(A),\ 2C=\Psi(2C),\ 
A =\Psi(A)+(0, -1, 0)^t,\ A+C = \Psi(A+C)+(0, -1, 0)^t;
\end{equation}
\noindent denote $a_1=(0,0,0)$ and $a_2=(0,-1,0).$

Projecting (\ref{ex:netProof1}) onto $\{A,C\}$ yields the reaction network 
$$A\rightleftharpoons 2C\quad A+C\rightleftharpoons A$$
whose dynamics is obtained by substituting (\ref{eq:exPsi1}) and (\ref{eq:PsiComp}) into
(\ref{eq:exProof1}):  

\begin{eqnarray}\label{eq:exProj1}
\dot {\bf y}=
\kappa_1(t)(\Psi{\bf y})^{\Psi(A)+a_1}
\left(
\begin{matrix}
-1\\
2\\
\end{matrix}
\right)+
\kappa_2(t)(\Psi{\bf y})^{\Psi(2C)+a_1}
\left(
\begin{matrix}
1\\
-2\\
\end{matrix}
\right)
\\
+\kappa_3(t)(\Psi{\bf y})^{\Psi(A+C)+a_2}
\left(
\begin{matrix}
0\\
-1\\
\end{matrix}
\right)+
\kappa_4(t)(\Psi{\bf y})^{\Psi(A)+a_2}
\left(
\begin{matrix}
0\\
1\\
\end{matrix}
\right).
\end{eqnarray}

This kinetics corresponds to the 2D-reduced mass-action system 
$({\cal N}_1,\Psi,\kappa',a_1)\cup ({\cal N}_2,\Psi,\kappa'',a_2),$
where 
$${\cal N}_1=\{A\rightleftharpoons 2C\},\  
{\cal N}_2=\{A+C\rightleftharpoons A\},$$
and 
$$\kappa'_{A\to 2C}(t)=\kappa_1(t),\ 
\kappa'_{2C\to A}(t)=\kappa_2(t),\
\kappa''_{A+C\to A}(t)=\kappa_3(t),\
\kappa''_{A\to A+C}(t)=\kappa_4(t).
$$
Corollary $\ref{cor:proj}$ implies that the trajectory of ${\bf y}$
is persistent; from (\ref{eq:PsiComp}) it follows that $T(c_0)$ is 
persistent as well.  

The argument above can be written in the general case without much additional 
effort; this is done in Proposition \ref{prop:prin0}. Note that, although
it illustrates very well the use of projections and 
2D-reduced mass-action, the example discussed above is rather special by
insisting that $S(c_0)$ contain the origin. To see what issues might
arise if this is not the case, let us next revisit the system 
(\ref{ex:MA}), which we assume to be $\kappa$-variable mass-action.   
Choose $c_0$ in the same stoichiometric compatibility class as $(1,1,1,1).$  
Since  $S(c_0)$ is bounded, so is $T(c_0).$ 
Note that for any ${\bf x}=(c_A,c_B,c_C,c_D)\in T(c_0)$ we have $c_A+c_D=c_B+c_C=2.$

We will now  give an heuristic explanation of the fact 
that $c_A$ cannot become too small (the same reasoning
may be applied to the rest of the concentrations). 
We aim, as in the previous example, to project our system onto a 2D face of 
$\mathbb R^4_{\ge 0}$ and realize the projected 
dynamics as 2D-reduced mass-action, in order to conclude that $c_A(t)$ stays bounded
away from zero. Let us consider the projection onto $\{A,B\}.$  
As  illustrated in Example \ref{ex:projNet}, the projected network 
\begin{equation}\label{eq:reducedNet1}
B\mathop{\rightleftharpoons}^{\overline\kappa_1}_{\overline\kappa_2} A\mathop{\to}^{\overline\kappa_3} 
A+B\mathop{\to}^{\overline\kappa_4} 0\quad 
2A\mathop{\to}^{\overline\kappa_5} A\mathop{\gets}^{\overline\kappa_6} 0
\end{equation}
\noindent is lower-endotactic and the projected dynamics can be written in the form
\begin{eqnarray}\label{eq:projForm1}
\left(
\begin{matrix}
\dot c_A\\
\dot c_B
\end{matrix}
\right)
&=&
\overline{\kappa}_1(t) c_B
\left(
\begin{matrix}
1\\
-1
\end{matrix}
\right)+
\overline\kappa_2(t) c_A
\left(
\begin{matrix}
-1\\
1
\end{matrix}
\right)
+\overline\kappa_3(t) c_A
\left(
\begin{matrix}
0\\
1
\end{matrix}
\right)
\\
&+&
\overline\kappa_4(t) c_Ac_B
\left(
\begin{matrix}
-1\\
-1
\end{matrix}
\right)
+
\overline\kappa_5(t) c_A^2
\left(
\begin{matrix}
-1\\
0
\end{matrix}
\right)+
\overline\kappa_6(t)
\left(
\begin{matrix}
1\\\nonumber
0
\end{matrix}
\right)
\end{eqnarray}
\noindent where
\begin{eqnarray}\label{ex:kappas}\nonumber
&\overline\kappa_1(t)=\kappa_1(t)(2-c_A(t)),\quad
\overline\kappa_2(t)=\kappa_2(t)(2-c_B(t)),\quad
\overline\kappa_3(t)=\kappa_3(t)(2-c_B(t)),\\
&\overline\kappa_4(t)=\kappa_4(t),\quad
\overline\kappa_5(t)=\kappa_5(t),\quad
\overline\kappa_6(t)=\kappa_6(t)(2-c_A(t))^2. 
\end{eqnarray} 

This is a 2D-reduced mass-action system (in fact, this would
be $\kappa$-variable mass-action system, provided we knew that    
$\overline\kappa_i$ are bounded away from zero).
Since $T(c_0)$ is bounded, Theorem \ref{thm:proj} implies  that there 
exists a set ${\cal L}^+_{A,B}\subset\mathbb R^{\{A,B\}}_{>0}$ as in
section \ref{sec:constr} such that the projection of $c_0$ onto 
$\{A,B\}$ lies in 
${\cal L}^+_{A,B}$ and whenever the phase point 
$(c_A(t_0),c_B(t_0))$ of (\ref{eq:projForm1}) is on the boundary of 
${\cal L}^+_{A,B}$, the vector field points inside ${\cal L}^+_{A,B}$; 
this, provided $\overline\kappa_i(t_0)$ 
belongs to a certain interval away from zero and infinity. By inspecting the rates 
$\overline\kappa_i$ in (\ref{ex:kappas}),
we see that this condition is equivalent to saying that $c_A(t_0)$ and $c_B(t_0)$ are not 
too close to 2 (recall that $\kappa_i$ are bounded away from zero 
and infinity, since (\ref{ex:MA}) is $\kappa$-variable mass-action system). 
The case when $c_A(t_0)$ is very close to 2 is not of interest to us as we want 
to show that $c_A$ cannot become too small, and therefore we
look what happens when $c_A(t)$ is close to zero. 

Now we refer back to  Figure \ref{fig:bigPicture}. 
The set ${\cal L}^+_{A,B}$ is a positive translation of the
positive quadrant located at a small distance $d$ from each of the
axes, and with a cut at the corner near the origin. 
To illustrate the point, let us make a gross oversimplification and assume
that ${\cal L}^+_{A,B}$ is a square. (Note, however, that it is not, 
and, although the cut near the origin can be made arbitrarily small, 
it still requires a delicate analysis). With this simplification in place, 
we argue that $c_A$ cannot become smaller than $d.$ Indeed, if at time $t=t_0,$
the trajectory $(c_A,c_B)$ reaches the boundary of ${\cal L}^+_{A,B}$ and $c_A=d,$ then,
as explained above, if  $c_B(t_0)$ is not too close to 2, the trajectory is pushed 
inside ${\cal L}^+_{A,B}$ and $c_A$ increases.
 
On the other hand,  Theorem \ref{thm:proj} does not apply 
for the projected system (\ref{eq:reducedNet1}) if $c_B(t_0)$ is close to 2. However,
in this case $c_C(t_0)=2-c_B(t_0)$ is small and we may project onto $\{A,C\}$ instead.
The projected reaction system 

$$0\mathop{\rightleftharpoons}^{\overline\kappa_1}_{\overline\kappa_2} A+C\mathop{\to}^{\overline\kappa_3} 
A\mathop{\to}^{\overline\kappa_4} C\quad 
2A\mathop{\to}^{\overline\kappa_5} A\mathop{\gets}^{\overline\kappa_6} 0,$$
has rate constant functions  
\begin{eqnarray}\nonumber
&\overline\kappa_1(t)=\kappa_1(t)(2-c_A(t)(2-c_C(t)),\quad
\overline\kappa_2(t)=\kappa_2(t),\quad
\overline\kappa_3(t)=\kappa_3(t),\\\nonumber
&\overline\kappa_4(t)=\kappa_4(t)(2-c_C(t)),\quad
\overline\kappa_5(t)=\kappa_5(t),\quad
\overline\kappa_6(t)=\kappa_6(t)(2-c_A(t))^2 
\end{eqnarray} 
which are all bounded away from zero at $t=t_0.$ If 
${\cal L}^+_{A,C}\subset \mathbb R^{\{A,C\}}$ 
is constructed in the same way as ${\cal L}^+_{A,B}$ and at the same distance $d$ 
from the coordinate axes, then, since $c_A(t_0)=d,$ we have 
$(c_A(t_0),c_C(t_0))\in \partial {\cal L}^+_{A,C}$ and Theorem \ref{thm:proj}
implies that the vector field at $t=t_0$ points inside ${\cal L}^+_{A,C}.$ Once again,
$c_A$ must increase. 

One may recast the discussion above by using a symmetric construction
of an invariant set $\cal T^+$. Namely, one  considers the
cylinder ${\cal L}_{A,B}^+\times \mathbb R^{\{C,D\}}$ and the similar 
cylinders coming from all possible projections to pairs of variables. Defining
$\cal T^+$ to be their intersection, the reasoning above translates into 
$\cal T^+$ being an invariant set for  $T(c_0)$. This is precisely
what we do in the proof of Theorem
\ref{thm:pers}. 

Note, however, that while the previous exposition
sheds some light on the basic idea of the proof, (presented in
section \ref{sec:proofMain}), the technical
details involved are subtle and require an extensive preparation,
which is the object of section \ref{subsec:further}.
In particular, the parameters required  in the construction of the
sets $\cal L^+$ need to take into 
account the geometry of $S(c_0)$ and must be chosen carefully. Lemmas
\ref{lem:fz} and \ref{lem:d} are part of this process. 
\end{subsection}

\begin{subsection}{Further preparation}\label{subsec:further} 
As anticipated in the discussion above, a special case of Theorem 
\ref{thm:pers} follows in a more or less straightforward way from 
Corollary \ref{cor:proj}:

\begin{prop}\label{prop:prin0}
Let $(\cal N,\kappa) = (\cal S,\cal  C ,\cal R, \kappa)$ be a
$\kappa$-variable mass-action system with two-dimensional stoichiometric subspace
and lower-endotactic stoichiometric subnetworks.
If the stoichiometric compatibility class $S(c_0)$ contains the
origin, then $T(c_0)$ is a persistent trajectory.
\end{prop}
\begin{proof}
We denote $W=\{1,\ldots, n\},$ so that the origin is the vertex $F_W$ of
$S(c_0).$ Let $(k,l)$ be a fixed pair in $W^1\times W^2,$ 
and let ${\bf v}^1(W)=(p_1,\ldots, p_n)$ and ${\bf
  v}^2(W)=(q_1,\ldots, q_n).$ Note that all $p_i$ and $q_i$ may be
chosen to be rational
because $S$ is generated by vectors of integer coordinates.
As explained in Remark \ref{rem:lsmk}, we may assume that $k<l$ and $p_l=q_k=1;$ 
note that we also have 
$p_i\ge 0,\ q_i\ge 0$ for all $i\in\{1,\ldots, n\}$
from (\ref{eq:signpq}).

Let $\mathcal N_s=(\mathcal S, \mathcal C_s, \mathcal R_s), \
s\in\{1,\ldots, p\}$ denote the stoichiometric subnetworks of $\cal N$. 
If $\Psi$ denotes the matrix with columns ${\bf v}^1(W)$ and ${\bf
  v}^2(W),$ then $\Psi$ has the form (\ref{eq:psi}). 
As explained in Remark \ref{rem:a_s},  
for each $s\in\{1,\ldots, p\}$ we may choose 
$a_s\in\mathbb R^n$ such that $\mathcal C_s\subset S+a_s$ and
$\pi_{k,l}(a_s)=(0,0).$
Since $c(t)\in S$ for all $t\ge 0$ we have
\begin{equation}\label{eq:prinzero}
c(t)=\Psi \pi_{k,l}(c(t))
\end{equation}
by (\ref{eq:Thetapi}) Remark \ref{rem:lsmk}. 
We have
$$\frac{d}{dt} c(t)=\sum_{s=1}^p\sum_{P\to P'\in {\mathcal
    R}_s}\kappa_{P\to P'}c(t)^P(P'-P)$$
and since, as implied by (\ref{eq:preImage}) Remark \ref{rem:injective},
the only reaction in ${\cal R}_s$ that projects onto $Q\to Q'$ via $\pi_{k,l}$
is $\Psi Q+a_s\to \Psi Q'+a_s,$ it follows that  
$$\frac{d}{dt} \pi_{k,l}(c(t))=\sum_{s=1}^p\sum_{Q\to Q'\in
  \pi_{k,l}({\mathcal R}_s)}
\overline{\kappa}_{s,Q\to Q'}(t)(\Psi\pi_{k,l}c(t))^{\Psi Q+a_s}(Q'-Q),$$
where we denoted
$$\overline{\kappa}_{s, Q\to Q'}(t)=\kappa_{\Psi Q+a_s\to \Psi
  Q'+a_s}(t)$$
for all $Q\to Q'\in\pi_{k,l}({\cal R}_s)$.
Therefore $\pi_{k,l}(T(c_0))$ is the trajectory of the 2D-reduced mass-action system 
$\bigcup_{s=1}^p(\pi_{k,l}(\mathcal N_s), \Psi, \overline{\kappa}_s, a_s)$
with initial condition $\pi_{k,l}(c_0)$.

Since all $\pi_{k,l}(\mathcal N_s)$ are lower-endotactic by Proposition
\ref{prop:transform}, Corollary \ref{cor:proj} implies that
coordinates $x_l(t)$ and $x_k(t)$ of $c(t)$  
stay larger than some $d>0$ for all $t\ge 0.$ 
Finally, (\ref{eq:prinzero}) implies that all the coordinates of
$c(t)$ remain bounded away from zero.
\end{proof}

The special case of Theorem \ref{thm:pers} contained in Proposition
  \ref{prop:prin0} illustrates  well how projected systems, 2D-reduced 
mass-action systems  and Theorem \ref{thm:proj} come into play.  
As anticipated in section \ref{sec:glimpse}, the general case requires
yet a little more preparation, which we discuss next. 

Fix a bounded trajectory  $T(c_0)$ of $(\cal N, \kappa)$ and  let
$M>1$ be such that $T(c_0)\subset [0,M]^n.$
As hinted in section \ref{sec:glimpse}, we construct a polyhedron
$\mathcal T^+\subset \mathbb R^n_{> 0}$ that stays away from
$\partial\mathbb R^n_{\ge 0}$ and such that
$T(c_0)\subset \mathcal T^+;$ in the process we use the tools we have
developed thus far. Namely, we project $(\mathcal N, \kappa)$ onto
well-chosen sets of two variables, we cast the projected system as a
2D-reduced mass-action system and we construct a corresponding $\mathcal L^+$ set in each such
two-dimensional face of $\mathbb R^n_{\ge 0}.$ Finally, we
construct certain cylinders out of the $\mathcal L^+$ sets and define $\mathcal T^+$
as the intersection of these cylinders. 

The projections to consider are of the form $\pi_{k,l}=\pi_{\{k,l\}}$
with $(k,l)\in W^1\times W^2$ for all vertices $F_W$ of $S(c_0).$

We fix a vertex $F_W$ and a pair $(k,l)\in W^1\times W^2$.
Based on Remark \ref{rem:lsmk}, if ${\bf v}^1(W)=(p_1,\ldots, p_n)$ 
and ${\bf v}^2(W)=(q_1,\ldots, q_n)$ then we may
assume that $l<k$ and that $p_l=q_k=1.$ 
Note that if $(k,l)\in W^1\times W^2$ then there is no other vertex  $F_V$ of $S(c_0)$ 
such that $(k,l)\in V^1\times V^2.$
Indeed, if $F_V=(f_1,\ldots, f_n)$ then we have
$F_V-F_W=f_l{\bf v}^1(W)+f_k{\bf v}^2(W).$
Since $k,l\in V_1\cup V_2\subseteq V,$ we have $f_l=f_k=0$ and so $F_V=F_W.$

Recall that  ${\cal N}_s=({\cal S}, {\cal C}_s, {\cal R}_s)$ denote the
stoichiometric subnetworks of $\cal N$ and $S$ denotes the
stoichiometric subspace of $\cal N.$ As in Remark  \ref{rem:a_s}, for
each $s\in\{1,\ldots, p\},$ let $a_s\in\mathbb R^n$ such that ${\cal
C}_s\subset a_s+S.$
Let 
$$\Theta = ({\bf v}^1(W)\quad {\bf v}^1(W))$$
be the matrix with columns ${\bf v}^1(W)$ and ${\bf v}^2(W)$
and define
\begin{equation}\label{eq:psikl}
\Psi_{k,l}=\pi_W\Theta =
(\pi_W({\bf v}^2(W)) \quad \pi_W({\bf v}^2(W)))
\end{equation}
to be the matrix with columns $\pi_W({\bf v}^1(W))$ and  $\pi_W({\bf v}^2(W)).$
Note that  $p_i$ and $q_i$ are non-negative for all $i\in W$ and
moreover, they are rational numbers
since the stoichiometric subspace $S$ of $\cal N$ is generated by vectors of integer coordinates. Therefore
$\Psi_{k,l}$ is of the form (\ref{eq:psi}).

Since, by Proposition
\ref{prop:transform}, $\pi_{k,l}(\mathcal N_s)$ is lower-endotactic 
for all $s\in\{1,\ldots, p\},$ we may construct the set 
\begin{equation}\label{eq:lkl}
\mathcal L^+_{k,l}=\mathcal L^+_{k,l}(\{\pi_{k,l}(\mathcal N_s)\}_{1\le s\le p}, \Psi_{k,l},
\eta', M)\subset Z_{\complement \{k,l\}}
\end{equation}
as in section \ref{sec:constr} such that $\pi_{k,l}(c_0)\in\mathcal L^+_{k,l}.$
We will choose $\eta'$ in what follows;
also, we will take advantage of the flexibility in
the construction of $\mathcal L^+_{k,l}$ to equip this set with a few
useful technical properties. 

We start with two lemmas which show the intuitively clear facts that
if a point of $S(c_0)$ is close to $Z_W$ then it is also close to $F_W,$
and that if some components of a point in $S(c_0)$ are small, then the 
point is close to a face where all those components are zero. 

\begin{lemma}\label{lem:fz}
Let $F_W$ be a face of $S(c_0).$ Then there exists  $\delta>0$ such that 
$\mathrm{dist}({\bf x}, F_W)\le \delta \mathrm{dist}({\bf x}, Z_W)$
for all ${\bf x}\in S(c_0).$
\end{lemma}
\begin{proof}
If $S(c_0)$ has dimension one, we denote by $\alpha$ the angle between $S(c_0)$
and $Z_W.$ Since $S(c_0)$ intersects the positive orthant we have
$\alpha>0$ and therefore $0<\sec\alpha<\infty.$ Then
$\mathrm{dist}({\bf x}, F_W)= \delta \mathrm{dist}({\bf x},
Z_W)$ for all ${\bf x}\in S(c_0),$ where $\delta = \sec\alpha.$ If
$S(c_0)$ has dimension two, for ${\bf x}\in S(c_0)$ we let ${\bf p}_F$ and ${\bf p}_Z$
denote the projections of $\bf x$ on
$F_W$ and $Z_W.$ Let $\alpha=\inf_{{\bf x}\in S(c_0)} \measuredangle ({\bf x}{\bf p}_F {\bf
  p}_Z).$ If $\alpha=0,$ then, since $S(c_0)$ is closed, there
exists ${\bf x}_0\in S(c_0)\backslash F_W$ such that 
$\measuredangle ({\bf x}_0{\bf p}_F {\bf p}_Z)=0.$ In turn, this 
implies that ${\bf p}_F={\bf p}_Z$ and therefore $S(c_0)$ contains the
line segment  ${\bf x}_0{\bf p}_F$ which is
perpendicular to $Z_W.$ It follows 
that  $\mathrm{dist}({\bf x}, F_W)=\mathrm{dist}({\bf x}, Z_W)$ for
all ${\bf x}\in S(c_0).$ If $\alpha>0$
then we let $\delta = \sec\alpha.$ For ${\bf x}\in S(c_0)$ we have 
$\mathrm{dist}({\bf x}, F_W)=\sec\measuredangle ({\bf x}{\bf p}_F
{\bf p}_Z)\mathrm{dist}({\bf x}, Z_W)\le\delta\mathrm{dist}({\bf x}, Z_W)$ for all ${\bf x}\in S(c_0).$  
\end{proof}

\begin{lemma}\label{lem:d}
There exists $\lambda>0$ such that if $I\subset \{1,\ldots,n\}$
and ${\bf d}=(d_1,\ldots,d_n)\in S(c_0)$ is such that $d_i<\lambda$ for $i\in
I,$ then for some face $F_W$ of $S(c_0)$ we have $I\subset W.$
\end{lemma}
\begin{proof}
If the origin is a face of $S(c_0)$ the claim in the lemma is clearly
true (any positive value for $\lambda$ will do).
Otherwise, for $J\subseteq \{1,\ldots,n\}$ we define $m(J)=\inf_{{\bf x}\in
  S(c_0)}\mathrm{dist}({\bf x}, Z_J)$ and we let 
$\lambda=\min\{m(J)\mid J\subseteq\{1,\ldots,n\},m(J)>0\}/\sqrt n$. 
We have $\lambda>0$ and 
$$m(I)^2\leq \mathrm{dist}({\bf d},Z_I)^2=\sum_{i\in I} d_i^2<n\lambda^2,$$
which shows that $m(I)=0$ and the conclusion follows. 
\end{proof}

In view of Proposition \ref{prop:prin0} we may assume that the origin
is not a vertex of $S(c_0)$. Let $v^{min}$ denote the smallest nonzero
coordinate of a vertex of $S(c_0)$ and, fixing a $\lambda$ given by Lemma
\ref{lem:d}, let 
\begin{equation}\label{eq:zeta}
\zeta = \min\{\lambda, v^{min}/2, 1\}.
\end{equation}
Moreover, let ${\bf 1}=(1,\ldots, 1)\in\mathbb R^n, 
\ E = \displaystyle\max_{P\in\mathcal{SC(N)}}(P\cdot {\bf 1})$ and define
\begin{equation}\label{eq:eta'}
\eta' =\min\left\{
\eta \zeta^E, \frac{\eta}{M^E}
\right\}.
\end{equation}
Recall from section \ref{sec:constr} that the construction of a set $\mathcal
L^+$ depends on the numbers $\xi$ and $d.$
Also recall that $\xi$ may be chosen arbitrarily small; once $\xi$
is fixed, $d$ may also be made small independently of the value of
$\xi.$ Since there are finitely many pairs 
$(k,l)\in W^1\times W^2$ (counting all vertices $F_W$ of $S(c_0)$) 
we can choose the same values of $\xi$ and $d$ 
in the construction of all sets $\mathcal L^+_{k,l}.$ We fix $\xi$ small
enough such that 
\begin{equation}\label{eq:xiIneq}
\max\Bigg\{
\bigcup_{\substack{{\{F_W \text{ vertex of }S(c_0)}\\{{\bf v}^1(W)=(p_1,\ldots, p_n)}\\{{\bf v}^2(W)=(q_1,\ldots, q_n)\}}}}\{|p_i/p_j|\mid i,j\in \complement W^1\}\cup
\{|q_i/q_j|\mid i,j\in \complement W^2\}
\Bigg\} \xi< \frac{v^{min}}{4}.
\end{equation}
As can be seen from Figure \ref{fig:bigPicture}, the shape of $\mathcal L_{k,l}$ near $(0,0)$ enables us to choose $\epsilon > 0$ such that  
\begin{equation}\label{eq:epsilon}
[0,\epsilon]^2\cap\mathcal L_{k,l}^+=\O \text{ for all vertices
}F_W\text{ of }S(c_0)\text{ and all } (k,l)\in
W^1\times W^2.
\end{equation}
We now choose $d$ such that 
\begin{equation}\label{eq:dIneq}
d< \min\Bigg\{\lambda,\zeta, \epsilon
\bigcup_{\substack{{\{F_W \text{ vertex of }S(c_0)}\\{{\bf v}^1(W)=(p_1,\ldots, p_n)}\\{{\bf v}^2(W)=(q_1,\ldots, q_n)\}}}}(\{p_i/p_j\mid i,j\in W\backslash W^1\}\cup
\{q_i/q_j\mid i,j\in W\backslash W^2\})
\Bigg\},
\end{equation}
\noindent where we recall that $\lambda$ is given by Lemma \ref{lem:d}, $\zeta$ is
defined in (\ref{eq:zeta}) and $\epsilon$ was chosen to satisfy (\ref{eq:epsilon}).

We shift $\mathcal L_{k,l}$ (and $\mathcal L^+_{k,l}$) along $\mathbb
R^{\complement \{k,l\}}$ and define 
$$\mathcal H_{k,l}=\{{\bf x}\in \mathbb R^n \mid (x_k,x_l)\in\mathcal
L_{k,l}\}\text{ and }\mathcal H^+_{k,l}=\{{\bf x}\in \mathbb R^n \mid (x_k,x_l)\in\mathcal
L^+_{k,l}\}.$$
Note that $\mathcal H_{k,l}^+=\mathrm{conv}(\mathcal H_{k,l}).$ Finally, let 
\begin{equation}\label{eq:defL+}
\mathcal T^+ = (d{\bf 1}+\mathbb R_{\geq 0}^n)
\cap\bigcap_{F_W\text{ vertex of }S(c_0) }\bigcap_{(k,l)\in W^1\times W^2}
\mathcal H_{k,l}^+.
\end{equation} 
By definition, the convex polyhedron $\mathcal T^+ $ lies in a positive translation
of the nonnegative orthant.
In view of Theorem \ref{thm:nagumo}, 
we shall be concerned with the behavior of the
flow $\dot c(t)$ on the boundary of $\mathcal T^+$; we conclude the
preparatory discussion with the following lemma, which shows that the
part of $\partial \cal T^+$ that is of interest to us does not include
the boundary of $(d{\bf 1}+\mathbb R_{\ge 0}^n)$, but only the
boundaries of ${\cal H}_{k,l}^+$, for $(k,l)\in W^1\times W^2.$

\begin{lemma}\label{lem:2}
$$\partial\mathcal T^+\cap S(c_0)\subset\bigcup_{F_W\text{ vertex of
  }S(c_0) } \bigcup_{(k,l)\in W^1\times W^2} (\mathcal H_{k,l}\cap S(c_0)).$$
\end{lemma}
\begin{proof}
We have 
$$\partial\mathcal T^+\subset \bigcup_{i=1}^n (d{\bf 1}+Z_{\{i\}}) \cup\bigcup_{F_W\text{ vertex of }S(c_0)}\bigcup_{(k,l)\in W^1\times W^2} \mathcal H_{k,l}.$$
Suppose ${\bf x}\in\partial\mathcal T^+\cap S(c_0)$ and ${\bf x}\in
d{\bf 1}+Z_{\{i\}}$. Since
$d<\lambda$ (see (\ref{eq:dIneq})), Lemma \ref{lem:d} implies that
there exists a face $F_W$ such that $i\in W.$ Possibly making $W$ larger,  
we can assume that  $F_W$ is a vertex of $S(c_0).$ 
Now we show that $i\in W^1\cup W^2;$  
suppose this was false and let $(k,l)\in W^1\times W^2.$ Assume that
$l<k$ (otherwise swap $W^1$ with $W^2$) and let ${\bf
  v}^1(W)=(p_1,\ldots, p_n)$ and ${\bf v}^2(W)=(q_1,\ldots, q_n)$
where $p_l=q_k=1.$ Since $i\in W\backslash( W^1\cup
W^2),$ both $p_i$ and $q_i$ are strictly positive, as explained in
Remark \ref{rem:signpq}. Using (\ref{eq:dIneq})
we obtain
$$\max\{x_lp_i, x_kq_i\} < x_lp_i+x_kq_i= x_i = d \le
\min\{(q_i/q_k)\epsilon,(p_i/p_l)\epsilon\}
= \min\{q_i\epsilon,p_i\epsilon\}$$
and so $(x_l, x_k)\in[0,\epsilon]^2.$ This, together with (\ref{eq:epsilon}) implies
that ${{\bf x}\notin \mathcal H_{k,l}^+}$ and therefore $\bf x\notin \mathcal T^+$, contradiction.
We conclude that $i\in W^1\cup W^2.$ Suppose $i\in W^1,$ and let $j\in W^2.$ We have 
$${\bf x}\in \mathcal H_{i,j}^+\cap (d{\bf 1}+Z_{\{i\}})\subset  \mathcal
H_{i,j}.$$   
\end{proof}
\end{subsection}

\begin{subsection}{Putting things together} \label{sec:proofMain}
We are ready to prove our main persistence result, Theorem
\ref{thm:pers}. We keep using the notations introduced thus far in
this section; in particular, if $F_W$ is a vertex of
$S(c_0)$ and $(k,l)\in W^1\times W^2$ we assume that $l<k,$ we
denote ${\bf v}^1(W)=(p_1,\ldots, p_n), {\bf v}^2(W)=(q_1,\ldots,
q_n)$ and we set $p_l=q_k=1.$

\smallskip

{\em Proof of Theorem \ref{thm:pers}}.
For any vertex $F_W$ of $S(c_0)$ and for any pair $(k,l)\in W^1\times W^2$
we have by construction $c_0\in\mathcal H^+_{k,l},$
and therefore $c_0\in \mathcal T^+.$ We will use Theorem
\ref{thm:nagumo} to show that
$T(c_0)\subset \mathcal T^+.$
Suppose $t\ge 0$ is such
that ${\bf x}=c(t)\in \partial\mathcal T^+;$ it is enough to show that 
\begin{equation}\label{eq:todo}
{\bf n}\cdot \dot c(t)\ge 0 
\end{equation}
for all ${\bf n}\in -N_{\mathcal T^+}(\bf x).$
Lemma \ref{lem:2} implies the existence of a vertex 
$F_W=(f_1,\ldots, f_n)$ of $S(c_0)$ and the existence of a pair $(k,l)\in W^1\times W^2$ such
that ${\bf x}\in  \mathcal H_{k,l}.$ The generators of the convex cone $-N_{\mathcal T^+}(\bf x)$ lie in the union
of $-N_{\mathcal H^+_{k,l}}(\bf x)$  for all pairs $(k,l)$ such that ${\bf x}\in
\mathcal H_{k,l},$ thus (\ref{eq:todo}) needs only be verified for vectors
$\bf n$ belonging to this union. Therefore we fix such a pair $(k,l)$ and we
let ${\bf n}\in -N_{\mathcal H^+_{k,l}}(\bf x)$. Since only the coordinates
$k$ and $l$ of ${\bf n}$ are nonzero, inequality  (\ref{eq:todo}) is equivalent to 
\begin{equation}\label{eq:toProve2}
\pi_{k,l}({\bf n})\cdot\pi_{k,l}(\dot c(t))\ge 0.
\end{equation}
According to Remark \ref{rem:a_s}, for each $s\in\{1,\ldots, p\}$ we may choose 
$a_s\in\mathbb R^n$ such that $\mathcal C_s\subset S+a_s$ and
$\pi_{k,l}(a_s)=(0,0).$

{\em Case (i).} Suppose $\pi_{k,l}({\bf x})$ lies on the
finite part of $\mathcal L_{k,l}.$ 
For $c\in T(c_0)$ we have $c-F_W\in S$ and we may write, using
Remark \ref{rem:lsmk} and the definition (\ref{eq:psikl}) of
$\Psi_{k,l}:$
$$\pi_W(c)=\pi_W(c-F_W)=\pi_W(\Theta \pi_{k,l}({c-F_W}))
=\Psi_{k,l} \pi_{k,l}(c-F_W)=\Psi_{k,l} \pi_{k,l}(c).$$

Therefore 
$\pi_W(c)=\Psi_{k,l}\pi_{k,l}(c)$ for any $c\in T(c_0).$ 
On the other hand, according to Remark \ref{rem:injective},
the only reaction in ${\cal R}_s$ mapped by $\pi_{k,l}$
 to $Q\to Q'\in \pi_{k,l}({\cal R}_s)$ is $\Theta Q+a_s\to \Theta Q'+a_s.$
It follows that  the dynamics of $(\mathcal N, \kappa)$ projected onto
$\{k,l\}$ may be written
\begin{eqnarray}\label{eq:projected}
&&\frac{d}{dt}\pi_{k,l}(c)=\sum_{s=1}^p\sum_{P\to P'\in\mathcal
  R_s}\kappa_{P\to P'}(t) \pi_{\complement
  W}(c)^{\pi_{\complement W} (P)}\pi_W(c)^{\pi_W(P)}
(\pi_{k,l}(P')-\pi_{k,l}(P))\nonumber
\\\nonumber
&=&
\sum_{s=1}^p\sum_{Q\to Q'\in\pi_{k,l}(\mathcal
  R_s)} \kappa_{\Theta Q+a_s\to \Theta Q'+a_s}(t) \pi_{\complement
  W}(c)^{\pi_{\complement W} 
(\Theta Q+a_s)}
\pi_W(c)^{\pi_W(\Theta Q+a_s)} (Q'-Q)\\
&=&
\sum_{s=1}^p\sum_{Q\to Q'\in\pi_{k,l}(\mathcal
  R_s)} \overline{\kappa}_{s,Q\to Q'}(t) 
(\Psi_{k,l}\pi_{k,l}(c))^{\Psi_{k,l}Q+a_s} (Q'-Q)
\end{eqnarray}
where 
\begin{equation}\label{eq:kbar}
\overline{\kappa}_{s, Q\to Q'}(t)
=\kappa_{\Theta Q+a_s\to \Theta Q'+a_s}(t) \pi_{\complement
  W}(c(t))^{\pi_{\complement W}(\Theta Q+a_s)}
\text{ for all } Q\to Q'\in\pi_{k,l} ({\cal R}_s).
\end{equation}

Therefore the system of differential equations (\ref{eq:projected}) is the
2D-reduced mass-action system 
$\bigcup_{s=1}^p(\pi_{k,l}(\mathcal N_s),\Psi_{k,l}, \overline \kappa_s, a_s).$ 

For all $i\in\complement W$ we have (recall (\ref{eq:ThetapiSc0})):
$$x_i = f_i+x_l p_i+x_k q_i \ge f_i-|p_i|x_l-|q_i|x_k
=f_i-|p_i/p_l|x_l-|q_i/q_k|x_k.$$
Moreover, since $\pi_{k,l}(\bf x)$ lies on the finite part of $\mathcal L_{k,l}$ we have
$x_k,x_l<\xi.$ This, together with (\ref{eq:xiIneq}) implies, for any
$i\in\complement W:$ 
$$f_i-|p_i/p_l|x_l-|q_i/q_k|x_k\ge
f_i-(|p_i/p_l|+|q_i/q_k|)\xi\ge 
f_i-v^{min}/2 \ge v^{min}/2.$$

Therefore, for all $i\in\complement W$ we have 
$$x_i\ge  f_i-v^{min}/2 \ge v^{min}/2$$
(recall that $v^{min}$ denotes the minimum of nonzero coordinates of $F_W$).
It then follows from (\ref{eq:zeta}) that $x_i\ge \zeta$ for any
$i\in\complement W.$ 
This yields 
$$M^{(\Theta Q+a_s)\cdot \bf 1}\ge \pi_{\complement
  W}({\bf x})^{\pi_{\complement W} (\Theta Q+a_s)}\ge \zeta^{(\Theta Q+a_s)\cdot \bf
  1}
\text{ for any } Q\to Q'\in\pi_{k,l}({\cal R}_s);$$
recalling (\ref{eq:kbar}) and using  (\ref{eq:eta'}), we then have
$\overline{\kappa}_{s,Q\to Q'}(t)\in (\eta',1/\eta')$.
Since $\pi_{k,l}({\bf n})\in -N_{\mathcal L^+_{k,l}}(\pi_{k,l}(\bf x))$
and $\mathcal L^+_{k,l}=\mathcal L^+_{k,l}(\{\pi_{k,l}(\mathcal N_s)\}_{1\le s\le p}, \Psi_{k,l},
\eta', M),$ Theorem \ref{thm:proj}
implies (\ref{eq:toProve2}).

{\em Case (ii).} Now suppose that $\pi_{k,l}({\bf x})$ lies
on the infinite part of  $\mathcal L_{\{k,l\}},$ for instance $x_k = d.$
Let $I=\{i\in\{1,\ldots, n\}\mid x_i<\zeta\}$ and note that, by
(\ref{eq:dIneq}) we have $k\in I.$ Since, by (\ref{eq:zeta}) we have
$\zeta<\lambda,$ 
Lemma \ref{lem:2} implies that there exists a face 
$F_{\overline W}$ of $S(c_0)$ with $I\subseteq {\overline W}.$ We may assume that 
$F_{\overline W}$ is a vertex. We claim that $k\in \overline
W^1\cup \overline W^2;$ indeed, otherwise, let
$(\bar k, \bar l)\in  \overline W^1\times \overline W^2$ 
and ${\bf v}^1(\overline W)=(p_1\ldots p_n),$ ${\bf v}^2(\overline
W)=(q_1\ldots q_n)$ such that $p_{\bar l}=q_{\bar k}=1$. 
Since, by (\ref{eq:epsilon}), at least one of $x_{\bar l}$ and
$x_{\bar k}$ is larger than $\epsilon,$
in view of  (\ref{eq:dIneq}) we have the following contradiction:
$$d = x_k = x_{\bar l}p_k+ x_{\bar k}q_k>
\epsilon\min\{p_k,q_k\}>d.$$

Therefore $k\in \overline W^1\cup \overline W^2;$ 
suppose $k\in \overline W^1$ and let $\overline l\in \overline W^2.$
Since for each $i\in\complement W$ we have $x_i\ge\zeta$ (this from our definition of $I$),  the same
argument as in case $(i)$ shows that $\pi_{k,\overline l}({\bf
  n})\cdot\pi_{k,\overline l}(\dot c(t))\ge 0.$ On the other hand, since
only the $k$-th coordinate of $\bf n$ is nonzero, we have
$$\pi_{k,l}({\bf n})\cdot\pi_{k,l}(\dot c(t))=\pi_{k,\overline l}({\bf n})\cdot\pi_{k,\overline l}(\dot c(t))$$
and (\ref{eq:toProve2}) is shown.
\end{subsection}

Recall that weakly reversible reaction networks are endotactic and in
particular lower-endotactic. The following corollary
states that the version of the Persistence Conjecture
proposed in \cite{anderson:oneLC} holds for systems with two-dimensional
stoichiometric subspace.

\begin{cor}
Any bounded trajectory of a weakly reversible  $\kappa$-variable
mass-action system with two-dimensional stoichiometric subspace is persistent.
\end{cor} 

\begin{example}
To conclude this section let us revisit the $\kappa$-variable mass-action
system (\ref{ex:MA}). We know that its stoichiometric subspace is
two-dimensional (Example \ref{ex:stoich}) and that its stoichiometric
subnetworks coincide with its linkage classes (Example \ref{ex:canon}). 
If $L_1$ denotes the first linkage class then the projection 
$\pi_{1,2}:\mathrm{aff}(L_1)\to Z_{\{3,4\}}$ is invertible. Since 
$\pi_{1,2}(L_1)=\{B\mathop\rightleftharpoons^{}_{} A
\mathop{\to}^{} 
A+B\mathop\to^{} 0\}$ is easily seen to be endotactic, according to
Proposition \ref{prop:transform}, the same is
true for $L^1=\pi_{1,2}^{-1}(\pi_{1,2}(L_1)).$ Similarly, the second
linkage class is endotactic. Since $c_A+c_D$ and $c_B+c_C$ remain
constant along trajectories, any trajectory is bounded.  
Therefore, Theorem \ref{thm:pers} implies that the dynamical system
(\ref{eq:exMA}) is persistent.
\end{example}
\end{subsection}
\end{section}

\begin{section}{The Global Attractor Conjecture for systems with 
three-dimensional stoichiometric subspace}\label{sec:gac} 
Recall from Introduction that in order to show the Global Attractor Conjecture it is
enough to prove that all trajectories of complex-balanced
mass-action  systems are persistent.
Theorem
  \ref{thm:pers} may be used to analyze the behavior of trajectories of weakly
  reversible mass-action systems near faces of $S(c_0)$ of codimension
  two. As we shall see below, trajectories can approach such a face
only if they approach its boundary. This is  made
precise in Theorem \ref{thm:repel}. 
On the other hand, as discussed in Introduction, 
vertices of $S(c_0)$ cannot be $\omega$-limit points for
trajectories of complex-balanced systems
\cite{anderson:primul, craciun_dickenstein_shiu_sturmfels}.
Moreover, codimension-one faces of  $S(c_0)$ are repelling
\cite{anderson_shiu} and we have the following result:

\begin{thm}[\cite{anderson_shiu} Corollary 3.3]\label{thm:anderson}
Let $c_0\in\mathbb R^n_{>0}$ and let $T(c_0)$ denote a bounded
trajectory of a weakly reversible complex-balanced mass-action system.
Also let $F_W$ be a codimension-one face of $S(c_0).$ If $T(c_0)$ does not
have $\omega$-limit points on the (relative) boundary of $F_W,$ then  
it does not have $\omega$-limit points on $F_W.$
\end{thm} 

For weakly reversible, complex-balanced systems with three-dimensional 
stoichiometric compatibility classes, the
results mentioned above cover all faces of $S(c_0)$ and can be
combined into a proof of the Global Attractor Conjecture for this case.
We start with the following lemma.

\begin{lemma}\label{lem:repel}
Let $(\mathcal S,\mathcal C,\mathcal R,\kappa)$ be a weakly-reversible $\kappa$-variable
mass-action system,
let $c_0\in\mathbb R^n_{>0}$ and let $F_W$ be a face
of $S(c_0)$ of codimension two. 
Then for any compact $K\subset \mathrm{int}(F_W)$ there exist 
$\tau>0$ and $\epsilon>0$ such that if for some $t',t''\in \mathbb
R_{>0}$ we have
$c(t)=(x_1(t),\ldots, x_n(t))\in \pi_{\complement W}(K)\times
[0,\epsilon]^W$ for all
$t\in[t',t'']$, then $\sum_{i\in W}x_i(t'')\ge \tau \sum_{i\in W}x_i(t').$
\end{lemma}
\begin{proof}
We denote $\mathcal N=(\mathcal S,\mathcal C,\mathcal R).$ 
Let $S$ denote the stoichiometric subspace of $\mathcal N,$ let $d=\dim
S,$ and let  $\pi_{W}|_S:S\to Z_{\complement W}$ denote the restriction of $\pi_{W}$ to $S.$
Since $F_W$ is of codimension two, we have $\dim (S\cap Z_{W})=d-2.$ 
But $S\cap Z_{W}=\ker (\pi_{W}|_S)$ and therefore
$\pi_{W}(S)=\mathrm{Im}(\pi_W|_S)$ has
dimension two. Note that $\pi_{W}(S)$ is the stoichiometric subspace
of $\tilde{\mathcal N}=\pi_W(\mathcal N).$ Since $\mathcal N$ is weakly reversible, so is
$\tilde{\mathcal N};$ in particular, $\tilde{\mathcal N}$ has two-dimensional stoichiometric subspace and
lower-endotactic subnetworks, which we denote by $\tilde{\mathcal
  N}_s=(\mathcal S, \tilde{\mathcal C}_s, \tilde{\mathcal R}_s)$ for  
$s\in\{1,\ldots, p\}$. 

We may now apply the results of the preceding sections. 
The face $\tilde F_{W}$ of the stoichiometric compatibility class
$\tilde S(\pi_W(c_0))$ of $\tilde{\mathcal N}$ is 
the origin of $\mathbb R^W.$ 
We let $\tilde{F}_{W^1}$ and $\tilde{F}_{W^2}$ denote
the two edges of $\tilde S(\pi_W(c_0))$ that are adjacent to $\tilde
F_{W}$ (recall that $W^1, W^2$ are subsets of $\{1,\ldots, n\}$ which
are contained in $W$). Let $(k,l)\in W^1\times W^2.$ We may assume $l<k$ and we scale 
the direction vectors ${\tilde{\bf v}}_1(W)$ and ${\tilde{\bf v}}_2(W)$ of 
$\tilde{F}_{W^1}$ and $\tilde{F}_{W^2}$ 
such that the $l$th coordinate of ${\tilde{\bf v}}_1(W)$
and the $k$th coordinate of ${\tilde{\bf v}}_2(W)$ are equal to 1.
Also, for each $s\in\{1,\ldots, p\},$ we choose
$a_s\in\mathbb R^W$ such that $\tilde{\mathcal C}_s\subset a_s+\tilde S$ and 
$\pi_{k,l}(a_s)=(0,0)$ (this is possible as explained in Remark \ref{rem:a_s}).

Suppose that $\kappa(t)\in (\eta,1/\eta)^{\mathcal R}$ for all $t\ge 0$ 
and let $\Psi_{k,l}$ be the matrix with columns
$\tilde{\bf v}_1(W)$ and $\tilde{\bf v}_2(W).$ We have

\begin{eqnarray*}
&&\frac{d}{dt}\pi_{k,l}(c)=
\sum_{Q\to Q'\in\pi_{k,l}(\mathcal R)}
\sum_{\substack{{\{P\to P'\in{\cal R} |}\\{\pi_{k,l}(P\to P')=Q\to Q'\}}}}
\kappa_{P\to P'}(t) c^P
(Q'-Q)
\\\nonumber
&=&
\sum_{s=1}^p\sum_{Q\to Q'\in\pi_{k,l}(\mathcal
  R_s)} 
\sum_{\substack{{\{P\to P'\in\cal R|}\\{\pi_W(P\to
      P')=\Psi_{k,l}Q+a_s\to \Psi_{k,l}Q'+a_s\}}}}
\kappa_{P\to P'}(t) c^P (Q'-Q),
\end{eqnarray*}
since, as explained in Remark \ref{rem:injective}, the only reaction of
${\cal R}_s$ that is mapped to $Q\to Q'$ by $\pi_{k,l}$ is 
$\Psi_{k,l}Q+a_s\to \Psi_{k,l}'Q+a_s.$ Further, note that
$\pi_W(c)=\Psi_{k,l}\pi_{k,l}(c)$ for any $c\in T(c_0);$ 
therefore, writing $c^P=\pi_W(c)^{\pi_W(P)}\pi_{\complement
  W}(c)^{\pi_{\complement W}(P)}$ we have 
$$\frac{d}{dt}\pi_{k,l}(c)=
\sum_{s=1}^p\sum_{Q\to Q'\in\pi_{k,l}(\mathcal R_s)} 
\overline{\kappa}_{s,Q\to Q'}(t)(\Psi_{k,l}\pi_{k,l}(c))^{\Psi_{k,l}Q+a_s}
(Q'-Q),
$$
where
$$\overline{\kappa}_{s, Q\to Q'}(t)=
\sum_{\substack{{\{P\to P'\in\cal R|}\\{\pi_W(P\to
      P')=\Psi_{k,l}Q+a_s\to \Psi_{k,l}Q'+a_s\}}}}\kappa_{P\to P'}(t) \pi_{\complement
  W}c(t)^{\pi_{\complement W}(P)}$$
for all $t\ge 0.$
Since $K\subset \mathrm{int}(F_W)$ there exists
$0<\zeta<1$ such that $\pi_{\complement W}(K)\subset (\zeta,
1/\zeta)^{\complement W}.$ 
It follows that there exists $\eta'<1$ such that $\bar\kappa_{s,Q\to
  Q'}(t)\in (\eta',1/\eta')$ for all $t\in[t',t''].$
The projection of $(\mathcal N,\kappa)$ onto
$\{k,l\}$ is a 2D-reduced mass-action system
$\bigcup_{s=1}^p(\pi_{k,l}(\tilde{\mathcal N}_s), \Psi_{k,l}, \overline{\kappa}_s,a_s).$ 
Since $k,l\in W,$ Corollary \ref{cor:x+y} implies that there exists 
$\epsilon>0$ and $\tau'>0$ such that  
if $c(t)\in \pi_{\complement W}(K)\times [0,\epsilon]^W$ for
$t\in[t',t''],$ then 
$x_k(t'')+x_l(t'')\ge \tau'(x_k(t')+x_l(t')).$
Since $\pi_W(c(t))=\Psi_{k,l}\pi_{k,l}(c(t))$ we have 
$\sum_{i\in W} x_i(t) = Ax_k(t)+Bx_l(t)$ for some positive numbers $A$
and $B$ and for all $t\ge 0.$ Therefore 
$$\sum_{i\in W} x_i(t'')\ge
\min\{A,B\}(x_k(t'')+x_l(t''))
\ge \tau'\min\{A,B\}(x_k(t')+x_l(t'))
\ge \frac{\tau'\min\{A,B\}}{\max\{A,B\}} \sum_{i\in W} x_i(t')
$$
and the conclusion follows by setting $\tau=\frac{\tau'\min\{A,B\}}{\max\{A,B\}}.$
\end{proof}

We may now extend the result of Corollary 3.3 in \cite{anderson_shiu} 
(Theorem \ref{thm:anderson}) to faces of codimension two.
\begin{thm}\label{thm:repel}
Let $(\mathcal S,\mathcal C,\mathcal R,\kappa)$ be a weakly-reversible $\kappa$-variable
mass-action system, let $c_0\in\mathbb R^n_{> 0}$ 
such that the forward trajectory $T(c_0)$ is bounded, and let $F_W$ be a face
of $S(c_0)$ of codimension two. Then, if $T(c_0)$ has
$\omega$-limit points on $F_W,$ 
it must also have $\omega$-limit points on the relative boundary of $F_W.$
\end{thm}
\begin{proof}
Suppose the claim of the theorem was false. Then $\lim_{\omega}
T(c_0)\cap F_W\subset \mathrm{int}(F_W)$ is compact 
(it is an intersection of closed sets and is
bounded since $T(c_0)$ is bounded). It follows that there exists a compact set
$K\subset \mathrm{int}(F_W)$ such that 
$\lim_{\omega}T(c_0)\cap F_W\subset \mathrm{int}(K).$
From Lemma \ref{lem:repel}, there exist $\tau>0$ and
$\epsilon>0$ such that if  $c(t)=(x_1(t),\ldots, x_n(t))\in \pi_{\complement W}(K)\times
[0,\epsilon]^W$ for all $t\in [t',t'']$ then $\sum_{i\in
  W}x_i(t'')\ge \tau\sum_{i\in W}x_i(t').$ It follows that, 
in order to approach an $\omega$-limit point in $K,$ $T(c_0)$ must
exit and reenter $K_\epsilon=\pi_{\complement W}(K)\times [0,\epsilon]^W$ infinitely
often. More precisely, there exist $t_1<t_2<\ldots$ with
$t_m\to\infty$ as $m\to \infty$ such that  $c(t_m)\in \partial K_\epsilon$ for all
$m\ge 1$ and $\sum_{i\in W} x_i(t_m)\to 0$ as $m\to \infty$.
Note that $K=\pi_{\complement W}(K)\times \{0\}^W$ and
therefore we must have $c(t_m)\in \partial(\pi_{\complement W}(K))\times
[0,\epsilon]^W$ for $m$ large enough. 
If ${\mathcal X}=\{c(t_1),c(t_2),\ldots\}$ then, since the closure $\overline{\mathcal X}$ of
${\mathcal X}$ in $\mathbb R^n$ is compact, there must exist ${\bf y}\in
\overline{\mathcal X}\subset \partial(\pi_{\complement W}(K))\times
[0,\epsilon]^W$ such that $\sum_{i\in W} y_i= 0$, which implies ${\bf y}\in
\partial(\pi_{\complement W}(K))\times \{0\}^W = \partial K$ (the
relative boundary of $K$). But ${\mathcal X}\cap F_W=\O$ since $T(c_0)$ does not
intersect the boundary of $\mathbb R^n,$ and it follows that ${\bf y}$ is
an accumulation point of ${\mathcal X}$, and therefore ${\bf y}\in
\lim_{\omega}T(c_0).$ But this is a contradiction since
$\lim_{\omega}T(c_0)\cap \partial K=\O.$
\end{proof}

A proof of the Global Attractor Conjecture may now be obtained for systems with
three-dimensional stoichiometric subspace. 

\begin{thm}\label{thm:gac} 
Consider a complex-balanced
system with three-dimensional stoichiometric subspace. 
Then, the unique positive equilibrium contained in a stoichiometric 
compatibility class is a global attractor 
of the relative interior of that stoichiometric compatibility class. 
\end{thm}
\begin{proof}
It is known that any trajectory of a complex-balanced system is
bounded \cite{feinberg:lectures}. Since a stoichiometric compatibility
class has only faces of
codimension two, codimension one and vertices, Theorems
\ref{thm:anderson} and \ref{thm:repel}, applied in this order, show that
if a trajectory $T(c_0)$ has $\omega$-limit points, then a
vertex of $S(c_0)$ is an $\omega$-limit point. But this is known to be false
\cite{anderson:primul, craciun_dickenstein_shiu_sturmfels}. 
\end{proof} 

\end{section}

\begin{section}*{Acknowledgements}
The author thanks his adviser Gheorghe Craciun for guidance, support, and for many illuminating discussions. I also thank Anne Shiu and the anonymous reviewers for many useful comments, which greatly improved the paper. This work was completed during the author's stay as a research associate at the Department of Mathematics and Department of Biomolecular Chemistry, University of Wisconsin-Madison. 
\end{section}


\begin{thebibliography}{99}

\bibitem{anderson:primul}
D.F. Anderson, 
{\em Global asymptotic stability for a class of nonlinear chemical equations,} 
{SIAM J. Appl. Math}, 68:5, 1464--1476, 2008. 

\bibitem{anderson:oneLC}
D.F. Anderson, 
{\em A proof of the Global Attractor Conjecture in the single linkage class case,}
submitted, {\tt arXiv:1101.0761v3}

\bibitem{anderson:oneLCbd}
D.F. Anderson, 
{\em Boundedness of trajectories for weakly reversible, single linkage class reaction systems,}
submitted, {\tt arXiv:1101.0761v3}


\bibitem{anderson_shiu}
D.F. Anderson, A. Shiu,
{\em The dynamics of weakly reversible population processes near
  facets,} SIAM J. Appl. Math 70 (2010) 1840--1858.

\bibitem{angeli_leenheer_sontag}
D.Angeli, P. De Leenheer, and E. Sontag, 
{\em A Petri net approach to persistence analysis in 
chemical reaction networks} in I. Queinnec, S. Tarbouriech, G. Garcia, and S-I. Niculescu, 
editors, Biology and Control Theory: Current Challenges (Lecture Notes in Control and 
Information Sciences Volume 357), 181–216. Springer-Verlag, Berlin, 2007. 

\bibitem{banaji_craciun:1}
M. Banaji and G. Craciun, 
{\em Graph-theoretic approaches to injectivity and multiple equilibria
  in systems of interacting elements,} Comm. Math. Sci. 7(4) (2009) 867--900.

\bibitem{banaji_craciun:2}
M. Banaji and G. Craciun, 
{\em Graph-theoretic criteria for injectivity and unique equilibria in
  general chemical reaction systems,} Adv. Appl. Math. 44 (2010) 168--184.

\bibitem{blanchini}
F. Blanchini, 
{\em Set invariance in control,}
Automatica 35, 1747--1767, 1999.

\bibitem{craciun_dickenstein_shiu_sturmfels}
G. Craciun, A. Dickenstein, A. Shiu, B. Sturmfels, 
{\em Toric Dynamical Systems, }
{Journal of Symbolic Computation}, 44:11, 1551--1565,  2009.

\bibitem{craciun_nazarov_pantea}
G. Craciun, F. Nazarov and C. Pantea,
{\em Persistence and permanence of mass-action and power-law dynamical systems},
{\tt arXiv:1010.3050v1}, submitted.

\bibitem{feinberg:1972}
M. Feinberg, 
{\em Complex balancing in general kinetic systems,} 
Arch. Rat. Mech. Anal. 49 (1972), 187--194. 

\bibitem{feinberg:lectures}
M.~Feinberg,
{\em Lectures on Chemical Reaction Networks},
written version of lectures given at the Mathematical Research 
Center, University of Wisconsin, Madison, WI, 1979. Available online from 
www.chbmeng.ohio-state.edu/$_{^{\sim}}\!$feinberg/LecturesOnReactionNetworks.

\bibitem {feinberg:chem_eng_sci}
M.~Feinberg,
{\em Chemical reaction network structure and the stability of complex 
isothermal reactors - I. the deficiency zero and deficiency one
theorems, review article 25}, Chem. Eng. Sci. 42 1987, 2229--2268. 

\bibitem{feinberg_horn}
M. Feinberg and F. J. M. Horn, 
{\em Dynamics of open chemical systems and the 
algebraic structure of the underlying reaction network}, 
Chem. Eng. Sci. 29 1974, 775--787. 

\bibitem{gopalkrishnan}
M. Gopalkrishnan, 
{\em Catalysis in reaction networks}, submitted. Available 
at {\tt http://arxiv4.library.cornell.edu/abs/1006.3627}. 

\bibitem{gnacadja}
G. Gnacadja, 
{\em Univalent positive polynomial maps and the equilibrium 
state of chemical networks of reversible binding reactions,} 
Adv. Appl. Math. 43 (2009), 394--414.

\bibitem{gunawardena}
J. Gunawardena, 
{\em Chemical reaction network theory for in-silico biologists} 
Available for download at {\tt http://vcp.med.harvard.edu/papers/crnt.pdf}, 2003. 

\bibitem{horn}
F.J.M. Horn, 
{\em Necessary and sufficient conditions for complex balancing in 
chemical kinetics,} Arch. Rat. Mech. Anal. 49 (1972), no. 3, 172--186.

\bibitem{horn:1974}
F.J.M. Horn, 
{\em The dynamics of open reaction systems,} SIAM-AMS Proceedings 
VIII (1974), 125--137. 

\bibitem{horn_jackson}
F.J.M. Horn, R. Jackson,
{\em General mass action kinetics, }
Archive for Rational Mechanics and Analysis, 47,  81--116, 1972.

\bibitem{pachter_sturmfels}
L. Pachter and B. Sturmfels, 
{\em Algebraic Statistics for Computational Biology,} 
Cambridge University Press, Cambridge, 2005. 
 
\bibitem{pantea:thesis}
C. Pantea,
{\em Mathematical and computational analysis of biochemical reaction
  networks},
Ph.D. Thesis, University of Wisconsin-Madison, 2010.

\bibitem{perez}
M. P\'erez Mill\'an, A. Dickenstein, A. Shiu, C. Conradi,
{\em Chemical reaction systems with toric steady states,} 
submitted, available at {\tt arXiv:1102.1590v1}. 

\bibitem{siegel_maclean}
D. Siegel and D. MacLean, 
{\em Global stability of complex balanced mechanisms}, 
J. Math. Chem. 27 (2004), no 1--2, 89--110. 

\bibitem{siegel_johnston}
D.Siegel and M.D.  Johnston,
{\em A stratum approach to global stability of complex balanced
  systems}, submitted, {\tt arXiv:1008.1622v2}.

\bibitem{sontag:tcell}
E.D. Sontag, 
{\em Structure and stability of certain chemical networks and 
applications to the kinetic proofreading of t-cell receptor signal transduc- 
tion}, IEEE Trans. Auto. Cont. 46 (2001), no. 7, 1028–1047. 

\bibitem{shiu_sturmfels}
A. Shiu, B. Sturmfels,
{\em Siphons in chemical reaction networks}, 
Bulletin of Mathematical Biology, 72:6, 1448--1463 (2010)

\bibitem{sturmfels:pol_eq}
B. Sturmfels,
{\em Solving Systems of Polynomial Equations}, 
CBMS Regional Conference Series in Mathematics 97,
AMS, 2002. 

\bibitem{takeuchi}
Y. Takeuchi,
{\em Global Dynamical Properties of Lotka-Volterra Systems}
World Scientific Publishing, 1996.

\end{thebibliography}
\end{document}